\documentclass[12pt]{amsart}
\usepackage{amsmath,amssymb}
\usepackage{eufrak}
\usepackage{amsthm}
\usepackage{enumitem}
\usepackage{setspace}
\usepackage{graphicx}
\usepackage{graphics}
\usepackage{color}
\usepackage[dvipsnames]{xcolor}
\usepackage{hyperref}
\usepackage{youngtab}
\usepackage{xcolor}

\hypersetup{ colorlinks,
 linkcolor={blue},  citecolor={BrickRed}, urlcolor={blue}}

\usepackage{amsthm,thmtools,xcolor}

\usepackage{marginnote}
\usepackage{lscape}
\usepackage{faktor}
\usepackage{hyperref}
\usepackage[normalem]{ulem}
\usepackage{bm}
\usepackage{afterpage}
\usepackage{mathrsfs}
\usepackage[all]{xy}
\xyoption{matrix}
\xyoption{arrow}
\usepackage{tikz}
\usetikzlibrary{calc,shadings,patterns}
\usetikzlibrary{matrix}
\usetikzlibrary{arrows}
\usetikzlibrary{matrix}
\usepackage{verbatim}
\usepackage{multirow}
\usetikzlibrary{decorations.pathreplacing}
\usepackage{enumitem,pifont,xcolor}
\definecolor{myblue}{RGB}{0,29,119}

\newtheorem{theorem}{Theorem}[section]
\newtheorem{proposition}[theorem]{Proposition}
\newtheorem{corollary}[theorem]{Corollary}
\newtheorem{lemma}[theorem]{Lemma}
\theoremstyle{definition}
\newtheorem{definition}[theorem]{Definition}
\newtheorem{example}[theorem]{Example}

\newtheorem{remark}[theorem]{Remark}
\newtheorem{remarks}[theorem]{Remarks}

\newtheorem{claim}[theorem]{Claim}

\newtheorem*{theorem*}{Theorem}

\makeatletter
\@addtoreset{equation}{section}
\makeatother

\newcommand{\Aa}{{\mathbb A}} 

\DeclareMathOperator{\End}{End}
\DeclareMathOperator{\coker}{coker}
\DeclareMathOperator{\im}{Im}
\DeclareMathOperator{\findim}{fin.\! dim}
\DeclareMathOperator{\gordim}{gor.\! dim}
\DeclareMathOperator{\gldim}{gl\,\! dim}
\DeclareMathOperator{\domdim}{dom.\! dim}
\DeclareMathOperator{\indim}{in.\!dim}
\DeclareMathOperator{\Hom}{Hom}
\DeclareMathOperator{\ext}{Ext}
\DeclareMathOperator{\pdim}{p.\!dim}
\DeclareMathOperator{\rank}{Rank}
\DeclareMathOperator{\add}{add}
\DeclareMathOperator{\soc}{soc}
\DeclareMathOperator{\topp}{top}
\DeclareMathOperator{\modd}{mod-\!}
\DeclareMathOperator{\rad}{rad}

\newcommand{\cB}{{\mathcal B}}

\newcommand{\cE}{{\mathcal E}}
\newcommand{\cF}{{\mathcal F}}

\newcommand{\cP}{{\mathcal P}}

\newcommand{\cS}{{\mathcal S}}
\newcommand{\cT}{{\mathcal T}}

\newcommand{\cW}{{\mathcal W}}

\renewcommand{\rank}{\operatorname{rank} }

\newcommand{\ho}[1]{\Hom_{\Lambda}(\cP,{#1})}

\newcommand{\syz}{\bm\varepsilon(\Lambda)}

\providecommand{\AMS}{$\mathcal{A}$\kern-.1667em%
\lower.25em\hbox{$\mathcal{M}$}\kern-.125em$\mathcal{S}$}

\newtheorem{innercustomthm}{Theorem}
\newenvironment{customthm}[1]
  {\renewcommand\theinnercustomthm{#1}\innercustomthm}{\endinnercustomthm}

\begin{document}

\title[Syzygy Filtrations of Cyclic Nakayama Algebras]{Syzygy Filtrations of Cyclic Nakayama Algebras}

\author{Emre SEN}

\email{\href{emresen641@gmail.com}{emresen641@gmail.com}}

\maketitle

{\let\thefootnote\relax}



\begin{center}
{\small{\it To Professors Kiyoshi Igusa and Gordana Todorov in celebration of seventieth birthdays  }}
\end{center}

\begin{abstract}
For any cyclic Nakayama algebra $\Lambda$, we construct  \emph{syzygy filtered algebra} $\bm\varepsilon(\Lambda)$ which corresponds to various syzygy modules as the name suggests. We prove that the category of modules over the syzygy filtered algebra $\bm\varepsilon(\Lambda)$ is equivalent to the wide subcategory cogenerated by projective-injective modules of the original algebra $\Lambda$ along with other categorical equivalences.

In terms of this new algebra, we interpret the following homological invariants of $\Lambda$: left and right finitistic dimension, left and right $\varphi$-dimension, Gorenstein dimension, dominant dimension and their upper bounds. For all of them, we obtain a unified upper bound $2r$ where $r$ is the number of relations defining the algebra $\Lambda$. We show that  the left finitistic and $\varphi$-dimensions are equal to the right finitistic and $\varphi$-dimensions respectively, as well as the difference between $\varphi$-dimension and the finitistic dimension is at most one. Furthermore, we recover various seemingly unrelated results in a uniform way.

\end{abstract}

\tableofcontents

\section{Introduction}
For an artin algebra $A$ over the field $\mathbb{K}$, a major open question is the finitistic dimension conjecture. It states that supremum of finite projective dimensions of $A$-modules is finite. If the representation dimension of $A$ is less than or equal to three then the conjecture holds, the proof was given by K. Igusa and G. Todorov in \cite{todo}. In the proof, they introduced a function, which we denote by $\varphi$, for each module $M$ and its value is always a finite number (see def. \ref{defvarphimodule} ). Therefore we can take supremum of those values and define $\varphi$-dimension of $A$ as
\begin{align*}
\varphi\dim A:=\sup\left\{\varphi(M)\vert\, M\,\text{is an A-module} \right\}.
\end{align*} 
It turns out that this new homological tool has applications. For instance, an algebra is selfinjective if and only if $\varphi$-dimension of algebra is zero \cite{lanz}. If an algebra is Gorenstein then Gorenstein dimension and $\varphi$-dimension are same \cite{ralf}. Possible numerical values of $\varphi$-dimension can be a useful homological measure for artin algebras.

\par An artin algebra is called Nakayama algebra if each left or right indecomposable projective module has a unique composition series. Nakayama algebras are either linear or cyclic depending on the underlying quiver. 
 In \cite{sen2018varphi} we showed that for a cyclic Nakayama algebra of infinite global dimension, $\varphi$-dimension is always an even number. In the same work, in order to study $\varphi$-dimension, we considered modules having a particular filtration which is called $\Delta$-filtration. In this work, we develop the syzygy filtration method which is based on the $\Delta$-filtrations introduced at \cite{sen2018varphi}.  In short, the core idea of the syzygy filtration method is constructing the algebra $\bm\varepsilon(\Lambda)$ called \emph{syzygy filtered algebra} whose modules are equivalent to the modules filtered by
the second syzygies of the original algebra $\Lambda$. Definition and features of the method are stated in the first section. We put together all the results regarding the syzygy filtered algebra $\bm{\varepsilon}(\Lambda)$ below where $\findim\Lambda$, $\gldim\Lambda$, $\domdim\Lambda$, $\gordim\Lambda$ stands for finitistic dimension, global dimension, dominant dimension, Gorenstein dimension of $\Lambda$ respectively.

\begin{customthm}{A}\label{bigthm1} Let $\Lambda$ be a cyclic Nakayama algebra. Then,
\begin{enumerate}[label=\roman*)]
\item\label{thmselfinjective} $\Lambda$ is selfinjective if and only if $\Lambda\cong\bm{\varepsilon}(\Lambda)$.
\item\label{thmfidimreduction} $\varphi\dim\Lambda=\varphi\dim\bm{\varepsilon}(\Lambda)+2$, provided that $\varphi\dim\Lambda\geq 2$ and $\gldim\Lambda=\infty$.
\item\label{thmfindimreduction} $\findim\Lambda=\findim\bm\varepsilon(\Lambda)+2$, provided that $\findim\Lambda\geq 2$ and $\gldim\Lambda=\infty$.
\item\label{thmgldimreduction} $\gldim\Lambda=\gldim\bm\varepsilon(\Lambda)+2$ provided that $\gldim\Lambda\geq 2$ and $\gldim\Lambda$ is finite.
\item\label{thmgordireduction} If $\Lambda$ is Gorenstein then $\bm{\varepsilon}(\Lambda)$ is Gorenstein. Moreover, $\gordim\Lambda=\gordim\bm\varepsilon(\Lambda)+2$ provided that $\gordim\Lambda\geq 2$.
\item\label{thmdomdimreduction} $\domdim\Lambda=\domdim\bm\varepsilon(\Lambda)+2$ provided that $\domdim\Lambda\geq 3$.
\item\label{thmleftrightfindim} Left and right finitistic dimensions are same i.e. $\findim\Lambda=\findim\Lambda^{op}$.
\item\label{thmleftrightfidim} Left and right $\varphi$-dimensions are same i.e. $\varphi\dim\Lambda=\varphi\dim\Lambda^{op}$.
\item\label{thmdifferencefiandfin} Difference between $\varphi$-dimension and the finitistic dimension of $\Lambda$ can be at most one i.e.  $\varphi\dim\Lambda-\findim\Lambda\leq 1$.
\end{enumerate}
\end{customthm}

By using the result \ref{thmgordireduction}, we give an elementary  proof of the equality of $\varphi$-dimension and Gorenstein dimension (\cite{ralf}, \cite{lanz}) for Nakayama algebras.

Indeed, all homological dimensions we discussed in Theorem \ref{bigthm1} share the same unified upper bound.
\begin{customthm}{B}
If $\Lambda$ is a cyclic non-selfinjective Nakayama algebra which is defined by $r$-many irredundant relations, then $\varphi\dim\Lambda$, $\findim\Lambda$, $\gordim\Lambda$, $\domdim\Lambda$ are bounded by $2r$.
\end{customthm}

The reductions \ref{thmfidimreduction}, \ref{thmfindimreduction}, \ref{thmgldimreduction}, \ref{thmgordireduction} and \ref{thmdomdimreduction} in Theorem \ref{bigthm1} enable us to use mathematical induction on the homological dimensions, because the syzygy filtered algebra $\bm\varepsilon(\Lambda)$ is also a Nakayama algebra (see \ref{PROP filtered algebra is Nakayama}). Therefore we can apply the syzygy filtration method to $\bm\varepsilon(\Lambda)$ provided that it is cyclic and get higher syzygy filtered algebras $\bm\varepsilon^k(\Lambda)$ (see \ref{defhigherfilteredalgebra}). The result concerning the higher syzygy filtered algebras is stated below.

\begin{customthm}{C}\label{bigthm2}\kern-\parskip\begin{enumerate}[label=\roman*)]
\item\label{thmlambdastable}(Thm \ref{THM fidim becomes stable }) If $\Lambda$ is a cyclic non-selfinjective Nakayama algebra of infinite global dimension, then there exists a non-negative integer $k$ such that $\bm\varepsilon^k(\Lambda)$ is not selfinjective and $\bm\varepsilon^{k+1}(\Lambda)$ is selfinjective Nakayama algebra.
\item\label{thmsplit}(Thm \ref{thm en son splitting}) If $\Lambda$ is a cyclic connected Nakayama algebra of finite global dimension, then there exists a non-negative $k$ such that $\bm\varepsilon^k(\Lambda)$  is a cyclic connected Nakayama algebra and $\bm\varepsilon^{k+1}(\Lambda)$ is not cyclic\footnote{We set $\bm\varepsilon^0(\Lambda)=\Lambda$}. The algebra $\bm\varepsilon^{k+1}(\Lambda)$ can split into components that are either linear Nakayama algebras or semisimple components or both.
\kern-\parskip\end{enumerate}\kern-\parskip
\end{customthm}
As an application of Theorem \ref{bigthm2} part \ref{thmsplit}, we give a refinement of the main result of \cite{madsen2018bounds}.

\begin{customthm}{D}(Thm \ref{RES upper bound gldim}) If global dimension of cyclic Nakayama algebra $\Lambda$ is finite, then it is bounded by $2m+r_{m-1}-C$ where $\bm\varepsilon^m(\Lambda)$ is linear Nakayama algebra and $r_{m-1}$ is the minimal number of relations which define cyclic Nakayama algebra $\bm\varepsilon^{m-1}(\Lambda)$ and $C$ is the number of connected components of $\bm\varepsilon^m(\Lambda)$.
\end{customthm}

In section \ref{section1}, we define the algebra $\bm{\varepsilon}(\Lambda)$ and study its properties and show some categorical equivalences between certain module categories (see \ref{subsec filt}). Each of the statements of Theorem \ref{bigthm1} is proved in section \ref{section2} separately. We give the proof and applications of Theorem \ref{bigthm2} in  section \ref{section3}. Also, in the subsection \ref{future} we list some results and open questions relying on the relationship between some syzygy modules and wide subcategories cogenerated by projective-injective modules (see \ref{PROP filtered module category is equivalent to wide subcategory}) which we hope to extend results of this paper into broader context.

\subsection{Acknowledgments} We are deeply thankful to Claus-Michael Ringel and Shijie Zhu for their help to the improvement of some ideas  as well as the anonymous referee for the detailed critics and valuable suggestions.

\section{Syzygy Filtered Algebras }\label{section1}
First, we want to set up the scene and the notation.\\
A module is called uniserial if its left and right composition series are unique. A bound quiver algebra is called Nakayama algebra if every indecomposable module is uniserial. There are two types of Nakayama algebras depending on the underlying quiver: either cyclic or linear.
\par Throughout the paper $\Lambda$ is a cyclic Nakayama algebra with the underlying quiver $Q$ (see the figure \ref{quiverQ}) over the field $\mathbb{K}$. $S_i$ is the simple module at vertex $i$, $P_i=P(S_i)$ is the projective cover of $S_i$, $I_i=I(S_i)$ is the injective envelope of $S_i$. We use the term \emph{rank} for the number of nonisomorphic simple modules. The socle and the top of an indecomposable $\Lambda$-module $M$ are the simple submodule of $M$ and the simple quotient of $M$ respectively. We denote them by $\soc M$ and $\topp M$. The length of an indecomposable uniserial module $M$ is the number of simple composition factors of $M$ and denoted by $\ell(M)$. The radical of an indecomposable module $M$, denoted by $\rad M$, is the longest submodule of $M$ i.e. $\faktor{M}{\rad M}\cong \topp M$ and $\ell(M)=\ell(\rad M)+1$. If $M$ is a (proper) submodule of $N$, we denote it by $M\subseteq N$ $(M\subset N)$.  We recall that a sequence of positive integers $\left(c_1,c_2,\ldots,c_N\right)$ satisfying $c_i\leq 1+c_{i+1}$ for $1\leq i\leq N-1$ and $c_N\leq c_1+1$ is called Kupisch series where each $c_i$ is the length of projective module $P_i=P(S_i)$. Any Nakayama algebra can be described by Kupisch series upto cyclic permutation.

\begin{figure}[t]\centering
\begin{tikzpicture}
\foreach \ang\lab\anch in {90/1/north, 45/2/{north east}, 0/3/east, 270/i/south, 180/{N-1}/west, 135/N/{north west}}{
  \draw[fill=black] ($(0,0)+(\ang:3)$) circle (.08);
  \node[anchor=\anch] at ($(0,0)+(\ang:2.8)$) {$\lab$};
}

\foreach \ang\lab in {90/1,45/2,180/{N-1},135/N}{
  \draw[->,shorten <=7pt, shorten >=7pt] ($(0,0)+(\ang:3)$) arc (\ang:\ang-45:3);
  \node at ($(0,0)+(\ang-22.5:3.5)$) {$\boldsymbol\alpha_{\lab}$};
}

\draw[->,shorten <=7pt] ($(0,0)+(0:3)$) arc (360:325:3);
\draw[->,shorten >=7pt] ($(0,0)+(305:3)$) arc (305:270:3);
\draw[->,shorten <=7pt] ($(0,0)+(270:3)$) arc (270:235:3);
\draw[->,shorten >=7pt] ($(0,0)+(215:3)$) arc (215:180:3);
\node at ($(0,0)+(0-20:3.5)$) {$\boldsymbol\alpha_3$};
\node at ($(0,0)+(315-25:3.5)$) {$\boldsymbol\alpha_{i-1}$};
\node at ($(0,0)+(270-20:3.5)$) {$\boldsymbol\alpha_i$};
\node at ($(0,0)+(225-25:3.5)$) {$\boldsymbol\alpha_{N-2}$};

\foreach \ang in {310,315,320,220,225,230}{
 \draw[fill=black] ($(0,0)+(\ang:3)$) circle (.02);
}
\end{tikzpicture}
\caption{Quiver Q}\label{quiverQ}
\end{figure}
\begin{definition}\label{defconsecutivemodules}
Two simple modules $S_i$, $S_j$ are called \emph{consecutive} if $j=i+1$ for $1\leq i\leq N-1$ or $S_i=S_N$ and $S_j=1$ when the rank of $\Lambda$ is $N$. Equivalently, $S_i,S_j$ are called consecutive if one of them is Auslander-Reiten translate of the other, i.e. $\tau S_i\cong S_j$.
\end{definition}

\begin{remark} Let $\rank\Lambda=N$. We order simple modules such that $\ext^1_{\Lambda}(S_i,S_j)$ is nontrivial if and only if $\tau S_i\cong S_j$. In other words, the set of the complete list of representatives of $\Lambda$-modules of length two is 
\begin{align*}
\left\{\begin{vmatrix}
S_1\\S_2
\end{vmatrix},\begin{vmatrix}
S_2\\S_3
\end{vmatrix},\ldots,\begin{vmatrix}
S_{N-1}\\S_N
\end{vmatrix},\begin{vmatrix}
S_{N}\\S_1
\end{vmatrix}\right\}.
\end{align*}

\end{remark}

Here we collect some results about uniserial modules.

\begin{lemma}\label{uniserialitylemma}\cite[lemma 2.5.7]{sen2018varphi} If two indecomposable $\Lambda$-modules $M,N$ have the isomorphic socle, then either $M$ is a submodule of $N$ or $N$ is a submodule of $M$. 
\end{lemma}
\begin{corollary}\label{uniserialitylemma injectiveCOR} If two indecomposable non-isomorphic $\Lambda$-modules $M,N$ have the isomorphic top, then either $M$ is quotient of $N$ or $N$ is quotient of $M$.
\end{corollary}

We refer to lemma \ref{uniserialitylemma} as \emph{the uniseriality lemma}.

\begin{lemma}\label{topsocConsecutive} \cite[lemma 2.5.11]{sen2018varphi} If $A$, $B$, $C$ are indecomposable uniserial modules and 
\begin{align*}
0\rightarrow A\rightarrow B\rightarrow C\rightarrow  0
\end{align*}
is nonsplit exact sequence, then
the top of $A$ and the socle of $C$ are consecutive modules, i.e. $\tau\soc C\cong\topp A$.
\end{lemma}

\begin{lemma}\label{LEM projectivecannotbequotientlemma} \cite[lemma 2.5.9.]{sen2018varphi}  A projective module $P$ cannot be a proper subquotient of another
indecomposable module $X$. Moreover, a projective module cannot be a submodule of a non-projective indecomposable module. An indecomposable projective-injective module cannot be proper submodule of another projective-injective.
\end{lemma}

\subsection{Nakayama algebras via relations}
Before introducing the system of relations independently, first we want to discuss its relationship with Kupisch series which is more common tool for the community. The direct approach was given in \cite{sen2018varphi}, but it brings its own burden, which is the cyclic ordering of particular simple and projective modules. To keep all the relevant data, intricate usage of indices is necessary and it is not easy to digest at first.  To avoid these complications, it is convenient for us to interpret Kupisch series in terms of the structures of the modules. To perform this, the first step we take is the notion of minimal projective module.

\begin{definition}\label{defminimalprojective gercek tanim}
 An indecomposable projective $\Lambda$-module is called \emph{minimal} if its radical is not projective. Dually, an indecomposable injective $\Lambda$-module is called minimal if its quotient is not injective.
\end{definition}

\begin{proposition}\label{PROP P is minimal iff ci<ci+1} Let $\Lambda$ be given  by Kupisch series $(c_1,\ldots,c_N)$. The projective module $P_i$ is minimal if and only if $c_i\leq c_{i+1}$ for $1\leq i\leq N-1$ and $c_N\leq c_1$ for $i=N$.
\end{proposition}
\begin{proof}
We work with the contrapositive of the statement.  \\ $(\Rightarrow)$. Assume that $P_i$ is not minimal projective, hence $\rad P_i\cong P_{i+1}$ is a projective module. If $\ell(P_i)=c_i$, then
\begin{align*}
c_{i+1}=\ell(P_{i+1})=\ell(\rad P_i)=c_i-1
\end{align*}
implies that $c_i>c_{i+1}$.\\
$(\Leftarrow)$. Assume that $c_i>c_{i+1}$. By the condition $c_i\leq c_{i+1}+1$, we get $c_i=c_{i+1}+1$. Simple modules appearing in the composition series of projective modules are consecutive, therefore we can express the socles of $P_i$ and $P_{i+1}$ by their top modules, i.e.
\begin{align*}
\begin{gathered}
\soc P_i\cong \tau^{c_i-1}\topp P_i\cong \tau^{c_i-1}S_i.\\
\soc P_{i+1}\cong \tau^{c_{i+1}-1}\topp P_{i+1}\cong\tau^{c_{i+1}-1}S_{i+1}.
\end{gathered}
\end{align*}
Because $c_i=c_{i+1}+1$, we get
\begin{align*}
\tau^{c_i-1}S_i\cong\tau^{c_i-2}S_{i+1}\cong\tau^{c_{i+1}-1}S_{i+1}.
\end{align*}
This means $P_i$ and $P_{i+1}$ have isomorphic socle. By lemma \ref{uniserialitylemma}, $P_{i+1}$ is submodule of $P_i$ and $P_{i}$ is not minimal.
We skip the case $i=N$ which can be proven similary. 
\end{proof}
 \begin{corollary}\label{COR P_i is minimal iff P i+1 is projective injective} $P_i$ is minimal projective if and only if $P_{i+1}$ is projective-injective.
 \end{corollary}
 \begin{proof}
 $(\Rightarrow)$ Let $P_i$ be minimal projective module with $\topp P_i\cong S_i$. Hence $\rad P_i$ is not projective module and $\topp\rad P_i\cong \tau \topp P_i=\tau S_i\cong S_{i+1}$. So $\rad P_i$ is quotient of $P_{i+1}$, i.e.
\begin{align*}
0\rightarrow \ker f\rightarrow P_{i+1}\rightarrow \rad P_i\rightarrow 0
\end{align*}
is an exact sequence  where $f: P_{i+1}\rightarrow \rad P_i$ is surjective.
  We need to show that $P_{i+1}$ is injective. Suppose not. So, there exists an injective envelope $I$ of $P_{i+1}$ where the first cosyzygy is nontrivial, i.e.
 \begin{align}\label{eq1}
 0\rightarrow P_{i+1}\rightarrow I\rightarrow \Sigma(P_{i+1})\rightarrow 0.
 \end{align}
$\ker f$ is submodule of $P_{i+1}$ and $P_{i+1}$ is submodule of $I$, we have an embedding of $\ker f$ into $I$ which induces the exact sequence
\begin{align}\label{eq2}
0\rightarrow \ker f\rightarrow I\rightarrow\faktor{I}{\ker f}\rightarrow 0.
\end{align}
Now we will compare the lengths of certain modules. We have:
\begin{enumerate}[label=\arabic*)]
\item $\ell(I)=\ell(P_{i+1})+\ell(\Sigma(P_{i+1}))$ by \ref{eq1}.
\item $\ell(I)=\ell(\ker f)+\ell(\faktor{I}{\ker f})$ by \ref{eq2}.
\end{enumerate}
 Since $P_i$ is minimal, by proposition \ref{PROP P is minimal iff ci<ci+1}, $c_i=\ell(P_i)\leq c_{i+1}=\ell(P_{i+1})$. Hence $\ell(I)>\ell(P_i)$.\\
 On the other hand $\soc\rad P_i\cong \soc \faktor{I}{\ker f}$ by applying lemma \ref{topsocConsecutive} to the exact sequences \ref{eq1} and \ref{eq2}. The sequence of the maps $\ker f\hookrightarrow P_{i+1}\hookrightarrow I$ implies that 
 \begin{align*}
 \begin{gathered}
 \ell(\rad P_i)=\ell\left(\faktor{P_{i+1}}{\ker f}\right)<\ell\left(\faktor{I}{\ker f}\right)\Rightarrow\\
 \ell(P_i)=\ell(\rad P_i)+1\leq \ell\left(\faktor{I}{\ker f}\right).
 \end{gathered}
 \end{align*}
 This shows that $P_i$ is a subquotient of $I$ which is not possible by lemma \ref{LEM projectivecannotbequotientlemma}. Therefore the injective envelope of $P_{i+1}$ is itself, $P_{i+1}$ is projective-injective module.\\
 $(\Leftarrow)$ Assume that $P_{i+1}$ is projective-injective. $P_{i+1}$ and $P_i=P(S_i)$ cannot have the isomorphic socle. Since $\rad P_i$ is the quotient of $P_{i+1}$, we have $\ell(\rad P_i)<\ell(P_{i+1})$. This implies
 \begin{align*}
 c_i=\ell(P_i)=\ell(\rad P_i)+1\leq \ell(P_{i+1})=c_{i+1}.
 \end{align*}
 By proposition \ref{PROP P is minimal iff ci<ci+1}, $P_i$ is minimal projective.
 \end{proof}
 
\begin{proposition}\label{PROP ci=ci+1 implies defect zero, minimal projective} If $c_i=c_{i+1}$, then we have
\begin{enumerate}[label=\roman*)]
\item\label{item sth2} $\tau\soc P_i\cong \soc P_{i+1}$.
\item\label{item sth1} $P_i$ is a minimal projective.
\item $P_{i+1}$ is projective-injective module which has no proper injective quotient.
\end{enumerate}
\end{proposition}

\begin{proof}
\begin{enumerate}[label=\roman*)]
\item Notice that $\topp\rad P_i\cong \topp P_{i+1}$. By the corollary to lemma \ref{uniserialitylemma}, either $\rad P_i$ is the quotient of $P_{i+1}$ or vice versa. Because $\ell (\rad P_i)=\ell(P_i)-1=c_i-1<c_{i}=c_{i+1}=\ell(P_{i+1})$, the latter is not possible. We get the surjective map $f: P_{i+1}\rightarrow \rad P_i$. The short exact sequence
 \begin{align}\label{eq3}
 0\rightarrow \ker f\rightarrow P_{i+1}\rightarrow \rad P_{i}\rightarrow 0
 \end{align}
 implies that $\ker f$ is simple module, because $\ell(\ker f)=\ell(P_{i+1})-\ell(\rad P_i)=c_{i+1}-c_i+1=1$. We conclude that $\ker f\cong \soc P_{i+1}$ is consecutive to $\soc \rad P_i\cong \soc P_i$. By the definition of consecutive modules \ref{defconsecutivemodules}, $\tau\soc P_i\cong \soc P_{i+1}$.
 \item If $P_i$ was not minimal projective, then $\rad P_i$ would be projective module. However the sequence \ref{eq3} shows that $\rad P_i$ is the proper quotient of a projective module. Projective module cannot be a quotient of another module by \ref{LEM projectivecannotbequotientlemma}. So, $P_i$  is minimal projective.
 \item By \ref{item sth1}, $P_i$ is minimal projective. By the corollary \ref{COR P_i is minimal iff P i+1 is projective injective}, $P_{i+1}$ is projective-injective.
 It is enough to show that $I(\soc P_i)\ncong P_{i+1}$ because if there exist another injective quotient $I'$ of $P_{i+1}$, then there has to be a sequence of surjective maps $P_{i+1}\rightarrow I(\soc P_i)\rightarrow I'$. If we assume $I(\soc P_i)\cong P_{i+1}$, then we get $\soc P_i\cong \soc I(\soc P_i)\cong \soc P_{i+1}$. However by \ref{item sth2}, we know that $\soc P_i\cong \tau^{-1}\soc P_{i+1}$, therefore $I(\soc P_i)\ncong P_{i+1}$ and $P_{i+1}$ has no proper injective quotients. 
 \end{enumerate}
\end{proof}

\begin{proposition}\label{PROP ci<ci+1 implies nonzero defect} If $c_i<c_{i+1}$, we have
\begin{enumerate}[label=\roman*)]
\item $\tau^{c_{i+1}-c_i+1}\soc P_i\cong \soc P_{i+1}$
\item $P_i$ is a minimal projective
\item $P_{i+1}$ is a projective-injective module and there are $c_{i+1}-c_i$ non-isomorphic proper injective quotients of $P_{i+1}$.
\end{enumerate}
\end{proposition}

\begin{proof} We set 
\begin{align}\label{eq expressing d}
d=c_{i+1}-c_i+1.
\end{align}
\begin{enumerate}[label=\roman*)]
\item Notice that $\topp\rad P_i\cong \topp P_{i+1}$. By the corollary \ref{uniserialitylemma injectiveCOR}, either $\rad P_i$ is the quotient of $P_{i+1}$ or vice versa. Because $\ell (\rad P_i)=\ell(P_i)-1=c_i-1<c_{i}<c_{i+1}=\ell(P_{i+1})$, the latter is not possible. We get the onto map $f: P_{i+1}\rightarrow \rad P_i$. The short exact sequence
 \begin{align}\label{eq6}
 0\rightarrow \ker f\rightarrow P_{i+1}\rightarrow \rad P_{i}\rightarrow 0
 \end{align}
and lemma \ref{topsocConsecutive} imply that $\topp\ker f\cong \tau\soc\rad P_{i}\cong \tau \soc P_i$ and $\ell(\ker f)=c_{i+1}-c_i+1=d$. If we apply Auslander-Reiten translate $\tau^{d-1}$ to the isomorphism $\tau\soc P_{i}\cong\topp\ker f$, then we get $\tau^d\soc P_i\cong \tau^{d-1}\topp\ker f$. Since the length of $\ker f$ is $d$, we obtain $\tau^{d-1}\topp\ker f\cong \soc \ker f\cong \soc P_{i+1}$ . Hence $\tau^d\soc P_i\cong \soc P_{i+1}$.
 \item If $P_i$ was not minimal projective, then $\rad P_i$ would be projective module. However the sequence \ref{eq6} shows that $\rad P_i$ is a proper quotient of a projective module. Projective modules cannot be quotient of another module. So, $P_i$  is minimal projective.
 \item It is enough to prove that the injective envelope $I$ of $\tau\soc P_i$ is the quotient of $P_{i+1}$. First of all, this proves that $P_{i+1}$ is an injective module because non-projective but injective modules has to be quotients of projective-injectives by the dual of lemma \ref{LEM projectivecannotbequotientlemma}. Secondly, when $I$ is an injective quotient, then the sequence of surjective maps
 \begin{align}\label{eq inj list}
 P_{i+1} \rightarrow I\left(\faktor{P_{i+1}}{\soc P_{i+1}}\right)\rightarrow \cdots\rightarrow I=I(\tau\soc P_i)
 \end{align}
 shows that each module are injective because $I$ is injective quotient of all of them.

 Notice that $\soc\left(\faktor{I}{\soc I}\right)\cong\soc\left(\faktor{I}{\tau S}\right)\cong S=\soc P_i$. Relying on the uniseriality lemma \ref{uniserialitylemma}, we get three cases:
 \begin{enumerate}[label=Case \roman*)]
 \item $P_i\cong \faktor{I}{\tau S}$  is not possible because a projective module cannot be a quotient of an injective by \ref{LEM projectivecannotbequotientlemma}.
 \item $P_i\subset \faktor{I}{S}$ is impossible because a projective module cannot be a subquotient of an injective by \ref{LEM projectivecannotbequotientlemma}.
 \item We conclude that $\faktor{I}{\tau S}$ is a proper submodule of $P_i$. 
  \end{enumerate}
Now there are two possibilities, either $\faktor{I}{\tau S}$ is isomorphic to $\rad P_i$ or it is a proper submodule of $\rad P_i$. We will show that the latter is impossible. If we suppose $\faktor{I}{\tau S}$ is a proper submodule of $\rad P_i$, we get 
\begin{align}\label{eq7}
\ell\left(\faktor{I}{\tau S}\right)<\ell(\rad P_i)=c_i-1.
\end{align}
 On the other hand, the exact sequence \ref{eq6} gives
 \begin{align}\label{eq8}
 \ell\left(\rad P_i\right)=\ell(P_{i+1})-\ell(\ker f)
 \end{align}
 Combining \ref{eq7} and \ref{eq8} we get
 \begin{align}\label{eq9}
 \ell\left(\faktor{P}{\tau S}\right)<\ell(\rad P_i)=\ell(P_{i+1})-\ell(\ker f).
 \end{align}
 If we substitute the lengths $\ell(P_{i})=c_i$, $\ell(P_{i+1})=c_{i+1}$, $\ell(\ker f)=d$ (see \ref{eq expressing d}) to \ref{eq9}, then
 \begin{align*}
 \begin{gathered}
 c_{i}-1<c_{i+1}-d\\
 c_i-1<c_{i+1}-\left(c_{i+1}-c_i+1\right)\\
 -1<-1
 \end{gathered}
 \end{align*}
 which is not a true statement. Therefore $\faktor{I}{\tau S}$ has to be isomorphic to the radical of $P_i$, and by the sequence \ref{eq6}, $P_{i+1}$ is the projective cover of $I$, which makes $P_{i+1}$ projective-injective. Moreover, this makes all the modules we considered in \ref{eq inj list} injective. Their socles are $\tau\soc P_i\cong \soc I,\tau^2\soc P_i,\ldots,\tau^{d-1} \soc P_{i},\tau^d\soc P_i\cong\soc P_{i+1}$. There are $d-1=c_{i+1}-c_i$ proper injective quotients of $P_{i+1}$.

 \end{enumerate}
\end{proof}

\begin{proposition}\label{PROP projective module classes vs } Let $\Lambda$ be given by the Kupisch series $(c_1,\ldots,c_N)$. Projective modules $P_{i+1},P_{i+2}\ldots,P_{k},\ldots,P_{j-1}, P_j$ have isomorphic socle $S$ and $\soc P_i\ncong S$, $\soc P_{j+1}\ncong S$ if and only if $c_i\leq c_{i+1}$ and $c_j\leq c_{j+1}$ and  $c_k>c_{k+1}$ for all $i+1\geq k\geq j$.
\end{proposition}

\begin{proof}
$(\Rightarrow)$. If projective modules $P_{k}$  have isomorphic socle $S$ when $i+1\leq k\leq j$, by proposition \ref{PROP P is minimal iff ci<ci+1} we get $c_k>c_{k+1}$. Because $\soc P_{j}\ncong S$, $P_{j+1}$ is not a submodule of $P_{j}$ which makes $P_{j}$ minimal projective. Therefore, proposition \ref{PROP P is minimal iff ci<ci+1} implies $c_{j}\leq c_{j+1}$. Similarly, $\soc P_i\ncong S$ implies that $P_{i+1}$ is not submodule of $P_i$, hence $P_i$ is minimal projective and we get $c_i\leq c_{i+1}$ by the same proposition \ref{PROP P is minimal iff ci<ci+1}.
\par $(\Leftarrow)$  $c_j\leq c_{j+1}$ and $c_i\leq c_{i+1}$ imply that $P_j$ and $P_{i}$ are minimal projectives. The condition $c_k>c_{k+1}$ for $i+1\leq k\leq j$ imply none of the projective modules from $P_{i+1}$ to $P_{j-1}$ are minimal. In other words, each $P_{k}$ is isomorphic to (higher) radical of $P_{i+1}$, i.e. $P_k\cong\rad^mP_{i+1}$ for some $m$. Therefore they have isomorphic socle.
\end{proof}

\begin{definition}\label{defprojectiveclass of modules} The collection of projective modules having isomorphic socle is called \emph{class of projective modules}.
\end{definition}

To find the relations, for each class of projective modules, we need to know the minimal projective and projective-injective of the class to express all the projective modules in it. By proposition \ref{PROP projective module classes vs }, the information of projective-injective module of each class can be obtained from minimal projective of the previous class. Therefore, the important data is the relative positions and the length of minimal projectives. If we know the top and the length of a minimal projective, it produces  the relation starting at the index of the top with the same length. We order each relation according to their relative position in the Kupisch series which describes classes of projective modules.
Let $(c_1,\ldots,c_N)$ be a Kupisch series and $(c_{i_1},c_{i_2},\ldots,c_{i_r}) $ be a subsequence of it such that each $c_{i_m}$ is the length of minimal projective $P_{i_m}$. Then the relations defining the algebra are 
\begin{align}
\begin{gathered}\label{eqsKupishtorels}
\alpha_{i_1+c_{i_1}-1}\circ\alpha_{i_1+c_{i_1}-2}\circ\cdots\circ\alpha_{i_1+1}\circ\alpha_{i_1}=0\\
\alpha_{i_2+c_{i_2}-1}\circ\alpha_{i_2+c_{i_2}-2}\circ\cdots\circ\alpha_{i_2+1}\circ\alpha_{i_2}=0\\
\vdots\\
\alpha_{i_r+c_{i_r}-1}\circ\alpha_{i_r+c_{i_r}-2}\circ\cdots\circ\alpha_{i_r+1}\circ\alpha_{i_r}=0
\end{gathered}
\end{align}

The converse is also true. We start with a system of relations which gives the index of minimal projectives. And the ordering of the relations gives us the ordering of the minimal projectives in the Kupisch series.

For our purposes, the socle of a projective module carries more information than the length of a projective module. Therefore, from now on we use relations (see \ref{relations} below) to define any cyclic Nakayama algebra.

\subsection{Properties of systems of relations}
We describe cyclic Nakayama algebras in terms of the system of relations. Let $\Lambda$ be Nakayama algebra of rank $N$ with $N\geq 2$ given by $r\geq 1$ many relations $\boldsymbol\alpha_{k_{2i}}\ldots\boldsymbol\alpha_{k_{2i-1}}=0$ where $1\leq i\leq  r$ and $k_{f}\in\left\{1,2,\ldots,N\right\}$ for quiver $Q$ as in the figure \ref{quiverQ}. Notice that each arrow $\boldsymbol\alpha_i$ starts at the vertex $i$ and ends at the vertex $i+1$ with the exception $\boldsymbol\alpha_N$ which starts at vertex $N$ and ends at the vertex $1$.\\
We assume that the algebra $\Lambda$ is given as the path algebra of the quiver $Q$ modulo the system of relations REL
\begin{gather}
\begin{gathered}\label{relations}
\boldsymbol\alpha_{k_2}\ldots\boldsymbol\alpha_{k_1+1}\boldsymbol\alpha_{k_1}\ \ =0 \\
\boldsymbol\alpha_{k_4}\ldots\boldsymbol\alpha_{k_3+1}\boldsymbol\alpha_{k_3}\ \ =0 \\
\dots \\
\boldsymbol\alpha_{k_{2r-2}}\ldots\boldsymbol\alpha_{k_{2r-3}+1}\boldsymbol\alpha_{k_{2r-3}}=0\\
\boldsymbol\alpha_{k_{2r}}\ldots\boldsymbol\alpha_{k_{2r-1}+1}\boldsymbol\alpha_{k_{2r-1}}=0.
\end{gathered}
\end{gather}
It is clear that $\Lambda$ is bound quiver algebra $\faktor{\mathbb{K}Q}{\left<REL\right>}$ where $\mathbb{K}$ is algebraically closed field.
First of all, we assume that this system of relations is irredundant i.e. none of the relations is a consequence of the other relations. 
Secondly, we order them according to their starting index $1\leq k_1 < k_3<\ldots<k_{2r-1}\leq N$, which is equivalent to cyclic permutation of the corresponding Kupisch series as in \ref{eqsKupishtorels}.


\begin{remark} The irredundant system of equations REL \ref{relations} satisfies:
\begin{enumerate}[label=\arabic*)]
\item There exists at most one relation starting at each $k_j\in[1,N]$.
\item There exists at most one relation ending at each $k_j\in[1,N]$.
\item The number of relations $r$ is less than or equal to the rank $N$.
\item There is no restriction on the lengths of relations except that the system of relations has to be irredundant. 
\item Each relation is composition of at least two arrows, because $\Lambda$ is a cyclic Nakayama algebra and there is no simple projective module.
\item 
As we stated in \ref{defprojectiveclass of modules}, projective modules can be described with respect to the socles. By using the relations \ref{relations}, classes of projective modules are; 
\begin{align}\label{classesofprojectives}
\begin{aligned}
P_{k_{1}}\hookrightarrow\ldots\hookrightarrow  P_{(k_{2r-1})+1} \quad\text{ have simple } S_{k_{2}} \text{  as their socle}\\
P_{k_3}\hookrightarrow\ldots\hookrightarrow P_{k_1+1}  \quad\text{ have simple } S_{k_4} \text{  as their socle}\\
\vdots\qquad\qquad\qquad\qquad\qquad\qquad\qquad\\
P_{k_{2r-1}}\hookrightarrow\ldots\hookrightarrow P_{(k_{2r-3})+1}\quad  \text{ have simple } S_{k_{2r}} \text{  as their socle}
\end{aligned}
\end{align}
where $P_{k_1}, P_{k_3},\ldots, P_{k_{2r-1}}$ are minimal projectives and $P_{k_{1}+1},P_{k_{3}+1}\ldots,P_{k_{2r-1}+1}$ are projective-injectives.  We recall the relationship between Kupisch series and the REL stated in \ref{relations}. 
\begin{enumerate}
\item $c_i>c_{i+1}$ $\iff$ $P(S_{i})$ is not minimal projective, i.e. $\rad P_i\cong P_{i+1}$.
\item $c_i=c_{i+1}$ $\iff$ $P(S_i)$ is minimal projective and $P(S_{i+1})$ is projective-injective and has no proper injective quotients.
\item $c_i<c_{i+1}$ $\iff$ $P(S_i)$ is minimal projective and $P(S_{i+1})$ has proper injective quotients.
\end{enumerate}

\end{enumerate}
\end{remark}

\subsection{The socle set and the base set}\label{subsec soclebase}

Let $\cS(\Lambda)$ be the complete set of representatives of socles of projective modules over $\Lambda$. By using the system of relations \ref{relations}, it is 
\begin{align}\label{defsocleset}
\cS(\Lambda)=\left\{S_{k_2}, S_{k_4},\ldots,S_{k_{2r}}\right\}.
\end{align}
$\cS(\Lambda)$ is called the \emph{socle set}.
We define the set $\cS'(\Lambda)$ which is the complete set of representatives of Auslander-Reiten translates of socles of indecomposable projective modules. Hence $S_i\in\cS(\Lambda)$ if and only if $\tau S_i\in \cS'(\Lambda)$. Because $\tau S_i\cong S_{i+1}$, we get 
\begin{align}\label{deftopset}
\cS'(\Lambda)=\left\{S_{k_{2}+1}, S_{k_4+1},\ldots,S_{k_{2r}+1}\right\}.
\end{align}
$\cS'(\Lambda)$ is called the \emph{top set}.

\begin{definition}\label{defshortestset} An indecomposable $\Lambda$-module $M$ satisfying $\topp M\in\cS'(\Lambda)$ and $\soc M\in\cS(\Lambda)$  is called \emph{shortest} if the composition factors of $M$ except $\topp M$ and $\soc M$ are not elements of the socle and top sets.
\end{definition}

\begin{definition}\label{defbaseset} Let $\Lambda$ be a cyclic Nakayama algebra defined by the system of $r$-many relations. For each $j\in\left\{1,2,\ldots,r\right\}$, let $\Delta_j$ be the shortest indecomposable uniserial module with $\soc\Delta_j\cong S_{k_{2j}}$ and $\topp \Delta_j\cong S_{k_{2(j-1)}+1}$. The complete set of representatives of modules $\Delta_j$'s is called \emph{the base set} and denoted by $\cB(\Lambda)$. Explicitly we have

\begin{align*}
\cB(\Lambda):=\left\{ \Delta_1\cong\begin{vmatrix}
    S_{k_{2r}+1} \\
    \vdots  \\
    S_{k_{2}}
\end{vmatrix}\!, \Delta_2\cong\begin{vmatrix}
    S_{k_{2}+1}  \\
    \vdots  \\
   S_{k_{4}}
\end{vmatrix}\!,..,\Delta_j\cong\begin{vmatrix}
   S_{k_{2(j-1)}+1}  \\
    \vdots  \\
    S_{k_{2j}}
\end{vmatrix}\!,..,\Delta_r\cong \begin{vmatrix}
   S_{k_{2r-2}+1}  \\
    \vdots  \\
    S_{k_{2r}}
\end{vmatrix}\!\right\}.
\end{align*}
\end{definition}

\begin{proposition}\label{PROP properties of the basesetproperties} Regarding the base set $\cB(\Lambda)$, we have:
\begin{enumerate}[label=\arabic*)]
\item\label{item1} The socle of each $\Delta_i$ is an element of the socle set $\cS(\Lambda)$, i.e. $\soc\Delta_i\in\cS(\Lambda)$. Any element of the socle set is a socle of an element of $\cB(\Lambda)$.
\item The top of each $\Delta_i$ is an element of the top set $\cS'(\Lambda)$. Any element of the top set is a top of an element of $\cB(\Lambda)$.
\item\label{item3} Any simple $\Lambda$-module $S$ appears in the composition series of exactly one $\Delta_i$. Equivalently, the simple composition factors of distinct $\Delta_j$'s are disjoint.
\item\label{itemitem 4} Distinct elements of the base set are Hom-orthogonal i.e. $\Hom_{\Lambda}\left(\Delta_i,\Delta_j\right)\cong 0$ when $i\neq j$ and $\Hom_{\Lambda}(\Delta_i,\Delta_i)\cong\mathbb{K}$.
\item\label{itemitem 5} Each $\Delta_i$ is a submodule of an indecomposable projective-injective module.
\item\label{item last in prop 2.18} $\Delta_i$ is simple $\Lambda$-module if and only if $S\cong\Delta_i$ satisfies $S\in\cS'(\Lambda)\cap\cS(\Lambda)$.
\end{enumerate}
\end{proposition}
 \begin{proof}
 \begin{enumerate}[label=\arabic*)]
 \item By the definition of the base set \ref{defbaseset}, $S\in \cS(\Lambda)$ if and only if $S\subset
 \Delta_i$ for some $i$.
 \item By the definitions \ref{defsocleset} and \ref{deftopset}, $S\in\cS'(\Lambda)$ if and only if $S\cong \faktor{\Delta_i}{\rad\Delta_i}$ for some $i$.
 \item Assume that $S$ appears in the simple composition factors of both $\Delta_i$ and $\Delta_j$. Therefore for some positive integers $a,b$ we have
 \begin{align*}
 S\cong\tau^a\topp\Delta_i\cong\tau^b\topp\Delta_j.
 \end{align*}
 Without loss of generality,  we assume that $a\geq b$. This implies
 \begin{align}\label{eqa1}
 \tau^{a-b}\topp\Delta_i\cong\topp\Delta_j.
 \end{align}
 $\Delta_i$ is uniserial, so for some number $c$,
 $\tau^c\topp\Delta_i\cong\soc\Delta_i$. Because $S$ is a simple module of the composition factors of $\Delta_i$, we get $c\geq a$. If we combine this observation with \ref{eqa1}, we conclude that $\tau^{a-b}\topp\Delta_i\cong\topp\Delta_j$ is a simple module in the composition factors of $\Delta_i$ which implies that $\Delta_i$ is not shortest module according to the definition \ref{defshortestset}. Therefore any simple module $S$ appears in the composition series of at most one $\Delta_i$.
\par $\Lambda$ is cyclic, if $S\in\cS(\Lambda)$, then $\tau^NS\cong S$. By the result \ref{PROP properties of the basesetproperties} \ref{item1}, let $S\cong\soc\Delta_i$ for some $i$. Therefore every simple $\Lambda$-module appears in the composition factors of at least one $\Delta_j$ for some $j$. 
 \item By the result \ref{item3}, the simple composition factors of distinct $\Delta_j$'s are
disjoint, therefore $\Hom_{\Lambda}\left(\Delta_i,\Delta_j\right)$ is trivial when $i\neq j$ and $\mathbb{K}$ when $i=j$.
\item Let $P$ be a projective-injective module. Because $\soc\Delta_i\in \cS(\Lambda)$ and $\soc P\in\cS(\Lambda)$, by the uniseriality lemma \ref{uniserialitylemma}, either $\Delta_i$ is submodule of $P$ or $P$ is submodule of $\Delta_i$. By lemma \ref{LEM projectivecannotbequotientlemma}, $\Delta_i$ is a submodule of some $P$. 
\item If $S\in\cS(\Lambda)\cap \cS'(\Lambda)$, then $S$ itself is the shortest module. Therefore $S\in\cB(\Lambda)$ by the definition \ref{defshortestset}. If $S\in\cB(\Lambda)$ and $S$ is a simple module, then $\topp S\in\cS'(\Lambda)$ and $\soc S\in\cS(\Lambda)$. Since $S$ is simple we get $S\cong\topp S\cong\soc S$ which means $S\in\cS(\Lambda)\cap\cS'(\Lambda)$.
 \end{enumerate}
 \end{proof}

\begin{remark}\cite[remark 2.2.2]{sen2018varphi} The following are equal:
\begin{enumerate}[label=\arabic*)]
\item the number of relations,
\item the number of minimal projectives,
\item the number of projective-injectives,
\item the number of minimal injectives,
\item the number of non-isomorphic socles of projective modules,
\item the number of non-isomorphic tops of injective modules,
\item the number of elements of the base set.
\end{enumerate}
\end{remark}

\subsection{Other realizations of the base set}

Here, we establish the link between the base set and particular syzygy modules.

\begin{lemma}\label{LEM base set is not projective injective} Elements of the base set cannot be projective-injective modules.
\end{lemma}
\begin{proof}
Assume that projective-injective module $P_i$ is an element of the base set. Therefore $\topp P_i\cong S_i\in\cS'(\Lambda)$ which implies that $\tau^{-1}S_{i}\cong S_{i-1}$ is an element of the socle set $\cS(\Lambda)$. Because algebra is cyclic, $\ell(P_{i-1})\geq 2$ and there is a map $f:P_{i}\rightarrow P_{i-1}$ which induces the short exact sequence
\begin{align*}
0\rightarrow \ker f\rightarrow P_i\rightarrow \rad P_{i-1}\rightarrow 0.
\end{align*}
If $\ker f$ is trivial, then $P_i\cong\rad P_{i-1}$, which makes $P_i$ proper submodule of $P_{i-1}$. This violates the projective-injectivity of $P_i$.\\
If $\ker f$ is nontrivial, then $\soc \ker f\cong\soc P_i$ and $\topp\ker f\cong\tau\soc\rad P_{i-1}\cong \tau\soc P_i$ by lemma \ref{topsocConsecutive}. Therefore the top of $\ker f$ is an element of the top set $\cS'(\Lambda)$, socle of $\ker f$ is an element of $\cS(\Lambda)$ and $\ker f\subset P_i$. Therefore $P_i$ is not the shortest module (see def. \ref{defshortestset}). This shows that $P_i$ cannot be an element of the base set $\cB(\Lambda)$.
\end{proof}

\begin{proposition}\label{PROP x in base set iff x is omegaone of radical} A module $X$ is an element of the base set if and only if $X$ is the first syzygy of the radical of a minimal projective.
\end{proposition}

\begin{proof}
$(\Rightarrow)$. Assume that $X\in \cB(\Lambda)$. By lemma \ref{LEM base set is not projective injective}, $X$ cannot be a projective-injective module. $\soc X\in\cS(\Lambda)$ implies that $X$ is a submodule of a projective-injective module $PI$ by the uniseriality lemma \ref{uniserialitylemma}. Let $Q$ be the quotient i.e. $Q=\faktor{PI}{X}$. The socle of $Q$ is $\tau^{-1}\topp X$ by lemma \ref{topsocConsecutive}, and $\topp X\in \cS'(\Lambda)$ implies that $\soc Q\in\cS(\Lambda)$. Therefore $Q$ is a submodule of a projective-injective module $PI'$. Moreover, $Q$ is a submodule of a minimal projective module $P$ where $\soc P\cong \soc PI'\cong \soc Q$. Such a $P$ exists since $Q$ is the quotient of a  projective module and quotients cannot be projective.\\
We want to prove that $Q\cong\rad P$. Assume to the contrary that let $Q\ncong \rad P$. $Q$ is submodule of $P$, therefore there exists $i$ such that $Q\cong \rad^iP$ where $i\geq 2$.

Let $P'$ be the indecomposable projective cover of $\rad^{i-1}P$. 
\begin{claim}\label{Cla 1   2.22} $\soc P'\ncong\soc P$.
\end{claim}
\begin{proof}
If $\soc P\cong \soc P'$, then $P'$ becomes a projective submodule of $P$ which violates minimality of $P$. Hence $\soc P\ncong\soc P'$.
\end{proof}

\begin{claim}\label{Cla 2   2.22} $\soc P'\ncong\soc PI$
\end{claim}
\begin{proof}
If $\soc P'\ncong \soc PI$, then $PI$ becomes a submodule of $P'$ by the uniseriality lemma. This violates projective-injectivity of $PI$, i.e. if $S$ is socle of a projective module, then indecomposable projective-injective module is the longest module with the socle $S$.
\end{proof}

\begin{claim}\label{Cla 3   2.22} $\topp \rad P'\cong\topp PI\cong\topp Q$.
\end{claim}
\begin{proof}
Since $P'=P(\rad^{i-1}P)$, then the top of $P'$ is $\tau^{i-1}\topp P$. Therefore $\topp\rad P'\cong\tau\left(\tau^{i-1}\topp P\right)\cong\tau^i\topp P$.\\
Since $Q\cong\rad^iP$, then $\topp Q\cong \tau^i\topp P$. Moreover, $Q$ is the quotient of $PI$, hence $\topp Q\cong \topp PI$.
As a result, we get the isomorphism.
\end{proof}

\begin{claim}\label{Cla  4  2.22}
$\rad P'$ is a proper quotient of $PI$.
\end{claim}
\begin{proof}
By the previous claim \ref{Cla 3   2.22}, $\topp \rad P'\cong \topp PI$. By the corollary \ref{uniserialitylemma injectiveCOR}, $\rad P'$ is the quotient of $PI$, since $PI$ is a projective-injective module. It is proper, because by the claim \ref{Cla 1   2.22}, $\soc\rad P'\cong\soc P'\ncong\soc PI$.
\end{proof}

\begin{claim}\label{Cla 5   2.22}
$Q$ is a proper quotient of $\rad P'$.
\end{claim}

\begin{proof}
By the claim \ref{Cla 3   2.22}, $\topp Q\cong\topp \rad P'$. By the corollary \ref{uniserialitylemma injectiveCOR}, we have three possibilities.
\begin{enumerate}[label=Case \arabic*)]
\item $Q\cong \rad P'$ is impossible, because it makes $Q$ a proper submodule of $P'$ and $P'$ a proper submodule of $P$, which violates the minimality of $P$.
\item Let $S\cong\topp \rad P'\cong\topp Q$. Consider the following exact sequences
\begin{align*}
\begin{gathered}
0\rightarrow Q\cong\rad^iP\rightarrow \rad^{i-1}P\rightarrow \tau^{-1}S\rightarrow 0\\
0\rightarrow \rad P'\rightarrow P'\rightarrow \tau^{-1}S\rightarrow 0.
\end{gathered}
\end{align*}
If $\rad P'$ were quotient of $Q$, then $P'$ would be quotient of $\rad^{i-1} P$. This violates projectivity of $P'$.
\item Thus, $Q$ is a proper quotient of $\rad P'$.
\end{enumerate}
\end{proof}

\begin{claim}\label{Cla 6   2.22} There exists a surjective map $f:PI\rightarrow\rad P'$ such that $\ell(\ker f)<\ell(X)$.
\end{claim}
\begin{proof}
If we combine the previous claims \ref{Cla  4  2.22} and \ref{Cla 5   2.22}, we get $PI\twoheadrightarrow \rad P'\twoheadrightarrow Q$, therefore there exist surjective map $f:PI\twoheadrightarrow \rad P'$, which induces the short exact sequence
\begin{align}\label{eq 2239}
0\rightarrow \ker f\hookrightarrow PI\twoheadrightarrow \rad P'\rightarrow 0.
\end{align}
Since $Q$ is the quotient of $PI$, we have another exact sequence
\begin{align}\label{eq239}
0\rightarrow X\rightarrow PI\rightarrow Q\rightarrow 0.
\end{align}
The lengths of the modules satisfy $\ell(PI)=\ell(X)+\ell(Q)=\ell(\ker f)+\ell(\rad P')$. Since $Q$ is a proper quotient of $\rad P'$, then $\ell(Q)<\ell(\rad P')$. Therefore, we get
\begin{align*}
\ell(X)>\ell(\ker f).
\end{align*}
Both $X$ and $\ker f$ are submodules of $PI$. Moreover, by the uniseriality lemma \ref{uniserialitylemma} together with \ref{eq239}, $\ker f$ is a proper submodule of $X$. 
\end{proof}
\begin{claim}\label{Cla 7   2.22} $soc\ker f\in\cS(\Lambda)$ and $\topp \ker f\in\cS'(\Lambda)$.
\end{claim}
\begin{proof}
Since $\ker f$ is submodule of $PI$, then $\soc \ker f\in\cS(\Lambda)$. Furthermore, the exact sequence \ref{eq 2239} shows that $\tau\soc \rad P'\cong\topp\ker f$  by lemma \ref{topsocConsecutive}. Since $\soc\rad P'\in\cS(\Lambda)$, we get $\topp\ker f\in\cS'(\Lambda)$.
\end{proof}
By the claims \ref{Cla 6   2.22} and \ref{Cla 7   2.22}, $X$ is not the shortest module, because $\ker f$ is the proper submodule of $X$ and shorter than $X$. We get the contradiction, if the quotient $Q$ is not isomorphic to $\rad P$, then $X$ cannot be an element of $\cB(\Lambda)$.

$(\Leftarrow)$ Let $Q\cong \rad P$ where $P$ is minimal projective. By the corollary \ref{COR P_i is minimal iff P i+1 is projective injective}, the projective cover of $Q$ is projective-injective module. Let's denote $P(Q)\cong P(\rad P)$ by $PI$. 

Assume to the contrary that module $X$ satisfying the exact sequence $0\rightarrow X\rightarrow PI\rightarrow Q\rightarrow 0$ is not an element of $\cB(\Lambda)$. Since $\soc X\in\cS(\Lambda)$, and $X$ is not shortest, then there exists a submodule $Y$ of $X$ such that $Y\in\cB(\Lambda)$. By lemma \ref{topsocConsecutive}, the socle of the quotient $Q'=\faktor{PI}{Y}$ is an element of $\cS(\Lambda)$. This forces that there has to be a projective module $P'$ such that $Q'$ is a proper submodule of $P'$ and its socle is not isomorphic to socles of $P$ and $PI$. We have
\begin{align*}
\soc P'\cong \tau^i\soc P,\,\, \soc PI\cong\tau^j\soc P'\cong \tau^{i+j}\soc P
\end{align*}
and in particular $\topp P'\ncong \topp P$ and $\topp P'\ncong \topp PI$. However, because of $\topp PI\cong\tau \topp P$, there is no projective module $P'$ satisfying the above conditions. Hence $X$ has to be the shortest module, hence $X\in\cB(\Lambda)$.\end{proof}

\begin{proposition}\label{PROP x is an element of the base set iff it is second syzygy}
A module $X$ is an element of the base set if and only if $X$ is the second syzygy of a simple module which is the top of a minimal projective module. 
\end{proposition}

\begin{proof}
We recall proposition \ref{PROP x in base set iff x is omegaone of radical}, i.e. $X\in\cB(\Lambda)$ if and only if $X\cong\Omega^1(\rad P)$ where $P$ is minimal projective. Therefore,  the simple module $S=\faktor{P}{\rad P}\cong \topp P$ proves the claim, i.e.  $\Omega^2(S)\cong \Omega^1(\rad P)\cong X$.
\end{proof}

\subsection{The category of filtered modules} \label{subsec filt}
We want to draw attention to $\Lambda$-modules which are filtered by the base set $\cB(\Lambda)$.
\begin{definition}\label{deffilteredmodules}
Let $Filt(\cB(\Lambda))$ denote the category of $\Lambda$-modules which are filtered by the base set $\cB(\Lambda)$, i.e. $M\in\modd\Lambda$ is an object of $Filt(\cB(\Lambda))$ if and only if $\soc M\in\cS(\Lambda)$ and $\topp M\in\cS'(\Lambda)$. All the maps in $Filt(\cB(\Lambda))$ are induced from $\Lambda$-module morphisms.
\end{definition}

\begin{proposition}{\label{PROP filtered modules have unique filtration}}
 Let $M$ be an indecomposable $\Lambda$-module. Suppose  $\topp M\in \mathcal S'(\Lambda)$ and $\soc M\in \mathcal S(\Lambda)$ up to isomorphisms. Then
 \begin{enumerate}[label=\roman*)]
 \item If $M$ is shortest among such modules, then $M\cong \Delta_i$ for some $i$.
 \item  The module $M$ has a filtration by the modules in $\cB(\Lambda)$.
\end{enumerate}
\end{proposition}
\begin{proof}
See \cite[proposition 2.5.8]{sen2018varphi}.
\end{proof}

\begin{remark}\label{REM delta lengths} Let $M$ be an indecomposable $\Lambda$-module. If $M\in Filt(\cB(\Lambda))$, then there is a composition series of modules
\begin{align*}
0=M_0\subset M_1\subset M_2\subset \cdots \subset M_n=M
\end{align*}
where each $M_i\in Filt(\cB(\Lambda))$ and $\faktor{M_i}{M_{i-1}}\in\cB(\Lambda)$  by the definition.
\end{remark}

 We will show equivalent formulations of this category.

\begin{proposition}\label{PROP pdim M>2 then all higher filtered} If $M$ is an indecomposable module with projective dimension greater or equal than two, then $\Omega^2(M)$ has $\cB(\Lambda)$-filtration.
\end{proposition}
\begin{proof}
We consider the projective resolution of $M$:
\begin{align*}
\cdots\rightarrow P_2\rightarrow P_1\rightarrow M\rightarrow 0.
\end{align*}
The socles of $P_1$ and $P_2$ are in $\cS(\Lambda)$. It is enough to show that $\topp\Omega^2(M)$ is in the top set $\cS'(\Lambda)$. If we apply lemma \ref{topsocConsecutive} to the exact sequence
\begin{align*}
0\rightarrow \Omega^2(M)\rightarrow P_2\rightarrow \Omega^1(M)\rightarrow 0
\end{align*}
we get $\topp\Omega^2(M)\cong \tau\soc\Omega^1(M) \cong\tau\soc P_1$. Since $\soc P_1\in\cS(\Lambda)$, $\topp\Omega^2(M)\in\cS'(\Lambda)$. Therefore $\Omega^2(M)$ has $\cB(\Lambda)$-filtration.
\end{proof}

\begin{proposition}\label{PROP all higher syzygies are filtered} If $M$ is a $\cB(\Lambda)$-filtered indecomposable module, then all of the nontrivial higher syzygies of $M$ and their projective covers have $\cB(\Lambda)$-filtration.
\end{proposition}
\begin{proof}
Socle of any projective module is in $\cS(\Lambda)$. It is enough to show that the top of the projective covers of the syzygies are in $\cS'(\Lambda)$ which implies all the projective covers are $\cB(\Lambda)$-filtered by the definition \ref{deffilteredmodules}.
\par Consider the projective resolution of $M$
\begin{align*}
\cdots\rightarrow P(\Omega^2(M))\rightarrow P(\Omega^1(M))\rightarrow P(M)\rightarrow M\rightarrow 0.
\end{align*}
$M\in Filt(\cB(\Lambda))$ implies that $\topp P(M)\cong\topp M\in\cS'(\Lambda)$, therefore $P(M)$ has $\cB(\Lambda)$-filtration.\\
By lemma \ref{topsocConsecutive}, $\topp\Omega^1(M)\cong \tau\soc M$ provided that $\Omega^1(M)$ is non-trivial, therefore $\topp\Omega^1(M)\in\cS'(\Lambda)$, which shows $\Omega^1(M)\in Filt(\cB(\Lambda))$. Claim holds by induction.
\end{proof}

\begin{proposition}\label{PROP X has filtration then it is second syzygy} If $X$ is an indecomposable $\cB(\Lambda)$-filtered $\Lambda$-module which is not projective-injective, then there exist a module $M$ such that $X\cong \Omega^2(M)$.
\end{proposition}
\begin{proof}
$X$ has $\cB(\Lambda)$-filtration so $S=\soc X\in\cS(\Lambda)$. Therefore $X$ is a submodule of projective-injective $P$ with socle $S$. The socle of the quotient $Q=\faktor{P}{X}$ is $\tau$-translate of the top of $X$, hence $Q$ is a submodule of another projective module $P'$. The exact sequences
\begin{align*}
\begin{gathered}
0\rightarrow X\rightarrow P\rightarrow Q=\faktor{P}{X}\rightarrow 0\\
0\rightarrow Q\rightarrow P'\rightarrow \faktor{P'}{Q}\rightarrow 0
\end{gathered}
\end{align*} 
show that $\Omega^2(M)\cong X$ where $M=\faktor{P'}{Q}$, since $Q$ is not projective.
\end{proof}

\begin{remark}\label{uniquefilt}
In \cite{sen2018varphi}, we used the notation $\bf\Delta$ instead of $\cB(\Lambda)$. 
\end{remark}

Here we recall the definition of wide subcategory cogenerated by projective-injective $\Lambda$-module. Let $\hat{P}$ be the additive generator of projective-injective $\Lambda$-modules. Then $\cF:=Cogen(\hat{P})=\left\{M\in\modd\Lambda\,\vert\, M\,\text{is submodule of}\,\hat{P} \right\}$ is a torsion-free class. The wide subcategory  associated to $\cF$ is
\begin{align*}
\cW:=\left\{X\in\cF\,\vert\, \forall(g:X\rightarrow Y)\in\cF, \text{then}\,\coker(g)\in\cF\right\}.
\end{align*}

\begin{proposition}\label{PROP filtered module category is equivalent to wide subcategory} The category $Filt(\cB(\Lambda))$ is equivalent to the wide subcategory $\cW$ cogenerated by projective-injective $\Lambda$-module $\hat{P}$.
\end{proposition}

\begin{proof}
$(\Rightarrow)$ Assume that an indecomposable $\Lambda$-module $M$ has $\cB(\Lambda)$-filtration. Therefore $\soc M\in\cS(\Lambda)$ and by the uniseriality lemma \ref{uniserialitylemma} together with \ref{LEM projectivecannotbequotientlemma}, it is a submodule of a projective-injective $\Lambda$-module. Therefore $M\in \cF$.
\par We need to show that for any module $Y\in \cF$ and any map $g: M\rightarrow Y$ in $\cF$, the $\coker(g)$ is in $\cF$. Since $M$ has $\cB(\Lambda)$-filtration, the top of $M$ is in $\cS'(\Lambda)$. This implies that $\soc\coker(g)\in\cS(\Lambda)$, because $\soc\coker(g)$ and $\topp M$ are consecutive simple $\Lambda$-modules. By lemma \ref{uniserialitylemma}, $\coker(g)$ is a submodule of projective-injective module. Therefore  $\coker(g)\in\cF$ for any $g:M\rightarrow Y$ in $\cF$.\par
$(\Leftarrow)$ Let $M$ be an indecomposable $\Lambda$-module in $\cW$. This implies that $M\in\cF$ and for any map $g:M\rightarrow Y$ in $\cF$, $\coker(g)$ is in $\cF$. The former implies that the socle of $M$ is an element of the set $\cS(\Lambda)$, because $M$ is a submodule of projective-injective module. Similarly, the latter implies that the socle of $\coker(g)$ is in $\cF$, hence a submodule of projective-injective module. We need to show that $\topp M$ is in $\cS'(\Lambda)$. Without loss of generality, let $Y$ be an indecomposable module in $\cF$. By the short exact sequence
\begin{align*}
0\rightarrow \im g\rightarrow Y\rightarrow \coker(g)\rightarrow 0
\end{align*}
the top of $\im g$ and the socle of $\coker(g)$ are consecutive modules. Because $\soc\coker(g)\in\cS(\Lambda)$, we get $\topp\im g\in\cS'(\Lambda)$. This shows that $M$ has $\cB(\Lambda)$-filtration since $\topp M\cong \topp\im g\in\cS'(\Lambda)$.
\end{proof}

\begin{corollary}\label{COR category filt is exact} The category $Filt(\cB(\Lambda))$ is extension-closed, exact and abelian. Its simple objects are the elements of $\cB(\Lambda)$.
\end{corollary}
\begin{proof}
By \ref{PROP filtered module category is equivalent to wide subcategory}, $Filt(\cB(\Lambda))$ is a wide subcatecory of $\modd\Lambda$. Therefore it is exact abelian category which is closed under extensions. Elements of $\cB(\Lambda)$ are simple objects because they are hom-orthogonal by the statement \ref{itemitem 4} of proposition \ref{PROP properties of the basesetproperties}.
\end{proof}

Now we are ready to construct new algebra which is filtered by the elements of $\cB(\Lambda)$.
\begin{definition}\label{deffilteredalg} Let $\bm{\varepsilon}(\Lambda)$ be the endomorphism algebra of the direct sum of projective $\Lambda$-modules which are projective covers of elements of the top set $\cS'(\Lambda)$, i.e.
\begin{align*}
\bm{\varepsilon}(\Lambda):=\End_{\Lambda}\left(\bigoplus\limits_{S\in \cS'(\Lambda)}P(S)\right).
\end{align*}
\end{definition}
Since $\bm{\varepsilon}(\Lambda)$ is the endomorphism algebra of $\cB(\Lambda)$-filtered projective modules and the category $Filt(\cB(\Lambda))$ is equivalent to the category of the second syzygies by propositions \ref{PROP all higher syzygies are filtered}, \ref{PROP pdim M>2 then all higher filtered} and \ref{PROP X has filtration then it is second syzygy}, we give the name \emph{syzygy filtered algebra} to $\bm\varepsilon(\Lambda)$.

\subsection{Modules of syzygy filtered algebras}\label{subsec modules over filtered1}

Here we recall some definitions and statements from \cite{aus} mainly chapter 2 to relate the module category of $\Lambda$ and the module category of the syzygy filtered algebra $\bm\varepsilon(\Lambda)$. For any module $A$ in $\modd \Lambda$, the functor $\Hom_{\Lambda}(A,-):\modd \Lambda\rightarrow\modd \Gamma$ where $\Gamma=\End_{\Lambda}(A)^{op}$ is called evaluation functor. 
Suppose $P$ is a projective $\Lambda$-module. Then we denote by $\modd P$ the full subcategory of $\modd\Lambda$ whose objects are those $A$ in $\modd \Lambda$ which have 
minimal projective presentations 
$P_1\rightarrow P_0\rightarrow A\rightarrow 0$ with $P_0,P_1 \in\add P$. 
\begin{proposition}\label{PROP Auslanders construction}\cite[section II, proposition 2.5]{aus}
Let $P$ be a projective $\Lambda$-module and let $\Gamma=\End_{\Lambda} (P)^{op}$. Then the restriction $\Hom(P,-)\vert_{\mod P}:\modd P\rightarrow \modd\Gamma$ of the evaluation functor $\Hom(P,-):\modd\Lambda\rightarrow\modd \Gamma$ is an equivalence of categories.
\end{proposition}

We apply this to our set up. Let $\cP$ be $\cB(\Lambda)$-filtered projective module, i.e.
\begin{align*}
\cP=\bigoplus\limits_{S\in\cS'(\Lambda)}P(S).
\end{align*}
 Therefore by proposition \ref{PROP Auslanders construction}, we get the categorical equivalence
\begin{align*}
\Hom_{\Lambda}(\cP,-):\modd\cP\rightarrow \modd\End_{\Lambda}(\cP)^{op}.
\end{align*}

Now we need to express $\modd\cP$ in terms of $Filt\cB(\Lambda)$.

\begin{proposition}\label{PROP filtered cat is mod P} $Filt(\cB(\Lambda))\cong\modd\cP$ where $\cP=\bigoplus\limits_{S\in\cS'(\Lambda)} P(S)$.
\end{proposition}
\begin{proof}
$\cP$ is the $\cB(\Lambda)$-filtered projective module, so $\cP\in Filt(\cB(\Lambda))$. Also, $\cP$ is in $\modd\cP$ by definition of the full subcategory. Therefore, the projective module $\cP$ is in both $\modd\cP$ and $Filt(\cB(\Lambda))$.
\par Let $M$ be an indecomposable non-projective $\cB(\Lambda)$-filtered module. By proposition \ref{PROP all higher syzygies are filtered}, $P(M)$ and $P(\Omega^1(M))$ are filtered by $\cB(\Lambda)$, hence the sequence
\begin{align*}
\cdots\rightarrow P(\Omega^1(M))\rightarrow P(M)\rightarrow M\rightarrow 0
\end{align*}
is minimal projective presentation of $M$ which shows $M\in\modd\cP$.
\par For the other direction, let $M$ have presentation by the modules in $\modd\cP$, i.e.
\begin{align}\label{presentation}
\xymatrixcolsep{5pt} 
\xymatrix{ P_1\ar[rrr]^{f} &&& P_0\ar[rrr] &&& M\ar[rrr]&&& 0.
}
\end{align}
Notice that $\soc\im f\cong\soc P_0\in\cS(\Lambda)$ and $\topp\im f\cong \topp P_1\in\cS'(\Lambda)$.  By lemma \ref{topsocConsecutive}, $\soc\coker f\in\cS(\Lambda)$. On the other hand $\topp\coker f\cong\topp P_0\in\cS'(\Lambda)$. After combining these, we conclude that $M\cong\coker f$ has $\cB(\Lambda)$-filtration due to proposition \ref{PROP filtered modules have unique filtration}.
\end{proof}

We see that $\modd\cP\cong Filt(\cB(\Lambda))$ and by the definition $\bm\varepsilon(\Lambda)$ is the endomorphism algebra of $\cP$.
 Therefore, we get the diagram
 
\begin{gather}
\begin{aligned}
\xymatrixcolsep{10pt}
\xymatrix{   \text{mod-}\Lambda\ar[rrr]^{\Hom(\cP,-)} &&&\text{mod-}\bm{\varepsilon}(\Lambda)\\
 Filt(\cB(\Lambda)) \ar[rrru]^{\cong} \ar@{^{(}->}[u]
}
\end{aligned}
\label{diagram categorical equivalence}
\end{gather}
where the restriction of $\Hom_{\Lambda}(\cP,-)$ onto $Filt(\cB(\Lambda))$ gives the equivalence of categories mod-$\bm{\varepsilon}(\Lambda)$ and  $Filt(\cB(\Lambda))$. As we will explain in \ref{subsechigher filtrations}, we did not take the opposite of the endomorphism algebra wittingly. Nevertheless we can describe modules over the syzygy filtered algebra in the following way. $\bm{\varepsilon}(\Lambda)=\End_{\Lambda}\cP$ is finite dimensional algebra and $\cP$ is a left $\bm{\varepsilon}(\Lambda)$-module. If $X$ is a right $\Lambda$-module, then $\Hom_{\Lambda}(\cP,X)$ is a right $\bm{\varepsilon}(\Lambda)$-module.
 The simple objects of $Filt(\cB(\Lambda))$ are the elements of the base set $\cB(\Lambda)$, therefore simple $\bm{\varepsilon}(\Lambda)$-modules are of the form 
$\Hom_{\Lambda}(\cP,\Delta_i)$ where $\Delta_i\in\cB(\Lambda)$. Projective $\bm\varepsilon(\Lambda)$-modules are of the form $\Hom_{\Lambda}(\cP,P)$ where $P$ is $\cB(\Lambda)$-filtered projective. In general, any $X\in\modd\bm\varepsilon(\Lambda)$ is isomorphic to $\Hom_{\Lambda}(\cP,X')$ for some $X'\in Filt(\cB(\Lambda))$.

\begin{remark}\label{Rem equivalnece of resolutions} The evaluation functor $\Hom_{\Lambda}(\cP,-)$ is exact, because $\cP$ is projective module. If $M\in Filt(\cB(\Lambda))$ have the projective resolution
\begin{align*}
\cdots\rightarrow P_2\rightarrow P_1\rightarrow P_0\rightarrow M\rightarrow 0
\end{align*}
then 
\begin{align*}
\cdots\rightarrow\Hom_{\Lambda}(\cP, P_2)\rightarrow\Hom_{\Lambda}(\cP, P_1)\rightarrow \Hom_{\Lambda}(\cP,P_0)\rightarrow \Hom_{\Lambda}(\cP,M)\rightarrow 0
\end{align*}
is the projective resolution of $X\in\modd\bm\varepsilon(\Lambda)$ where $X=\Hom_{\Lambda}(\cP,M)$. Each $\Hom_{\Lambda}(\cP,P_i)$, $i\geq 0$ are projective $\bm\varepsilon(\Lambda)$-modules.
Any nonsplit exact sequence 
\begin{align*}
0\rightarrow X_1\rightarrow X_2\rightarrow X_3\rightarrow 0
\end{align*}
in $\modd\bm\varepsilon(\Lambda)$ can be carried into $\modd\Lambda$ by the categorical equivalence, i.e.
\begin{align*}
0\rightarrow\Delta X_1\rightarrow\Delta X_2\rightarrow\Delta X_3\rightarrow 0
\end{align*}
is exact in $\modd\Lambda$ where each $X_i\cong\Hom_{\Lambda}(\cP,\Delta X_i)$.
By the abuse of terminology, we say $M\in\modd\Lambda$ is the \emph{corresponding module} to $X\in\modd\bm\varepsilon(\Lambda)$ if $\Hom_{\Lambda}(\cP,M)\cong X$. To make distinction, we use the following notation, for any $X\in\modd\bm\varepsilon(\Lambda)$, we denote the corresponding $\cB(\Lambda)$-filtered $\Lambda$-module by $\Delta X$ i.e. $\Hom_{\Lambda}(\cP,\Delta X)\cong X$. 
\end{remark}
 We exploit this idea when comparing various homological dimensions based on the projective dimensions of the modules.

\begin{proposition}\label{PROP carrying resolution on lambda to epsilon} If $M\in\modd\Lambda$ has finite projective dimension $d\geq 2$, then the projective dimension of $M'=\Hom_{\Lambda}(\cP,\Omega^2(M))$ in $\modd\bm\varepsilon(\Lambda)$ is $d-2$. If $M\in\modd\Lambda$ has infinite projective dimension, then the projective dimension of $M'=\Hom_{\Lambda}(\cP,\Omega^2(M))$ in $\modd\bm\varepsilon(\Lambda)$ is infinite. 
\end{proposition}
\begin{proof}
We consider the projective resolution of $M$ in $\modd\Lambda$
\begin{align*}
\cdots\rightarrow P_2\rightarrow P_1\rightarrow P_0\rightarrow M\rightarrow 0.
\end{align*}
By proposition \ref{PROP all higher syzygies are filtered} each $P_i$ and $\Omega^i(M)$ where $2\leq i\leq d$ have $\cB(\Lambda)$-filtration.
The evaluation functor $\Hom_{\Lambda}(\cP,-)$ is exact, therefore the projective resolution of $\Omega^2(M)$ in $\modd\Lambda$
\begin{align*}
\cdots\rightarrow P_3=P(\Omega^3(M))\rightarrow P_2=P(\Omega^2(M))\rightarrow \Omega^2(M)\rightarrow 0
\end{align*}
can be carried into $\modd\bm\varepsilon(\Lambda)$ as
\begin{align}\label{eqeq1}
\cdot\rightarrow \Hom_{\Lambda}(\cP,P(\Omega^3(M)))\rightarrow \Hom_{\Lambda}(\cP,P(\Omega^2(M)))\rightarrow \Hom_{\Lambda}(\cP,\Omega^2(M))\rightarrow 0
\end{align}
By the categorical equivalence \ref{diagram categorical equivalence}, 
 $\Hom_{\Lambda}(\cP,P(\Omega^i(M)))$ where $i\geq 2$ is projective $\bm\varepsilon(\Lambda)$-module. By the remark \ref{Rem equivalnece of resolutions}, \ref{eqeq1} is the projective resolution of $M'$ where $\Omega^2(M)$ is the corresponding module, i.e. $M'\cong\Hom_{\Lambda}(\cP,\Omega^2(M))$. We conclude that
\begin{align*}
\pdim_{\bm\varepsilon(\Lambda)}M'&=\pdim_{\bm\varepsilon(\Lambda)}\Hom_{\Lambda}(\cP,\Omega^2(M))\\
&=\pdim_{\Lambda}\Omega^2(M)\\
&=\pdim_{\Lambda}M-2\\
&=d-2.
\end{align*}
If $d$ is infinite, then $\Omega^2(M)$ has also infinite projective dimension. By \ref{eqeq1}, we get $\pdim_{\Lambda} M=\pdim_{\Lambda}\Omega^2(M)=\pdim_{\syz}\Hom_{\Lambda}(\cP,\Omega^2(M))=\infty$.
\end{proof}

 \begin{proposition}\label{PROP gldim reduction by 2} If $\Lambda$ is a cyclic Nakayama algebra with finite global dimension $d\geq 2$, then $\gldim\Lambda=\gldim\bm\varepsilon(\Lambda)+2$.
 \end{proposition}
\begin{proof}
Because $\gldim\Lambda\geq 2$, we can express it by the second syzygies, i.e.
\begin{align*}
\gldim\Lambda&=\sup\left\{\pdim M\vert\, M\in\modd\Lambda\right\}\\
&=\sup\left\{\pdim\Omega^2(M)\vert\, M\in\modd\Lambda\right\}+2.
\end{align*}
By proposition \ref{PROP pdim M>2 then all higher filtered}, $\Omega^2(M)$ is a $\cB(\Lambda)$-filtered module, therefore
\begin{align*}
\gldim\Lambda&=2+\sup\left\{\pdim \Delta X\vert\, \Delta X\in Filt(\cB(\Lambda)), \Delta X\cong\Omega^2(M)\in\modd\Lambda\right\}.
\end{align*}
By the categorical equivalence $\Hom_{\Lambda}(\cP,\Delta X)\cong X$, we get
\begin{align*}
\gldim\Lambda&=2+\sup\left\{\pdim_{\bm\varepsilon(\Lambda)} X\vert\,\Delta X\in Filt(\cB(\Lambda),\Hom_{\Lambda}(\cP,\Delta X)=X \right\}\\
&=2+\sup\left\{\pdim_{\bm\varepsilon(\Lambda)}X\vert\,  X\in\modd \bm\varepsilon(\Lambda))\right\}\\
&=2+\gldim\bm\varepsilon(\Lambda).
\end{align*}\end{proof}

It turns out that the syzygy filtered algebra is again a Nakayama algebra which allows us to make the mathematical induction on the various homological dimensions of the original algebra by exploiting the idea in the proof of proposition \ref{PROP gldim reduction by 2} (see proposition \ref{PROP carrying resolution on lambda to epsilon}). 

\begin{proposition}\label{PROP filtered algebra is Nakayama} If $\gldim\Lambda>2$, then  $\bm{\varepsilon}(\Lambda)$ is a Nakayama algebra of rank $r$ where $r$ is the number of irredundant relations defining $\Lambda$.
\end{proposition}

\begin{proof}
When global dimension of $\Lambda$ is two, by the reduction \ref{PROP gldim reduction by 2}, global dimension of $\bm\varepsilon(\Lambda)$ is zero.  In the last section, we show that it is a semisimple algebra (see proposition \ref{PROP gldim 2 implies filtered algebra is semisimple}). So we need to exclude this case.\par 
When the global dimension is greater than two, by the reduction \ref{PROP gldim reduction by 2}, we get $\gldim\bm\varepsilon(\Lambda)\geq 1$ which shows the syzygy filtered algebra is not semisimple. By the construction \ref{diagram categorical equivalence}, the category of $\bm\varepsilon(\Lambda)$-modules is equivalent to the category of $\cB(\Lambda)$-filtered modules. By proposition \ref{PROP filtered module category is equivalent to wide subcategory}, $Filt(\cB(\Lambda))$ is wide subcategory of $\modd\Lambda$. Since $\Lambda$ is Nakayama algebra, the indecomposable objects of $Filt(\cB(\Lambda))$ are uniserial. By proposition \ref{PROP filtered modules have unique filtration}, each $\cB(\Lambda)$-filtered module has a unique filtration, therefore $\bm\varepsilon(\Lambda)$ is a Nakayama algebra.\par
The number of simple objects of $Filt(\cB(\Lambda))$ is equal to the cardinality of the base set which is $r$. Because $\bm\varepsilon(\Lambda)$ is an endomorphism algebra of $r$-many $\cB(\Lambda)$-filtered projective modules, we obtain $\rank\bm\varepsilon(\Lambda)=r$.
\end{proof}

\begin{proposition}\label{PROP infinite global dimension, implies epsilon is cyclic}
If global dimension of $\Lambda$ is infinite, then
\begin{enumerate}[label=\roman*)]
\item global dimension of $\bm{\varepsilon}(\Lambda)$ is infinite,
\item $\bm\varepsilon(\Lambda)$ is cyclic.
\end{enumerate}
\end{proposition}
\begin{proof}
 There exists $\Lambda$-module $M$ with $p\dim M=\infty$ by the assumption, hence\\ $\pdim_{\bm\varepsilon(\Lambda)}\Hom_{\Lambda}(\cP,\Omega^2(M))$ is infinite by proposition \ref{PROP carrying resolution on lambda to epsilon}. So global dimension of the syzygy filtered algebra is unbounded. By proposition \ref{PROP filtered algebra is Nakayama}, it is a Nakayama algebra. Hence the underlying quiver has to be cyclic, because linear Nakayama algebras are representation-directed and have finite global dimension.
\end{proof}

We want to justify why our approach does not coincide with \cite{chen2012retractions} in general. If $\Lambda$ is not selfinjective, then passing from $\Lambda$ to $\bm\varepsilon(\Lambda)$ reduces  projective dimensions of modules  exactly by two. Also, by proposition \ref{PROP gldim reduction by 2}, our construction works for the case of finite global dimension.

\begin{remark}\label{Enakayama}  $\bm{\varepsilon}(\Lambda)$ is Nakayama algebra so in principle we can express its defining relations. They are managed by $\cB(\Lambda)$ in the following way. If $\Hom_{\Lambda}(\cP,P)$ is projective module which is minimal in $\modd\varepsilon(\Lambda)$, then it gives the relation 

\begin{align*}
\beta_{i_j+c_{i_j}-1}\circ\beta_{i_j+c_{i_j}-2}\circ\dots\circ\beta_{i_j+1}\circ\beta_{i_j}=0
\end{align*}
where $i_j$ is the index of the top of $\Hom_{\Lambda}(\cP,P)$ and $c_{i_j}$ is the length of $\Hom_{\Lambda}(\cP,P)$ in $\modd\bm\varepsilon(\Lambda)$ or equivalently by the remark \ref{REM delta lengths}, $c_{i_j}$ is the number of composition factors of $P$ with respect to $\cB(\Lambda)$-filtration because of the categorical equivalence \ref{diagram categorical equivalence}.
\end{remark}

\begin{remark}\cite[remark 3.4.1]{sen2018varphi} \label{Rem characterization of selfinjective algebras} Let $\rank\Lambda=N$. Then the following statements are equivalent:
\begin{enumerate}[label=\arabic*)]
\item $\Lambda$ is selfinjective.
\item All projective modules are injective.
\item All indecomposable projective modules have the same length.
\item\label{ccc2} Non-isomorphic projectives have non-isomorphic socle.
\item All radicals of indecomposable projective modules have the same length.
\item\label{ccc1} Every indecomposable projective $\Lambda$-module is minimal.
\item Every indecomposable projective module is projective-injective.
\item Each class of projectives (def \ref{defprojectiveclass of modules}) has exactly one object.
\end{enumerate}
\end{remark}

We prove the statement \ref{thmselfinjective} of Theorem \ref{bigthm1}.

\begin{proposition}\label{PROP algebra is selfinjective iff equivalent to filtered algebra} For any cyclic Nakayama algebra $\Lambda$,
$\Lambda\cong\bm{\varepsilon}(\Lambda)$ if and only if $\Lambda$ is selfinjective.
\end{proposition}
\begin{proof}
Assume that $\Lambda$ is selfinjective. Then we get
\begin{align*}
\bm{\varepsilon}(\Lambda):=\End\left(\bigoplus\limits_{S\in \cS'(\Lambda)}P(S)\right)=\End\left(\bigoplus\limits_{1\leq i\leq N} P_i\right)\cong \Lambda,
\end{align*}
since every projective $\Lambda$-module is minimal by the remark \ref{Rem characterization of selfinjective algebras} \ref{ccc1} which implies that the top set  $\cS'(\Lambda)$ contains all simple $\Lambda$-modules.\par
For the other direction, assume that $\Lambda\cong\bm{\varepsilon}(\Lambda)$. By the construction \ref{diagram categorical equivalence}, $Filt(\cB(\Lambda))\cong\modd\Lambda$. Hence the number of non-isomorphic simple $\Lambda$-modules has to be equal to the number of non-isomorphic $\cB(\Lambda)$-filtered simple modules. By the last statement \ref{item last in prop 2.18} of proposition \ref{PROP properties of the basesetproperties}, we conclude that $\cS(\Lambda)=\cS'(\Lambda)=\left\{S_1,S_2,\ldots,S_N\right\}$. Therefore, each non-isomorphic projective module have non-isomorphic socle. By the remark \ref{Rem characterization of selfinjective algebras} \ref{ccc2}, $\Lambda$ is selfinjective.
\end{proof}

Now we can add the equivalent statements below to the list given in the remark \ref{Rem characterization of selfinjective algebras}.
\begin{enumerate}[label=\arabic*)]
\item $\Lambda$ is selfinjective.
\item $\Lambda\cong\bm\varepsilon(\Lambda)$.
\item $\cS(\Lambda)=\cS'(\Lambda)=\cB(\Lambda)=\left\{S_1,S_2,\ldots,S_N\right\}$.
\item $r=N=\left\vert\cS'(\Lambda)\right\vert=\left\vert\cS(\Lambda)\right\vert=\left\vert\cB(\Lambda)\right\vert$.
\end{enumerate}

\subsection{Higher syzygy filtered algebras}\label{subsechigher filtrations}
Based on the construction of the syzygy filtered algebras \ref{deffilteredalg}, we can construct higher syzygy filtered algebras recursively.
\begin{definition} \label{defhigherfilteredalgebra} Let $\Lambda$ be a cyclic Nakayama algebra and $\bm{\varepsilon}(\Lambda)$ be its syzygy filtered algebra. \emph{ $n$-th syzygy filtered algebra} $\bm{\varepsilon}^n(\Lambda)$ is the endomorphism algebra of direct sum of projective covers of simple $\bm\varepsilon^{n-1}(\Lambda)$-modules which are elements of $\cS'(\bm\varepsilon^{n-1}(\Lambda))$,  i.e.
\begin{align*}
\bm{\varepsilon}^n(\Lambda):=\End_{\bm{\varepsilon}^{n-1}(\Lambda)}\left(\bigoplus\limits_{S\in\cS'(\bm{\varepsilon}^{n-1}(\Lambda))}P(S)\right)
\end{align*}
provided that $\bm{\varepsilon}^{n-1}(\Lambda)$ is a cyclic Nakayama algebra.
\end{definition}
We want to explain the reason behind the requirement that $\bm\varepsilon^{n-1}(\Lambda)$ is cyclic. To apply the construction in the definition \ref{deffilteredalg}, the socle, top and base sets of $\bm\varepsilon^{n-1}(\Lambda)$ have to be well-defined, which is true only for cyclic Nakayama algebras. The statements \ref{item3} and \ref{itemitem 5} of proposition \ref{PROP properties of the basesetproperties}  is not true for linear Nakayama algebras because projective-injective module $P_1$ cannot have $\cB(\Lambda)$-filtered submodule which should arise as the second syzgy of the simple projective module. Obviously it is not possible.
\par We can interpret $\modd\bm{\varepsilon}^n(\Lambda)$ in terms of $\cB(\bm{\varepsilon}^{n-1}(\Lambda))$-filtered $\bm{\varepsilon}^{n-1}(\Lambda)$ modules as we discussed in the subsection \ref{subsec filt} where $n=1$ and $\bm\varepsilon^0(\Lambda)=\Lambda$. Because of the cyclicity assumption on $\bm\varepsilon^{n-1}(\Lambda)$, all $\cS(\bm{\varepsilon}^{n-1}(\Lambda)),\cS'(\bm{\varepsilon}^{n-1}(\Lambda)),\cB(\bm{\varepsilon}^{n-1}(\Lambda)),Filt(\cB(\bm{\varepsilon}^{n-1}(\Lambda)))$ are well-defined and the results can be carried in this set up. In other words, the $n$-th syzygy filtered algebra $\bm\varepsilon^{n}(\Lambda)$ is the syzygy filtered algebra of the  cyclic Nakayama algebra $\bm\varepsilon^{n-1}(\Lambda)$, that is
\begin{align}\label{cons recursion}
\bm\varepsilon^{n}(\Lambda)\cong\bm\varepsilon\left(\bm\varepsilon^{n-1}(\Lambda)\right).
\end{align}

If we denote projective modules in $Filt\cB(\bm{\varepsilon}^{n-1}(\Lambda))$ by $\cP_{\bm\varepsilon^{n-1}(\Lambda)}$, i.e. 
\begin{align*}
\cP_{\bm\varepsilon^{n-1}(\Lambda)}=\bigoplus\limits_{S\in\cS'(\bm{\varepsilon}^{n-1}(\Lambda))} P(S),
\end{align*}
then we see that $\modd\cP_{\bm\varepsilon^{n-1}(\Lambda)}\cong Filt(\cB(\bm\varepsilon^{n-1}(\Lambda)))$ by proposition  \ref{PROP filtered cat is mod P}, and by the definition of  syzygy filtered algebra, $\bm\varepsilon^n(\Lambda)$ is the endomorphism algebra of $\cP_{\bm\varepsilon^{n-1}(\Lambda)}$.
 Therefore, we get the diagram
\begin{gather}
\begin{aligned}
\xymatrixcolsep{10pt}
\xymatrix{ \text{mod-}\bm{\varepsilon}^{n-1}(\Lambda)\ar[rrrrr]^{\Hom(\cP_{\bm\varepsilon^{n-1}(\Lambda)},-)} &&&&&\text{mod-}\bm{\varepsilon}^n(\Lambda)\\
 Filt(\cB(\bm{\varepsilon}^{n-1}(\Lambda))) \ar[rrrrru]^{\cong} \ar@{^{(}->}[u]
}
\label{diagram for higher filtration}
\end{aligned}
\end{gather}

The restriction of $\Hom_{\bm{\varepsilon}^{n-1}(\Lambda)}(\cP_{\bm\varepsilon^{n-1}(\Lambda)},-)$ onto $Filt(\cB(\bm{\varepsilon}^{n-1}(\Lambda)))$ gives the equivalence of categories mod-$\bm{\varepsilon}^{n}(\Lambda)$ and  $Filt(\cB(\bm{\varepsilon}^{n-1}(\Lambda)))$. Because of the recursive nature of the construction (see \ref{cons recursion}) we do not want to take the opposite algebra back and forth. We can view all the modules as right modules. We can describe modules over the higher syzygy filtered algebras in the following way. $\bm{\varepsilon}^n(\Lambda)=\End_{\bm{\varepsilon}^{n-1}(\Lambda)}\cP_{\bm\varepsilon^{n-1}(\Lambda)}$ is finite dimensional algebra and $\cP_{\bm\varepsilon^{n-1}(\Lambda)}$ is left $\bm{\varepsilon}^n(\Lambda)$-module. If $X$ is right $\bm{\varepsilon}^{n-1}(\Lambda)$-module, then $\Hom_{\bm{\varepsilon}^{n-1}(\Lambda)}(\cP_{\bm\varepsilon^{n-1}(\Lambda)},X)$ is right $\bm{\varepsilon}^n(\Lambda)$-module. Briefly we have:
\begin{enumerate}[label=\roman*)]
\item The simple objects of $Filt(\cB(\bm{\varepsilon}^{n-1}(\Lambda)))$ are the elements of the base set $\cB(\bm{\varepsilon}^{n-1}(\Lambda))$, therefore simple $\bm{\varepsilon}^n(\Lambda)$-modules are of the form 
$\Hom_{\bm{\varepsilon}^{n-1}(\Lambda)}(\cP_{\bm\varepsilon^{n-1}(\Lambda)},\Delta_i)$ where $\Delta_i\in\cB(\bm{\varepsilon}^{n-1}(\Lambda))$. 
\item Projective $\bm\varepsilon^n(\Lambda)$-modules are of the form $\Hom_{\bm{\varepsilon}^{n-1}(\Lambda)}(\cP_{\bm\varepsilon^{n-1}(\Lambda)},P)$ where $P$ is $\cB(\bm{\varepsilon}^{n-1}(\Lambda))$-filtered projective. 
\item In general, any $X\in\modd\bm{\varepsilon}^{n-1}(\Lambda)$ is isomorphic to $\Hom_{\bm{\varepsilon}^{n-1}(\Lambda)}(\cP_{\bm\varepsilon^{n-1}(\Lambda)},X')$ for some $X'\in Filt(\cB(\bm{\varepsilon}^{n-1}(\Lambda)))$.
\end{enumerate}

\subsection{Relationship between $\Lambda$ and $\varepsilon^n(\Lambda)$}
 One can construct the higher syzygy filtered algebras directly from the original algebra. Just to illustrate the idea, we focus on the case $n=2$. Let $\Lambda$ and $\bm\varepsilon(\Lambda)$ be cyclic Nakayama algebras (see \ref{PROP infinite global dimension, implies epsilon is cyclic}, \ref{PROP pdim simple 4 implies cyclic then linear} when $\syz$ is cyclic). Consider the following diagram:

\begin{gather}
\begin{aligned}
\xymatrixcolsep{10pt}
\xymatrix{   \text{mod-}\Lambda\ar[rrr]^{\Hom(\cP_{\Lambda},-)} &&&\text{mod-}\bm{\varepsilon}(\Lambda)\ar[rrr]^{\Hom(\cP_{\bm\varepsilon(\Lambda)},-)} &&&\text{mod-}\bm{\varepsilon}^2(\Lambda)\\
 Filt(\cB(\Lambda))\ar@{^{(}->}[u] \ar[rrru]^{\cong} &&&Filt(\cB(\bm\varepsilon(\Lambda))) \ar[rrru]^{\cong} \ar@{^{(}->}[u]
}
\end{aligned}\label{seconddiagram}
\end{gather}
\begin{proposition}
The composition $\Hom_{\syz}(\cP_{\syz},-)\circ\ho{-}:\modd\Lambda\rightarrow \modd\bm\varepsilon^2(\Lambda)$
 is equivalent to
 \begin{align}\label{functor 2}
 \Hom_{\Lambda}(\cP',-):\modd\Lambda\rightarrow \modd\bm\varepsilon^2(\Lambda)
 \end{align}
   where
   \begin{align}
  \cP'=\bigoplus\limits_{S\in\cS'(\Lambda)}P(\Omega^2(S))=\bigoplus\limits_{_{\substack{S\cong\topp P\\ P\in\modd\Lambda,\,\text{minimal}}}} P(\Omega^4(S))
  \end{align}
   and $S$ is top of minimal projective $\Lambda$-module in the right hand side. 
 \end{proposition}
 \begin{proof}
  We proved that $Filt(\cB(\Lambda))$ is equivalent to $\modd\bm\varepsilon(\Lambda)$ via the evaluation functor $\Hom_{\Lambda}(\cP_{\Lambda},-)$ and by the same reason $ Filt(\cB(\bm\varepsilon(\Lambda)))$ is equivalent to $\modd\bm\varepsilon^2(\Lambda)$ via the evaluation functor $\Hom_{\bm\varepsilon(\Lambda)}(\cP_{\bm\varepsilon(\Lambda)},-)$. We claim that any element of the base set $\cB(\syz)$ can be realized as the fourth syzygies of simple $\Lambda$-modules which can be interpreted as the generalization of proposition \ref{PROP x is an element of the base set iff it is second syzygy}.

 $M\in\cB(\bm\varepsilon(\Lambda))$ if and only if $M\cong\Omega^2(S')$ in $\modd\bm\varepsilon(\Lambda)$ and $S'$ is top of minimal projective module $P$ by propositions \ref{PROP x in base set iff x is omegaone of radical} and \ref{PROP x is an element of the base set iff it is second syzygy}. Since $M\cong\Omega^2(S')$, we have the exact sequences
\begin{gather*}
0\rightarrow M\rightarrow P(\Omega^1(S'))\rightarrow \Omega^1(S')\rightarrow 0\\
0\rightarrow\Omega^1(S')\rightarrow P(S')\rightarrow S'\rightarrow 0
\end{gather*}
in $\modd\bm\varepsilon(\Lambda)$. By the categorical equivalence between $Filt(\cB(\Lambda))$ and $\modd\bm\varepsilon(\Lambda)$ via the evaluation functor, the following sequences are exact in $\modd\Lambda$
\begin{gather*}
0\rightarrow \Delta M\rightarrow\Delta P(\Omega^1(S'))\rightarrow \Delta\Omega^1(S')\rightarrow 0\\
0\rightarrow\Delta\Omega^1(S')\rightarrow\Delta P(S')\rightarrow\Delta S'\rightarrow 0
\end{gather*}
where $\Hom_{\Lambda}(\cP,\Delta M)\cong M,\Hom_{\Lambda}(\cP,\Delta S')\cong S',\Hom_{\Lambda}(\cP,\Delta \Omega^1(S'))\cong \Omega^1(S'),\\ \Hom_{\Lambda}(\cP,\Delta P(\Omega^1(S')))\cong P(\Omega^1(S'))$ and $ \Hom_{\Lambda}(\cP,\Delta P(S'))\cong P(S')$.\\

$S'$ is a simple $\bm\varepsilon(\Lambda)$-module, so the corresponding module $\Delta S'\in\modd\Lambda$ is an element of the base set $\cB(\Lambda)$ by the categorical equivalence \ref{diagram categorical equivalence} and the remark \ref{Rem equivalnece of resolutions}. On the other hand, there exists a simple $\Lambda$-module $S$ which is a top of a minimal projective $\Lambda$-module such that $\Omega^2(S)\cong \Delta S'$ by the same argument. As a result $\Omega^4(S)\cong\Omega^2(\Delta S')\cong \Delta M$ in $\modd\Lambda$. By the functor $\Hom_{\Lambda}(\cP',-)$ introduced in \ref{functor 2}, simple $\bm\varepsilon^2(\Lambda)$-modules are of the form
\begin{align}\label{functor 3}
\Hom_{\Lambda}(\cP',\Omega^4(S))=\Hom_{\Lambda}(\cP',\Delta M).
\end{align}

Let $Filt(\cB^2(\Lambda))$ denote the category of $\Lambda$-modules which are filtered by the fourth syzygies of simple $\Lambda$-modules. 
By the definition \ref{defhigherfilteredalgebra} we have 
\begin{align*}
\bm\varepsilon^2(\Lambda)&=\End_{\bm\varepsilon(\Lambda)}\left(\bigoplus\limits_{S\in\cS'(\bm\varepsilon(\Lambda))} P(S)\right)
\end{align*} which is equivalent to
\begin{align*}
\bm\varepsilon^2(\Lambda)&=\End_{\bm\varepsilon(\Lambda)}\left(\bigoplus\limits_{M\in\cB(\bm\varepsilon(\Lambda))} P(M) \right)
\end{align*}
where $ \topp M\in\cS'(\bm\varepsilon(\Lambda))$.
By the remark \ref{Rem equivalnece of resolutions} we can lift $\modd\syz$ to $\modd\Lambda$, i.e.
\begin{align*}
\bm\varepsilon^2(\Lambda)&=\End_{\Lambda}\left(\bigoplus\limits_{\substack{\Delta M\in\modd\Lambda\\\Hom_{\Lambda}(\cP,\Delta M)\cong M}} P(\Delta M) \right)
\end{align*}
where  $\Hom_{\Lambda}(\cP,\Delta M)\cong M \in\cB(\bm\varepsilon(\Lambda))$.
By the result \ref{functor 3}, any $\Delta M$ is the fourth syzygy of simple $\Lambda$-module, so
\begin{align*}
\bm\varepsilon^2(\Lambda)&=\End_{\Lambda}\left(\bigoplus\limits_{\substack{S\in\modd\Lambda\\\Omega^4(S)\cong\Delta M}} P(\Omega^4(S))\right)
\end{align*}

By the definition of $\cB^2(\Lambda)$-filtered modules, we conclude that

\begin{align*}
\bm\varepsilon^2(\Lambda)&=\End_{\Lambda}\left(\bigoplus\limits_{\substack{N\in\cB^2(\Lambda)\\}} P(N) \right).
\end{align*}
\end{proof}
Therefore the diagram \ref{seconddiagram} can be extended into

\[\xymatrixcolsep{10pt}
\xymatrix{   \text{mod-}\Lambda\ar@/^2.5pc/[rrrrrr]^{\Hom_{\Lambda}(\cP',-)}\ar[rrr]^{\Hom(\cP_{\Lambda},-)} &&&\text{mod-}\bm{\varepsilon}(\Lambda)\ar[rrr]^{\Hom(\cP_{\bm\varepsilon(\Lambda)},-)} &&&\text{mod-}\bm{\varepsilon}^2(\Lambda)\\
 Filt(\cB(\Lambda))\ar@{^{(}->}[u] \ar[rrru]^{\cong} &&&Filt(\cB(\bm\varepsilon(\Lambda))) \ar[rrru]^{\cong} \ar@{^{(}->}[u]\\
 Filt(\cB^2(\Lambda))\ar@{^{(}->}[u]\ar@/_2.5pc/@[black][rrrrrruu]_{\cong}
}\]

where $Filt(\cB^2(\Lambda))\cong\modd\bm\varepsilon^2(\Lambda)$ via the functor $\Hom_{\Lambda}(\cP',-)$.

\subsection{Examples}
We want to give examples to make the concepts we have introduced so far concrete.

\begin{example} Let $\Lambda$ be cyclic Nakayama algebra with $\rank\Lambda=5$ given by the relations
\begin{align*}
\alpha_{5}\alpha_{4}\alpha_{3}\alpha_{2}\alpha_{1}=0\\
\alpha_{1}\alpha_{5}\alpha_{4}\alpha_{3}\alpha_{2}=0\\
\alpha_{3}\alpha_{2}\alpha_{1}\alpha_{5}\alpha_{4}\alpha_{3}=0\\
\alpha_{4}\alpha_{3}\alpha_{2}\alpha_{1}\alpha_{5}\alpha_{4}=0.
\end{align*}
The corresponding Kupisch series are $(5,5,6,6,6)$. All projective modules are minimal except $P_5$. The socle set is $\left\{S_1,S_3,S_4,S_5\right\}$, the top set is $\left\{S_1,S_2,S_4,S_5\right\}$. The base set is 
\begin{align*}
\cB(\Lambda)=\left\{S_1,\begin{matrix}
S_2\\S_3
\end{matrix},S_4,S_5 \right\}.
\end{align*}
$P_1$, $P_2$, $P_4$ and $P_5$ are $\cB(\Lambda)$-filtered. Therefore the Kupisch series of $\bm\varepsilon(\Lambda)$ is $(4,4,5,5)$ where the simple modules are $\Delta_1=\Hom(\cP,S_1)$, $\Delta_2=\Hom(\cP,\begin{matrix}
S_2\\S_3
\end{matrix})$,$\Delta_3=\Hom(\cP,S_4)$, $\Delta_4=\Hom(\cP,S_5)$ and the rank is $4$. We can apply the construction again. Algebra with Kupisch series $(4,4,5,5)$ have the relations
\begin{align*}
\alpha_4\alpha_3\alpha_2\alpha_1=0\\
\alpha_1\alpha_4\alpha_3\alpha_2=0\\
\alpha_3\alpha_2\alpha_1\alpha_4\alpha_3=0
\end{align*}
The corresponding base set is $\left\{\Delta_1,\begin{matrix}
\Delta_2\\\Delta_3
\end{matrix},\Delta_4 \right\}$. The filtered projective modules are $P_1,P_2$ and $P_4$. Therefore the Kupisch series of $\bm\varepsilon^2(\Lambda)$ is $(3,3,4)$ where the simple modules are of the form
\begin{align*}
\cE_1=\Hom(\cP,\Delta_1), \cE_2=\Hom(\cP,\begin{matrix}
\Delta_2\\\Delta_3\end{matrix} ), \cE_3=\Hom(\cP,\Delta_4).
\end{align*}
 The relations are
\begin{align*}
\alpha_3\alpha_2\alpha_1=0\\
\alpha_1\alpha_3\alpha_2=0
\end{align*} so the base set is $\left\{\cE_1,\begin{matrix}
\cE_2\\\cE_3
\end{matrix} \right\}$. The Kupisch series of $\bm\varepsilon^3(\Lambda)$ is $(2,2)$ which is selfinjective.\\
We want to illustrate the characterization of selfinjective algebras \ref{Rem characterization of selfinjective algebras} here, the relations of selfinjective algebra given by $(2,2)$ is $\alpha_2\alpha_1=0$ and $\alpha_1\alpha_2=0$. The base set becomes $\left\{S_1,S_2\right\}$ which is the simple modules of $\bm\varepsilon^3(\Lambda)$. Therefore all projective modules have $\cB(\Lambda)$-filtration and $\bm\varepsilon^3(\Lambda)\cong\bm\varepsilon^4(\Lambda)$.
\end{example}

\begin{example} We point out that our method is different  than the one given in \cite{chen2012retractions}. Consider the algebra with Kupisch series $(4,3,3,4,3,3)$. The relations are
\begin{align*}
\alpha_4\alpha_3\alpha_2=0\\
\alpha_5\alpha_4\alpha_3=0\\
\alpha_1\alpha_6\alpha_5=0\\
\alpha_2\alpha_1\alpha_6=0
\end{align*}
The base set is 
\begin{align*}
\left\{S_5,\begin{matrix}
S_6\\S_1
\end{matrix},S_2,\begin{matrix}
S_3\\S_4
\end{matrix} \right\}.
\end{align*}
Therefore $P_2,P_3.P_5,P_6$ have $\cB(\Lambda)$-filtration, and the Kupisch series of $\bm\varepsilon(\Lambda)$ is $(2,2,2,2)$ which we get in the first step as opposed to the one in \cite{chen2012retractions}.
\end{example}

\begin{example} Let $N=5$ and the relations are  $\alpha_3\alpha_2=0, \alpha_1\alpha_5=0$. This determines minimal projectives directly which are $P_2=\begin{vmatrix}
2\\3
\end{vmatrix}$ and $P_5=\begin{vmatrix}
5\\1
\end{vmatrix}$. $P_3$ and $P_4$ contain $P_5$ as a submodule, and $P_1$ contain $P_2$ as a submodule. The complete list of representatives of projective modules is
\begin{align*}
P_1=\begin{vmatrix}
1\\2\\3
\end{vmatrix}, P_2=\begin{vmatrix}
2\\3
\end{vmatrix}, P_3=\begin{vmatrix}
3\\4\\5\\1
\end{vmatrix},P_4=\begin{vmatrix}
4\\5\\1
\end{vmatrix}, P_5=\begin{vmatrix}
5\\1
\end{vmatrix}
\end{align*}.
The corresponding Kupisch series is $(3,2,4,3,2)$. Notice that there are two different socles of projective modules $S_3$ and $S_1$. The base set is $\left\{P_2, P_4\right\}$, therefore $\syz$ is $\Aa_1\oplus\Aa_1$.
\end{example}

\begin{example} We want to give examples with finite global dimension. Consider algebra given by $(4,3,2,4,3,2,4,3,2)$. The relations are $\alpha_4\alpha_3=0,\alpha_7\alpha_6=0,\alpha_1\alpha_9=0$. The base set is $\left\{P_5,P_8,P_2\right\}$ , hence $\bm\varepsilon(\Lambda)$ is semisimple algebra $\Aa_1\oplus\Aa_1\oplus\Aa_1$.
\end{example}

We view hereditary algebra of type $\Aa_n$ as linear Nakayama algebra, given by the single relation $\alpha_n=0$.

\begin{example} Let $\Lambda$ be given by Kupisch series $(3,3,3,4,3,3,3,3,2)$. Relations are
\begin{gather*}
\alpha_3\alpha_2\alpha_1=0\\
\alpha_4\alpha_3\alpha_2=0\\
\alpha_5\alpha_4\alpha_3=0\\
\alpha_7\alpha_6\alpha_5=0\\
\alpha_8\alpha_7\alpha_6=0\\
\alpha_9\alpha_8\alpha_1=0\\
\alpha_1\alpha_9=0
\end{gather*}
The base set is
\begin{align*}
\left\{S_1,\begin{matrix}
S_2\\S_3
\end{matrix},S_4,S_5,\begin{matrix}
S_6\\S_7
\end{matrix},S_8,S_9\right\}.
\end{align*}
The syzygy filtered algebra is given by Kupisch series $(2,2,3,2,2,3,2)$. If we apply the construction once more, $\bm\varepsilon^2(\Lambda)$ is isomorphic to linear Nakayama algebra $(2,2,1)\oplus \Aa_2$. 
\end{example}

\section{Applications of Syzygy Filtered Algebras}\label{section2}
In each subsection, we will define a homological dimension and prove Theorem \ref{bigthm1} in parts. We start with $\varphi$-dimension.

\subsection{Results about $\varphi$-dimension}
Let $A$ be an artin algebra.
Let $K_0$ be the abelian group generated by all symbols $\left[X\right]$, where $X$ is a finitely generated $A$-module, modulo the relations
\begin{enumerate}[label=\roman*)]
\item $\left[A_1\right]=\left[A_2\right]+\left[A_3\right]$ if $A_1\cong A_2\oplus A_3$
\item $\left[P\right]=0$ if $P$ is projective.
\end{enumerate}
Then $K_0$ is the free abelian group generated by the isomorphism classes of indecomposable finitely generated non-projective $A$-modules. For any finitely generated $A$-module $M$, let $L\left[M\right]:=\left[\Omega M\right]$ where $\Omega M$ is the first syzygy of $M$. Since $\Omega$ commutes with direct sums and takes projective modules to zero this gives a homomorphism $L : K_0 \rightarrow K_0$. For every finitely generated $A$-module $M$, let $\left\langle \add M\right\rangle$ denote the subgroup of $K_0$ generated by additive generator of $M$, i.e. all the indecomposable summands of $M$. $\langle\add M\rangle$ is a free abelian group because it is a subgroup of free abelian group $K_0$.
In other words, if $M=\oplus_{i=1}^m{M_i}^{n_i}$ then $\left\langle \add M\right\rangle=\left\langle \{[M_i]\}_{i=1}^m\right\rangle$ and $L^t(\left\langle \add M\right\rangle)= \left\langle \{[\Omega^t M_i]\}_{i=1}^m\right\rangle$.
\begin{definition}\label{defvarphimodule}
For a given module $M$ in $\modd A$, $\varphi$-dimension of $M$ is
\begin{align*}
\varphi(M):=\min\left\{t\,\vert\rank \left(L^t\langle \add M\rangle\right)=\rank\left(L^{t+j}\langle \add M\rangle\right)\text{ for}\, j\geq 1\right\}.
\end{align*}
\end{definition}

Notice that $\varphi(M)$ is finite for each module $M$, since the rank has to become stable at some step. For example, if projective dimension of $M$ is finite, then $\varphi(M)=p\dim(M)$.

\begin{definition}\label{defvarphidimension} If $A$ is an artin algebra, then $\varphi$-dimension of $A$ is
$$\varphi\dim A:=\sup\{\varphi(M)\ |\ M \in \modd \text{A}\}.$$ 
\end{definition}

We collect some statements about $\varphi$-dimension, proofs and other properties can be found in \cite{sen2018varphi}. 

\begin{remark}\label{rem findim smaller than fidim }
\begin{enumerate}[label=\roman*)]
\item  Let $A$ be an artin algebra of finite representation type. Let $\{M_1,\dots,M_m\}$ be a complete set of representatives of isomorphism classes of indecomposable $\Lambda$-modules. Then $\varphi \dim(A)=\varphi(\oplus_{i=1}^mM_i)$. 
\item\label{rem part2ccc} If global dimension of algebra $A$ is finite, then $\gldim A=\varphi\dim A$. Recall that $\findim A= \sup\{\pdim X\ |\ \pdim X< \infty, X\in\text{mod-}A\}$, which implies $\findim A\leq \varphi\dim A$.
\end{enumerate}
\end{remark}

We recall the result about small values of $\varphi$-dimension as stated in \cite{sen2018varphi}.

\begin{theorem}\label{Sen smallvalues} \cite[Thm 3.5.1]{sen2018varphi} Let $\Lambda$ be a cyclic Nakayama algebra of infinite global dimension. Then, 
\begin{enumerate}[label=\arabic*)]
\item $\varphi\dim\Lambda=0$ if and only if $\Lambda$ is selfinjective algebra.
\item $\varphi\dim\Lambda\neq 1$.
\item $\varphi\dim\Lambda=2$ if and only if each element of the base set $\cB(\Lambda)$ are periodic modules,  i.e. $\Delta_i\cong\Omega^j(\Delta_i)$ for any $1\leq i\leq r$.
\end{enumerate}
\end{theorem} 

Now, we want to translate it into the syzygy filtered algebra set up. By proposition \ref{PROP algebra is selfinjective iff equivalent to filtered algebra}, we can replace the first statement by "$\varphi\dim\Lambda=0$ if and only if  $\Lambda\cong\bm\varepsilon(\Lambda)$." For the third statement, we can apply the construction \ref{diagram categorical equivalence} and conclude that any simple $\bm\varepsilon(\Lambda)$-module  $S_i\cong\Hom_{\Lambda}(\cP,\Delta_i)$ is a periodic module. Therefore for any simple $\bm\varepsilon(\Lambda)$-module, $P(S)$ is minimal projective, otherwise $\rad P(S)$ would be a projective module which is not periodic module. By the characterization \ref{Rem characterization of selfinjective algebras}, we conclude that $\bm\varepsilon(\Lambda)$ is selfinjective. We restate the result in the new terminology.

\begin{theorem}\label{THM restated Sen smallvalues} Let $\Lambda$ be a cyclic Nakayama algebra of infinite global dimension. Then, 
\begin{enumerate}[label=\arabic*)]
\item $\varphi\dim\Lambda=0$ if and only if $\Lambda$ is selfinjective algebra if and only if $\Lambda\cong\bm\varepsilon(\Lambda)$.
\item\label{part fidim of cyclic cannot be one} $\varphi\dim\Lambda\neq 1$.
\item $\varphi\dim\Lambda=2$ if and only if $\bm{\varepsilon}(\Lambda)$ is selfinjective and $\Lambda$ is not.
\end{enumerate}

\end{theorem}

Here we give  some results on Nakayama algebras by using the syzygy filtered algebras. First we give the proof of the statement \ref{thmfidimreduction} of Theorem \ref{bigthm1}.

\begin{theorem}\label{thm reductionvar}
If $\Lambda$ is a cyclic Nakayama algebra, then $\varphi\dim\Lambda=\varphi\dim\bm{\varepsilon}(\Lambda)+2$ provided that $\varphi\dim\Lambda\geq 2$ and $\gldim\Lambda=\infty$.
\end{theorem}
\begin{proof}
For any indecomposable $\bm\varepsilon(\Lambda)$-module $X$, there exists $\cB(\Lambda)$-filtered $\Lambda$-module $\Delta X$ such that $\Hom_{\Lambda}(\cP,\Delta X)=X$ by the remark \ref{Rem equivalnece of resolutions}. By the categorical equivalence \ref{diagram categorical equivalence} we get
\begin{align*}
\varphi_{\bm\varepsilon(\Lambda)}(X)=\varphi_{\Lambda}(\Delta X).
\end{align*}
By proposition \ref{PROP X has filtration then it is second syzygy}, for any non-projective $\Delta X$ there exists $M\in\modd\Lambda$ such that $\Omega^2(M)\cong\Delta X$. Therefore 
\begin{align*}
\varphi(M)=2+\varphi(\Omega^2(M))=2+\varphi(\Delta X)=2+\varphi_{\bm\varepsilon(\Lambda)}(X).
\end{align*} 
We can take the supremum,
\begin{align*}
\varphi\dim\Lambda&=\sup\left\{\varphi(M)\vert\,M\in\modd\Lambda\right\}\\
&=\sup\left\{\varphi(\Omega^2(M))\vert\,M\in\modd\Lambda\right\}+2\\
&=\sup\left\{\varphi(\Delta M)\vert\,\Delta M\in Filt(\cB(\Lambda))\right\}+2\\
&=\sup\left\{\varphi(M)\vert\, M\in\modd\bm\varepsilon(\Lambda)\right\}+2\\
&=\varphi\dim\bm\varepsilon(\Lambda)+2
\end{align*}
which proves the statement.
\end{proof}
\begin{corollary}\label{cor fidim=2d + sth} If $\Lambda$ is a cyclic Nakayama algebra, then 
\begin{align*}
\varphi\dim\Lambda=2d+\varphi\dim\bm\varepsilon^d(\Lambda)
\end{align*}
for some $d\geq 1$, provided that $\gldim\Lambda=\infty$ and $\varphi\dim\bm\varepsilon^{d-1}(\Lambda)\geq 2$.
\end{corollary}
\begin{proof}
By proposition \ref{PROP filtered algebra is Nakayama} each syzygy filtered algebra is Nakayama algebra. Moreover, by Theorem \ref{thm reductionvar}, each of them have infinite global dimension hence cyclic. Therefore we can apply Theorem \ref{thm reductionvar}  iteratively to get reduction two in each step, i.e.
\begin{align*}
\varphi\dim\Lambda &=2+\varphi\dim\bm\varepsilon(\Lambda)\\
&=4+\varphi\dim\bm\varepsilon^2(\Lambda)\\
&\vdots\\
&=2(d-1)+\varphi\dim\bm\varepsilon^{d-1}(\Lambda)
\end{align*}
By the assumption $\varphi\dim\bm\varepsilon^{d-1}(\Lambda)\geq 2$, we can make one more reduction to get
\begin{align*}
\varphi\dim\Lambda=2d+\varphi\dim\bm\varepsilon^{d}(\Lambda).
\end{align*}
\end{proof}

Now, we give the proof of Theorem \ref{bigthm2} part \ref{thmlambdastable}.
\begin{theorem}\label{THM fidim becomes stable } If $\Lambda$ is a cyclic Nakayama algebra of infinite global dimension, then there exist $d$ such that $\varphi\dim\bm\varepsilon^{d-1}\Lambda=2$ and $\bm\varepsilon^d(\Lambda)$ is selfinjective.
\end{theorem}
\begin{proof}
By the corollary \ref{cor fidim=2d + sth}, we can apply the syzygy filtered algebra construction upto positive integer $d$ such that $0\leq \varphi\dim\bm\varepsilon^{d}(\Lambda)\leq 2$, and $\varphi\dim\bm\varepsilon^{d-1}(\Lambda)=2+\varphi\dim\bm\varepsilon^d(\Lambda)\geq 2$. Therefore 
\begin{align*}
0\leq \varphi\dim\Lambda-2d=\varphi\dim\bm\varepsilon^d(\Lambda)\leq 2.
\end{align*} 
We have there cases to examine.
\begin{enumerate}[label=Case \roman*)]
\item $\varphi\dim\Lambda=2d$, which makes $\varphi\dim\bm\varepsilon^d(\Lambda)=0$. By the characterization  \ref{Rem characterization of selfinjective algebras}, $\bm\varepsilon^{d}(\Lambda)$ is selfinjective algebra and $\varphi\dim\bm\varepsilon^{d-1}(\Lambda)=2$.
\item $\varphi\dim\Lambda=2d+2$, which makes $\varphi\dim\bm\varepsilon^d(\Lambda)=2$. We can apply Theorem \ref{thm reductionvar} to get $\varphi\dim\bm\varepsilon^{d+1}(\Lambda)=0$, so it is selfinjective algebra by the remark \ref{Rem characterization of selfinjective algebras}.
\item By the result \ref{THM restated Sen smallvalues} \ref{part fidim of cyclic cannot be one}, $\varphi$-dimension of cyclic Nakayama algebra cannot be one.
\end{enumerate}
\end{proof}

The following result was proved in \cite{sen2018varphi}. We state it in terms of the syzygy filtered algebras and give a different proof based on the syzygy filtered algebra construction.

\begin{theorem}\cite[Thm 5.1.1]{sen2018varphi} \label{Sen evenfidim}
Assume that global dimension of cyclic Nakayama algebra $\Lambda$ is infinite. Then $\varphi$-dimension of $\Lambda$ is even. Furthermore, the upper bound of $\varphi\dim\Lambda$ is $2r$, where $r=\vert\cB(\Lambda)\vert$.
\end{theorem}
\begin{proof}
By Theorem \ref{thm reductionvar}, in each step $\varphi$-dimension reduces by two. By Theorem \ref{THM fidim becomes stable }, there exist $d$ such that $\varphi\dim\bm\varepsilon^d(\Lambda)=2$, which makes $\varphi\dim\Lambda=2d+2$, an even number.\par
To prove the second part, observe that number of simple modules of $\bm{\varepsilon}^i(\Lambda)$ is given  by $\vert \cB(\bm{\varepsilon}^{i-1}(\Lambda))\vert$ which has the same cardinality with the irredundant system of relations defining $\bm\varepsilon^{i-1}(\Lambda)$. If we set $\varphi\dim\Lambda$ to $2d$, then by Theorem \ref{THM fidim becomes stable } $\bm\varepsilon^{d}(\Lambda)$ is selfinjective and $\varphi\dim\bm\varepsilon^{d-1}(\Lambda)=2$. Let $r_i$ denote the rank of $\bm\varepsilon^{i+1}(\Lambda)$ for $i\geq 1$. Therefore $\rank\bm\varepsilon^{d}(\Lambda)=r_{d-1}$ and each rank satisfies $r_{i+1}\leq r_i-1$. If we add the inequalities we get $r_{d-1}\leq r-(d-1)$. The minimal possible value of $r_{d-1}$ is one, therefore $d\leq r$, which implies $\varphi\dim\Lambda=2d\leq 2r$. In particular, $r\leq N-1$, in terms of the rank of $\Lambda$, the upper bound is $2N-2$.
\end{proof}


\subsection{Results about finitistic dimension}

First we state and prove results regarding possible small values of the finitistic dimension and $\varphi$-dimension. Then we state the reduction method for the finitistic dimension.
We recall definition of the finitistic dimension of artin algebra A, \begin{align*}
\findim A:=\sup\left\{\pdim M\vert\,M\,\text{is an} \,A\text{-module with}\,\,\pdim M<\infty \right\}.
\end{align*}

We start with the following observation.

\begin{lemma}\label{LEM fidim 2 implies findim is 1 or 2 }
If $\varphi\dim\Lambda=2$, then $\varphi\dim\Lambda-\findim\Lambda\leq 1$.
\end{lemma}
\begin{proof}
Since $\varphi\dim\Lambda=2$, by the remark \ref{Rem characterization of selfinjective algebras}  $\Lambda$ is not selfinjective and there exists a projective module $P$ which is not minimal. Therefore, the exact sequence
\begin{align*}
0\rightarrow \rad P\rightarrow P\rightarrow \topp P\rightarrow 0
\end{align*}
implies that $\pdim\topp P\geq 1$ and as a result $\findim\Lambda\geq 1$.
On the other hand, $\findim\Lambda\leq\varphi\dim\Lambda$ which is stated in the remark \ref{rem findim smaller than fidim } part \ref{rem part2ccc}, we conclude that $\varphi\dim\Lambda-\findim\Lambda\leq 1$.
\end{proof}

Now we give the proof of the statement \ref{thmfindimreduction} of Theorem \ref{bigthm1}.

\begin{proposition} \label{PROP findim reduction} If $\Lambda$ is a cyclic Nakayama algebra, then
$\findim \Lambda=\findim \bm{\varepsilon}(\Lambda)+2 $, provided that $\findim\Lambda\geq 2$ and $\gldim\Lambda=\infty$.
\end{proposition}
\begin{proof} 
Both of $\Lambda$ and the syzygy filtered algebra are finite representation type, so there exists  $\bm\varepsilon(\Lambda)$-module $X$ such that $\pdim X=\findim\bm\varepsilon(\Lambda)$. If $\Delta X\in\modd\Lambda$ is the corresponding module i.e. $\Hom_{\Lambda}(\cP,\Delta X)\cong X$, then $\pdim_{\bm\varepsilon(\Lambda)}X=\pdim_{\Lambda}\Delta X$ by proposition \ref{PROP carrying resolution on lambda to epsilon}. Therefore we get
\begin{align}
\findim\Lambda &=\sup\left\{\pdim M\vert\,M\in\modd\Lambda,\,\pdim M<\infty\right\}\label{al 1}\\
&=\sup\left\{\pdim\Omega^2(M)\vert\,M\in\modd\Lambda,\,\pdim M<\infty\right\}+2.\label{al 2}
\end{align}
Any $\Omega^2(M)$ has $\cB(\Lambda)$-filtration by proposition \ref{PROP X has filtration then it is second syzygy} which implies
\begin{align*}
\findim\Lambda&=2+\sup\left\{\pdim \Delta X\vert\,\Delta X\in Filt(\cB(\Lambda)),\,\pdim \Delta X<\infty\right\}\\
&=2+\sup\left\{\pdim X\vert\,X\in\modd\bm\varepsilon(\Lambda),\,\pdim X<\infty\right\}\\
&=2+\findim\bm\varepsilon(\Lambda)
\end{align*}
where $\Hom_{\Lambda}(\cP,\Delta X)\cong X$. 

We want to analyze the case of $\findim\Lambda=2$. The reduction above implies that $\findim\bm\varepsilon(\Lambda)=0$. This forces that all the simple $\bm\varepsilon(\Lambda)$-modules have periodic resolution i.e $\Omega^i(S)\cong\Omega^{i+j}(S)$ for some $i,j$ where $S$ is simple module. This implies that the radical of any projective module $P$
\begin{align*}
0\rightarrow \rad P(S)\rightarrow P(S)\rightarrow S\rightarrow 0
\end{align*}
is non-projective. Therefore, every projective $\bm\varepsilon(\Lambda)$-module is minimal which is equivalent to $\bm\varepsilon(\Lambda)$ is selfinjective by the remark \ref{Rem characterization of selfinjective algebras}.
\end{proof}

We placed the condition on the lower bound of the  finitistic dimension, otherwise we could not pass from \ref{al 1} to \ref{al 2}. For instance, if $\alpha_2\alpha_1=0, \alpha_1\alpha_3\alpha_2=0$ are relations of rank $3$ algebra $\Lambda$, then   $\findim\Lambda=1$ and $\findim\bm\varepsilon(\Lambda)=0$, so the reduction is not two.  

\begin{corollary}\label{COR fin dim reduction} If $\Lambda$ is a cyclic Nakayama algebra, then
\begin{align*}
\findim\Lambda=2d+\findim\bm\varepsilon^{d}(\Lambda)
\end{align*}
for any $d\geq 1$, provided that $\gldim\Lambda=\infty$ and $\findim\bm\varepsilon^{d-1}(\Lambda)\geq 2$.
\end{corollary}

\begin{proof}
By proposition \ref{PROP filtered algebra is Nakayama}, the syzygy filtered algebra is cyclic Nakayama algebra and of infinite global dimension. Therefore we can apply propositions \ref{PROP filtered algebra is Nakayama} and \ref{PROP findim reduction}  iteratively to get reduction two in each step, i.e.
\begin{align*}
\findim\Lambda &=2+\findim\bm\varepsilon(\Lambda)\\
&=4+\findim\bm\varepsilon^2(\Lambda)\\
&\hspace{2cm}\vdots\\
&=2(d-1)+\findim\bm\varepsilon^{d-1}(\Lambda)
\end{align*}
By the assumption $\findim\bm\varepsilon^{d-1}(\Lambda)\geq 2$, we can make one more reduction to get
\begin{align*}
\findim\Lambda=2d+\findim\bm\varepsilon^{d}(\Lambda).
\end{align*}
\end{proof}
\begin{corollary}\label{COR there is d such that fin dim d is between 1 and 2} If $\Lambda$ is cyclic Nakayama algebra of infinite global dimension, then there exists $d$ such that 
\begin{align*}
1\leq \findim\bm\varepsilon^{d}(\Lambda)\leq 2.
\end{align*}
\end{corollary}
\begin{proof}
By Theorem \ref{THM fidim becomes stable }, there exists $d$ such that $\varphi\dim\bm\varepsilon^d(\Lambda)=2$ and $\bm\varepsilon^{d+1}(\Lambda)$ is selfinjective. If we apply lemma \ref{LEM fidim 2 implies findim is 1 or 2 } to  $\bm\varepsilon^d(\Lambda)$, then\begin{align*}
1\leq \findim\bm\varepsilon^{d}(\Lambda)\leq 2.
\end{align*}
\end{proof}

Now we can prove the statement \ref{thmdifferencefiandfin} of Theorem \ref{bigthm1}.
\begin{theorem}\label{Thm difference between dims}
If $\Lambda$ is a cyclic Nakayama algebra then $\varphi\dim\Lambda-\findim\Lambda\leq 1$
\end{theorem}

\begin{proof} If global dimension of $\Lambda$ is finite, statement is true, since $gl\dim\Lambda=\varphi\dim\Lambda=\findim\Lambda$. So, we assume that global dimension is infinite.\\

By Theorem \ref{THM fidim becomes stable }, there exists $d$ such that $\bm{\varepsilon}^d(\Lambda)$ is a selfinjective algebra but $\bm{\varepsilon}^{d-1}(\Lambda)$ is not. By the reduction \ref{thm reductionvar} we get $\varphi\dim\Lambda=2d$, $\varphi\dim\bm{\varepsilon}^d(\Lambda)=0$ and $\varphi\dim\bm{\varepsilon}^{d-1}(\Lambda)=2$. By the corollary \ref{COR there is d such that fin dim d is between 1 and 2} we get 
\begin{align}\label{inequality}
1\leq \findim\bm{\varepsilon}^{d-1}(\Lambda)\leq 2.
\end{align} 
On the other hand we have $2+\findim\bm{\varepsilon}^{k}(\Lambda)=\findim\bm{\varepsilon}^{k-1}(\Lambda)$  for any $1\leq k\leq d-1$ because of the corollary \ref{COR fin dim reduction}.
This implies
\begin{align*}
1+2d-2\leq \findim\Lambda\leq 2+2d-2\Rightarrow
\varphi\dim\Lambda-1\leq \findim\Lambda\leq \varphi\dim\Lambda.
\end{align*}
We obtain the desired inequality $\varphi\dim\Lambda-\findim\Lambda\leq 1$.
\end{proof}

\subsection{Results about Gorenstein dimension} 
We begin with definition of Gorenstein dimension. By using duality and representation type of Nakayama algebras, instead of injective dimension of $\Lambda$, we prefer to study projective dimensions of injective modules. Then we state and prove results regarding Gorenstein homological properties. 
\begin{definition}\label{defGorensteindimension} If projective resolution of injective modules are finite, algebra is called \emph{Gorenstein}. Supremum of projective dimensions of injective modules is called \emph{Gorenstein dimension}, and denoted by $\gordim$. The formulation is 
\begin{align*}
\gordim\Lambda=\sup\left\{\pdim I\vert\, I\in\modd\Lambda\,\text{is injective}\right\}<\infty.
\end{align*}
\end{definition}

First, we need possible small values of Gorenstein dimension. Then we show the reduction method.
\begin{lemma}\label{LEM gordim cannot be one}
If $\Lambda$ is cyclic Nakayama algebra which is Gorenstein, then Gorenstein dimension cannot be one.
\end{lemma}
\begin{proof}
If $\Lambda$ is selfinjective, by definition, $\gordim\Lambda=0$. If $\Lambda$ is not selfinjective, there exists a projective-injective module $PI$ such that its quotient is injective $I$ and  the sequence 
\begin{align*}
0\rightarrow \soc PI\rightarrow PI\rightarrow I\rightarrow 0
\end{align*}
is exact. Since $\Lambda$ is cyclic, there is no simple projective module. Therefore $\pdim I=1+\pdim\soc PI\geq 2$ which implies $\gordim\Lambda\neq 1$. 
\end{proof}

\begin{lemma}\label{LEM the submodule of delta gives injective} Let $\Delta_i$ be an element of $\cB(\Lambda)$ such that it is submodule of an indecomposable projective-injective $\Lambda$-module $PI$ and it has a proper submodule $X$, i.e.
\begin{align*}
X\hookrightarrow\Delta_i\hookrightarrow PI
\end{align*}
Then the quotient $\faktor{PI}{X}$ is an injective $\Lambda$-module.
\end{lemma}

\begin{proof}
$PI$ is projective-injective module, therefore $P=P(\tau^{-1}\topp PI)$ is minimal projective by the corollary \ref{COR P_i is minimal iff P i+1 is projective injective}. By proposition \ref{PROP x in base set iff x is omegaone of radical}, $\Delta_i$ is the first syzygy of $\rad P$. The difference of the lengths of $PI$ and $P$ is
\begin{align*}
\ell(PI)-\ell(P)=\ell(\rad P)+\ell(\Delta_i)-\left(\ell(\rad P)+1\right)=\ell(\Delta)-1.
\end{align*}
Since $X$ is proper submodule of $\Delta_i$, $\ell(\Delta_i)\geq 2$. By proposition \ref{PROP ci<ci+1 implies nonzero defect}, there are $\ell(\Delta_i)-1\geq 1$ proper injective quotients of $PI$ and it is clear that the quotient $\faktor{PI}{X}$ is one of them.
\end{proof}

\begin{proposition}\label{PROP every injective eps mod is second syzygy of injective lambda} If $I$ is an injective but non-projective $\bm\varepsilon(\Lambda)$-module, then there exists an injective but non-projective $\Lambda$-module $I_{\Lambda}$ such that $\Hom_{\Lambda}\left(\cP,\Omega^2(I_{\Lambda})\right)\cong I$.
\end{proposition}

\begin{proof}
$I$ is an injective $\bm\varepsilon(\Lambda)$-module, by the categorical equivalence \ref{diagram categorical equivalence} and the remark \ref{Rem equivalnece of resolutions}, there exists $\Lambda$-module $\Delta I$ such that $\Hom_{\Lambda}(\cP,\Delta I)\cong I$. We will show that $\Delta I$ is isomorphic to a second syzygy of an injective $\Lambda$-module $I_{\Lambda}$.\\
\begin{claim}\label{Cla 1 318   } There exist a projective $\Lambda$-module $P$ such that  $\Delta I$ is submodule of $P$.
\end{claim}
\begin{proof}
$\Delta I$ has $\cB(\Lambda)$-filtration, therefore $\soc \Delta I\in\cS(\Lambda)$. By the categorical equivalence \ref{diagram categorical equivalence} there is a projective module $P$ such that $\soc \Delta I\cong \soc P$ (see remark \ref{Rem equivalnece of resolutions}). By the  uniseriality lemma \ref{uniserialitylemma}, either $P\subseteq \Delta I$ or $\Delta I\subset P$. The former is impossible because it makes $\Delta I$ a projective module by \ref{LEM projectivecannotbequotientlemma} and $\Hom_{\Lambda}(\cP,\Delta I)$ would be projective $\bm\varepsilon(\Lambda)$-module which violates the assumption of proposition. Therefore $\Delta I$ is a proper submodule of projective $P$.
\end{proof}
 
\begin{claim}\label{Cla 2 318   } $P$ is not $\cB(\Lambda)$-filtered module.
\end{claim}
\begin{proof}
Assume that $P$ has $\cB(\Lambda)$-filtration. Then $I\cong\Hom_{\Lambda}(\cP,\Delta I)$ becomes proper submodule of projective $\bm\varepsilon(\Lambda)$-module $\Hom_{\Lambda}(\cP,P)$ which violates the assumption on the injectivity of $I$ in $\modd\bm\varepsilon(\Lambda)$.
\end{proof}
Let's denote the quotient $\faktor{P}{\Delta I}$ by $Q$, i.e.
\begin{align*}
0\rightarrow \Delta I\rightarrow P\rightarrow Q=\faktor{P}{\Delta I}\rightarrow 0
\end{align*}
\begin{claim}\label{Cla 3 318   }
$Q$ is a proper submodule of an element of the base set $\cB(\Lambda)$.
\end{claim}
\begin{proof}
Since $\soc Q\cong\tau^{-1}\topp \Delta I\in\cS(\Lambda)$, there exist $\Delta_i\in\cB(\Lambda)$ such that $\soc\Delta_i\cong\soc Q$. By the uniseriality lemma we have three cases to examine.
\begin{enumerate}[label=Case \arabic*)]
\item $\Delta_i\cong Q$ is not possible, otherwise we would conclude that $P$ has $\cB(\Lambda)$-filtration which violates the claim \ref{Cla 2 318   }.
\item $\Delta_i\subset Q$ is not possible. Assume to the contrary that $\Delta_i$ is a proper submodule of $Q$. Then, the extension $\Delta I'=\begin{vmatrix}
\Delta_i\\\Delta I
\end{vmatrix}$ of $\Delta_i$ by $\Delta I$ becomes $\cB(\Lambda)$-filtered indecomposable proper submodule of $P$. If we apply the evaluation functor $\Hom_{\Lambda}(\cP,-)$ to $\Delta I$ and $\Delta I'$ we get $\Hom_{\Lambda}(\cP,\Delta I)\hookrightarrow\Hom_{\Lambda}\left(\cP,I'\Delta\right)$ in $\modd\bm\varepsilon(\Lambda)$. This makes $I$ a proper submodule of $I'$. However, an injective module cannot be a proper submodule of any module by definition.
\item Hence $Q$ is a proper submodule of $\Delta_i$.
\end{enumerate}
\end{proof}
Because $Q$ is proper submodule of $\Delta_i$ and $\Delta_i$ is a submodule of projective-injective module $P'$ by proposition \ref{PROP properties of the basesetproperties}, we get the exact sequence
\begin{align*}
0\rightarrow Q\rightarrow P'\rightarrow \faktor{P'}{Q}\rightarrow 0.
\end{align*}

If we apply proposition \ref{PROP every injective eps mod is second syzygy of injective lambda} to $Q\hookrightarrow \Delta_i\hookrightarrow P'$, then the quotient $\faktor{P'}{Q}$ is an injective $\Lambda$-module.  We can choose $I_{\Lambda}$ as
\begin{align*}
I_{\Lambda}\cong\faktor{P'}{\left(\faktor{P}{\Delta I}\right)}
\end{align*}
which satisfies $\Omega^2(I_{\Lambda})\cong \Delta I$ and $\Hom_{\Lambda}(\cP,\Delta I)\cong I$ in $\modd\bm\varepsilon(\Lambda)$.

\end{proof}

This result is useful, in the sense that for any injective but non-projective $\bm\varepsilon(\Lambda)$-module $I$, there exists at least one injective $\Lambda$-module $I_{\Lambda}$ such that $\Hom_{\Lambda}(\cP,\Omega^2(I_{\Lambda}))\cong I$. However, the converse of this statement is not true in general. First of all, the second syzygy of injective $\Lambda$-modules can be either projective or periodic modules or even trivial. For example, if $\Lambda$ is given by the Kupisch series $(4,3,4,3)$, then $\bm\varepsilon(\Lambda)$ is selfinjective algebra with the Kupisch series $(2,2)$ and its simple modules are $\Hom_{\Lambda}(\cP,\Omega^2(I_1))$ and $\Hom_{\Lambda}(\cP,\Omega^2(I_3))$. For the algebra given by the Kupisch series $(5,4,3,2,2)$, we get  $\pdim I_2=\pdim I_3=1$. Second observation is that $I_{\Lambda}$ might not be the unique module verifying proposition \ref{PROP every injective eps mod is second syzygy of injective lambda}. For example, algebra given by the Kupisch series $(5,4,4,4,3)$, we have $\Omega^2(I_3)\cong\Omega^2(I_4)$. Nevertheless, the corollary below follows from propositions \ref{PROP every injective eps mod is second syzygy of injective lambda} and \ref{PROP carrying resolution on lambda to epsilon}.

\begin{corollary} If $I$ is an injective non-projective $\bm\varepsilon(\Lambda)$-module of finite projective dimension, then $\pdim _{\Lambda}I_{\Lambda}=\pdim_{\bm\varepsilon(\Lambda)} I+2$ where $I_{\Lambda}$ is an injective module satisfying $\Hom_{\Lambda}(\cP,\Omega^2(I_{\Lambda}))\cong I$.
\end{corollary}

We give the proof of the statement \ref{thmgordireduction} of Theorem \ref{bigthm1}.
\begin{theorem}\label{THM gorenstein reduction} If $\Lambda$ is a cyclic Nakayama algebra which is Gorenstein then $\gordim\Lambda=\gordim\bm\varepsilon(\Lambda)+2$ provided that $\gordim\Lambda\geq 2$, and $\bm\varepsilon(\Lambda)$ is also Gorenstein.
\end{theorem}
\begin{proof}
If $\Lambda$ is Gorenstein, then
\begin{align}
\gordim\Lambda &=\sup\left\{\pdim I\vert\,I\in\modd\Lambda\,\text{is injective} \right\}\nonumber\\
&=2+\sup\left\{\pdim\Omega^2(I)\vert\,I\in\modd\Lambda\,\text{is injective} \right\}\nonumber\\
&=2+\sup\left\{\pdim \Delta I\vert\,\Delta I\cong\Omega^2(I)\in Filt(\cB(\Lambda)),I\,\text{is injective} \right\}\label{al 4}\\
&=2+\sup\left\{\pdim_{\bm\varepsilon(\Lambda)} I'\vert\, I'\in\modd\bm\varepsilon(\Lambda),\, \Hom_{\Lambda}(\cP,\Delta I)\cong I'\, \text{is injective} \right\}\label{al 3}\\
&=2+\gordim\bm\varepsilon(\Lambda).\nonumber
\end{align}
where we used proposition \ref{PROP every injective eps mod is second syzygy of injective lambda} to pass from \ref{al 4} to \ref{al 3}. 
In particular, the finiteness of $\gordim\Lambda$ implies $\gordim\bm\varepsilon(\Lambda)$ is finite, so $\bm\varepsilon(\Lambda)$ is Gorenstein.
\end{proof}

\begin{proposition}\label{PROP small gordim} If $\Lambda$ is Gorenstein and $\varphi\dim\Lambda=2$, then $\gordim\Lambda=2$.
\end{proposition}
\begin{proof}
Since $\Lambda$ is Gorenstein, projective dimension of any injective module is finite. Therefore $\gordim\Lambda\leq\findim\Lambda$, and by Theorem \ref{Thm difference between dims}, we get $\gordim\Lambda\leq\varphi\dim\Lambda=2$.\par
Because $\varphi\dim\Lambda=2$, $\Lambda$ is not selfinjective, hence $\gordim\Lambda\neq 0$. We need to show that one is not in the range. Assume to the contrary that let $\gordim\Lambda=1$.  We can choose an injective module such that it is the quotient of a projective-injective module, i.e. $\Omega^1(I)$ is simple module. Since $\Lambda$ is cyclic, there is no simple projective module, we derive a contradiction. Hence $\gordim\Lambda\neq 1$. We obtain $\gordim\Lambda=\varphi\dim\Lambda=2$.
\end{proof}
Based on the reduction stated in \ref{THM gorenstein reduction}, we give another proof of equality about $\varphi$-dimension and Gorenstein dimension \cite{ralf}.
\begin{theorem}\label{gordimisfidim}
Let $\Lambda$ be a cyclic Nakayama algebra whic is Gorenstein. Then, $\varphi$-dimension and Gorenstein dimension of $\Lambda$ match.
\end{theorem}

\begin{proof}
If global dimension of $\Lambda$ is finite, then it is equal to both $\varphi\dim\Lambda$ and $\gordim\Lambda$, which is trivial case. Also, if $\Lambda$ is selfinjective then $\varphi\dim\Lambda=\gordim\Lambda=0$. So, we analyze the case of infinite global dimension with nonzero $\varphi$-dimension.\par

By Theorem \ref{THM fidim becomes stable }, there exists $d$ such that $\bm{\varepsilon}^d(\Lambda)$ is selfinjective but $\bm{\varepsilon}^{d-1}(\Lambda)$ is not. 
If we combine this observation with proposition \ref{PROP small gordim} we get
\begin{align}
\gordim\bm\varepsilon^{d-1}(\Lambda)=\varphi\dim\bm\varepsilon^{d-1}(\Lambda)=2.\nonumber
\end{align}
By the reductions \ref{thm reductionvar} and \ref{THM gorenstein reduction}, we conclude that
\begin{align*}
\gordim\bm\varepsilon^{d-1}(\Lambda)=\varphi\dim\bm\varepsilon^{d-1}(\Lambda)=2 &\iff \gordim\bm\varepsilon^{d-2}(\Lambda)=\varphi\dim\bm\varepsilon^{d-2}(\Lambda)=4 \\
&\hspace{4cm}\vdots\\
&\iff \gordim\bm\varepsilon(\Lambda)=\varphi\dim\bm\varepsilon(\Lambda)=2d-2 \\
&\iff\gordim\Lambda=\varphi\dim\Lambda=2d.
\end{align*}
\end{proof}

In Theorem \ref{gordimisfidim}, we showed the equality of $\varphi$-dimension and Gorenstein dimension of cyclic Nakayama algebras. We combine it with the result \ref{Sen evenfidim} to deduce possible values of Gorenstein dimension.

\begin{corollary}
If $\Lambda$ is Gorenstein with infinite global dimension, then $\gordim\Lambda$ is even. 
\end{corollary}

Gorenstein properties of Nakayama algebras were studied by Ringel in \cite{ringel2013gorenstein}. In proposition 6 of the same work, it is stated that Gorenstein dimension is an even number under some conditions on the lengths of modules. The corollary above extends this result.

\subsection{Results about dominant dimension}
We study dominant dimension in terms of the syzygy filtered algebras.

\begin{definition}\cite{aus}\label{defdomdim} Dominant dimension of algebra $A$ is:
\begin{align*}
\domdim A=\sup\left\{m\vert\,I_i\, \text{is projective for }i=0,1,\ldots m\right\}+1
\end{align*}
where $0\rightarrow\,_AA\rightarrow I_0\rightarrow I_1\ldots$ is the minimal injective resolution of $A$. 
Dominant dimension of semisimple algebra is set to be infinity. We can define dominant dimension of any module $M$ by
 \begin{align*}
\domdim M=\sup\left\{m\vert\,I_i\, \text{is projective for }i=0,1,\ldots m\right\}+1
\end{align*}
where $0\rightarrow _A M\rightarrow I_0\rightarrow I_1\ldots$ is the minimal injective resolution of $M$. If we restrict $M$ to projective modules we get the following characterization of dominant dimension 
\begin{align}\label{gercek tanim2}
\domdim A=\sup\left\{\domdim P\,\vert\, P \text{ is projective non-injective } A\, \text{module}\right\}.
\end{align}
It is convenient for us to study projective dimensions of injective modules, because for a Nakayama algebra $\Lambda$, $\domdim\Lambda=\domdim\Lambda^{op}$ holds. Therefore the statement \ref{gercek tanim2}
 is equivalent to
 \begin{align}\label{defdomdim gercek tanim}
\domdim A=\sup\left\{\domdim I\,\vert\, I \text{ is injective non-projective } A\, \text{module}\right\}.
\end{align}
 We always use the characterization of dominant dimension \ref{defdomdim gercek tanim} in the proofs. 
\end{definition}

We begin with a lemma which shows that the first syzygy of an injective module has to be a submodule of an element of the base set.

\begin{lemma}\label{LEM injective first szyg is submodule of delta} If $I$ is an injective non-projective $\Lambda$-module, then $\Omega^1(I)$ is proper a submodule of an element  $\Delta_i$ of the base set $\cB(\Lambda)$.
\end{lemma}

\begin{proof}
Let $P_{i}$ be the injective envelope of $I$. The short exact sequence
\begin{align*}
0\rightarrow \Omega^1(I)\rightarrow P_i\rightarrow I\rightarrow 0
\end{align*}
implies that 
\begin{align}\label{eqSep29 1}
\ell(P_i)-\ell(I)=\ell(\Omega^1(I))\neq 0,
\end{align}
because $I$ is not projective. \par
On the other hand, $P_{i-1}$ is minimal injective by the corollary \ref{COR P_i is minimal iff P i+1 is projective injective} which implies $\ell(P_{i-1})\leq\ell(P_{i})$. Moreover by the result \ref{PROP x is an element of the base set iff it is second syzygy}, $\Omega^2(\topp P_{i-1})\cong \Delta_i$ is an element of the base set $\cB(\Lambda)$ and it satisfies 
\begin{align}\label{eqSep29 2}
\ell(\Delta_i)=\ell(P_i)-\ell(\rad P_{i-1}).
\end{align} 
Notice that $\ell(\Delta_i)\neq 1$, otherwise by the result \ref{PROP ci=ci+1 implies defect zero, minimal projective} $P_i$ would not have a proper injective quotient. 
If we subtract \ref{eqSep29 1} from \ref{eqSep29 2}, we get
\begin{align}
\ell(\Delta_i)-\ell(\Omega^1(I))=\ell(I)-\ell(\rad P_{i-1}).
\end{align}
Because $\topp P_{i-1}\cong \topp I$ we can apply the corollary \ref{uniserialitylemma injectiveCOR} and find the only possible case.
\begin{enumerate}[label=Case \arabic*)]
\item $\rad P_{i-1}\cong I$ is not possible, it makes $I$ a submodule of $P_{i-1}$ which violates dual of lemma \ref{LEM projectivecannotbequotientlemma}.
\item $I$ cannot be quotient of $\rad P_{i-1}$, otherwise $I$ would become subquotient of $P_{i-1}$ and injective modules cannot be subquotients.
 \end{enumerate}
 As a result, $\rad P_{i-1}$ is the quotient of $I$ which implies $\ell(\Delta_i)>\ell(\Omega^1(I))$. They share the isomorphic socle, by the uniseriality lemma $\Omega^1(I)$ is a submodule of $\Delta_i$.
\end{proof}

\begin{lemma}\label{LEM filtered proj inj implies it is filtered in epsilon}
If $PI$ is an indecomposable projective-injective $\Lambda$-module which is filtered by $\cB(\Lambda)$, then the module $\Hom_{\Lambda}(\cP,PI)$ is projective-injective $\bm\varepsilon(\Lambda)$-module.
\end{lemma}
\begin{proof}
Since $PI$ is projective module and $PI\in Filt(\cB(\Lambda))$, $\Hom_{\Lambda}(\cP,PI)$ is projective $\bm\varepsilon(\Lambda)$-module.  Therefore it is enough to show that $\Hom_{\Lambda}(\cP,PI)\in\modd\bm\varepsilon(\Lambda)$ is injective. To emphasize its filtration, let's denote $PI$ by $\Delta X$ and the corresponding module by $X$, i.e. $\Hom_{\Lambda}(\cP,PI)=\Hom_{\Lambda}(\cP,\Delta X)\cong X$.\\
Suppose that $X$ is not injective, so there has to be an injective envelope $I(X)$ which gives the nonsplit exact sequence
\begin{align}\label{al 5}
0\rightarrow X\rightarrow I(X)\rightarrow Q=\faktor{I(X)}{X}\rightarrow 0
\end{align}
in $\modd\bm\varepsilon(\Lambda)$ with nontrivial $Q$. By the categorical equivalence, there are $\cB(\Lambda)$-filtered modules $\Delta X$, $\Delta I$, $\Delta Q$ such that $\Hom_{\Lambda}(\cP,\Delta X)\cong X$, $\Hom_{\Lambda}(\cP,\Delta I)\cong I$, $\Hom_{\Lambda}(\cP,\Delta Q)\cong Q$ and the sequence
\begin{align}\label{al 6}
0\rightarrow \Delta X\rightarrow\Delta I\rightarrow\Delta Q\rightarrow 0
\end{align}
is nonsplit and exact in $\modd\Lambda$. However $\Delta X$ is the indecomposable projective-injective module $PI$ and it cannot be a proper  submodule of another module. Therefore $\Delta Q$ and $Q$ have to be trivial in \ref{al 6} and \ref{al 5} respectively, which contradicts to assumption that $Q$ is not trivial. As a result, $X=\Hom_{\Lambda}(\cP,PI)$ is injective $\bm\varepsilon(\Lambda)$-module with the corresponding module $PI$.
\end{proof}

The converse of the above statement is not true in general, a projective-injective $\bm\varepsilon(\Lambda)$-module has to have the corresponding $\Lambda$-module which is a projective but not necessarily a projective-injective.

The converse of proposition \ref{PROP every injective eps mod is second syzygy of injective lambda} is not true in general. However, under some conditions it can be valid. The key observation is stated below.
\begin{proposition}\label{PROP domdim geq 3 implies bijection} Let $\Lambda$ be a cyclic Nakayama algebra such that $\domdim\Lambda\geq 3$. Then, there is a bijection between indecomposable injective $\Lambda$-modules and indecomposable injective $\bm\varepsilon(\Lambda)$-modules via the functor $\Hom_{\Lambda}(\cP,\Omega^2(-))$.
\end{proposition}

\begin{proof} 
Consider the projective resolution of an injective module $I$
\begin{align}\label{resolution domdim}
\cdots P_3\rightarrow P_2\rightarrow P_1\rightarrow P_0\rightarrow I
\end{align}
in $\modd\Lambda$. Since $\domdim\Lambda\geq 3$, we deduce that $P_0, P_1, P_2$ are projective-injective modules. Also, $\Omega^3(I)$ is nontrivial module, otherwise $P_2\cong\Omega^2(I)$ implies that a syzygy module is projective-injective which is not possible by lemma \ref{LEM projectivecannotbequotientlemma}. By proposition \ref{PROP all higher syzygies are filtered}, we conclude that  $P_2,\,\Omega^2(I)\in Filt(\cB(\Lambda))$.\\

Assume that $\Hom_{\Lambda}(\cP,\Omega^2(I))$ is not an injective module in $\modd\bm\varepsilon(\Lambda)$. $\Hom_{\Lambda}(\cP,P_2)$ is projective-injective $\bm\varepsilon(\Lambda)$-module by lemma \ref{LEM filtered proj inj implies it is filtered in epsilon}, therefore $\Hom_{\Lambda}(\cP,\Omega^2(I))$ is a quotient of $\Hom_{\Lambda}(\cP,P_2)$. We divide the proof into steps.

\begin{claim}\label{claim 331} For any injective non-projective $\Lambda$-module $I$, the projective $\syz$-module $\Hom_{\Lambda}(\cP,P_2)$ which covers $\Hom_{\Lambda}(\cP,\Omega^2(I))$ has injective quotients in $\modd\syz$.
\end{claim}

\begin{proof} Assume to the contrary that $PI=\ho{P_2}$ does not have any injective quotient in $\modd\bm\varepsilon(\Lambda)$. In particular $\ho{\Omega^2(I)}$ is not injective module, therefore there exists a projective module $\ho{\Delta P}$ such that $\ho{\Omega^2(I)}\subset\ho{\Delta P}$ by the uniseriality lemma \ref{uniserialitylemma}. By the construction \ref{diagram categorical equivalence} and the remark \ref{Rem equivalnece of resolutions}, $\Delta P$ is $\cB(\Lambda)$-filtered projective $\Lambda$-module such that $\Omega^2(I)\subset\Delta P$ in $\modd\Lambda$. If we combine this observation with the resolution \ref{resolution domdim}, we get the short exact sequences
\begin{gather}
\begin{gathered}\label{res domdim 2}
0\rightarrow \Omega^2(I)\rightarrow \Delta P\rightarrow Q=\faktor{\Delta P}{\Omega^2(I)}\rightarrow 0\\
0\rightarrow \Omega^2(I)\rightarrow P_1\rightarrow\Omega^1(I)\rightarrow 0
\end{gathered}
\end{gather}
where the quotient $Q$ has $\cB(\Lambda)$-filtration by lemma \ref{topsocConsecutive}.

 Notice that $P_1$ and $\Delta P$ have isomorphic socles. By the uniseriality lemma \ref{uniserialitylemma} we have either $\Delta P\subseteq P_1$ or $P_1\subset \Delta P$. The latter is not possible, because $P_1$ is projective-injective module. Therefore either $\Delta P\cong P_1$ or $\Delta P$ is a proper submodule of $P_1$. This implies that either $Q\cong \Omega^1(I)$ or $Q\subset \Omega^1(I)$ respectively which is deduced from the exact sequences \ref{res domdim 2}. However this is not possible because of lemma \ref{LEM injective first szyg is submodule of delta}, $\Omega^1(I)$ cannot have $\cB(\Lambda)$-filtered submodules.
\end{proof}

\begin{claim}\label{claim 332} For any injective non-projective $\Lambda$-module $I$, $\ho{\Omega^2(I)}$ is injective $\syz$-module.
\end{claim}
\begin{proof}
 Assume that $\ho{\Omega^2(I)}$ is not an injective $\syz$-module. Also, it is not a projective module by the assumption, i.e. $\Omega^3(I)$ is nontrivial. Let $I'$ be its injective envelope in $\modd\syz$ and $\Delta I'$ be the corresponding module, i.e. $\ho{\Delta I'}\cong I'$. Their existence follows from the fact that each of these modules are $\cB(\Lambda)$-filtered. 
 
The exact sequences
\begin{gather}
\begin{gathered}\label{res domdim 3}
0\rightarrow \Omega^2(I)\rightarrow \Delta I'\rightarrow Q=\faktor{\Delta I'}{\Omega^2(I)}\rightarrow 0\\
0\rightarrow \Omega^2(I)\rightarrow P_1\rightarrow\Omega^1(I)\rightarrow 0
\end{gathered}
\end{gather}
imply that $P_1$ and $\Delta I'$ have isomorphic socles.  By the uniseriality lemma \ref{uniserialitylemma} we have either $\Delta I'\subseteq P_1$ or $P_1\subset \Delta P$. The latter is not possible, because $P_1$ is projective-injective module. Therefore either $\Delta I'\cong P_1$ or $\Delta I'$ is a proper submodule of $P_1$. This implies that either $Q\cong \Omega^1(I)$ or $Q\subset \Omega^1(I)$ respectively. This contradicts with lemma \ref{LEM injective first szyg is submodule of delta}.
\end{proof}

\begin{claim}\label{claim different injectives} If there exist at least two injective $\Lambda$-modules $I_1,\,I_2$ such that $\Omega^2(I_1)\cong\Omega^2(I_2)$, then the projective covers are isomorphic, i.e. $P(I_1)\cong P(I_2)$.
\end{claim}

\begin{proof}
Consider the projective resolutions of $I_i$ 
\begin{center}
 $\xymatrixcolsep{5pt}
\xymatrix{& \cdots\ar[rr]\ar[rd] && P(\Omega^1(I_i))\ar[rr]\ar[rd]   && P(I_i)\ar[rr] &&I_i\\
&&\Omega^2(I_i)\ar[ru] && \Omega^1(I_i)\ar[ru]}$
 \end{center}
 where $1\leq i\leq 2$. Because $\Omega^2(I_i)$ is a submodule of $P(\Omega^1(I_i))$, by lemma \ref{topsocConsecutive} $\soc \Omega^1(I_1)\cong \soc\Omega^1(I_2)$.  In particular it is isomorphic to socles of $P(I_i)$. Notice that $P(I_i)$'s are projective-injective modules since they are projective covers of injective modules. Indecomposable projective-injective modules with the same socle is unique upto isomorphism, therefore $P(I_1)\cong P(I_2)$.
\end{proof}
\begin{claim}\label{cor of different injectives} If there exists at least two injective $\Lambda$-modules $I_1$, $I_2$ satisfying $\Omega^2(I_1)\cong\Omega^2(I_2)$, then the dominant dimension of $\Lambda$ is one.
\end{claim}
\begin{proof}
By the claim \ref{claim different injectives}, the projective covers of $I_1$ and $I_2$ are isomorphic, however $\Omega^1(I_1)\ncong \Omega^1(I_2)$ because of $I_1\ncong I_2$. Therefore $P(\Omega^1(I_1))\ncong P(\Omega^1(I_2))$.  Notice that their socles are isomorphic by the uniserialty lemma \ref{uniserialitylemma}, since $P(\Omega^1(I_1))\supset\Omega^2(I_1)\cong\Omega^2(I_2)\subset P(\Omega^1(I_2))$. They are uniserial modules, without loss of generality let $P(\Omega^1(I_1))$ be a proper submodule of $P(\Omega^1(I_2))$. Therefore $P(\Omega^1(I_1))$ cannot be a projective-injective module, by the definition \ref{defdomdim gercek tanim}, the dominant dimension of $I_1$ is one which makes $\domdim\Lambda=1$.
\end{proof}

We have all the ingredients for the proof. We remind that $\domdim\Lambda\geq 3$. By the claims \ref{claim 331} and \ref{claim 332}, the second syzygies of injective $\Lambda$-modules are the corresponding modules for the injective $\syz$-modules, i.e. $\ho{\Omega^2(I)}$ is injective $\syz$-module for any injective non-projective $I\in\modd\Lambda$. On the other hand, by the claim \ref{cor of different injectives}, $\ho{\Omega^2(-)}$ is a bijection, because non-isomorphic injective non-projective $\Lambda$-modules is equivalent to non-isomorphic injective non-projective $\syz$-modules.

\end{proof}

We state and prove the part \ref{thmdomdimreduction} of Theorem \ref{bigthm1}.
\begin{theorem}\label{THM domdim reduction}
Assume that dominant dimension of cyclic Nakayama algebra $\Lambda$ is greater or equal than three. Then, we have the reduction
\begin{align*}
\domdim\Lambda=\domdim \bm{\varepsilon}(\Lambda)+2.
\end{align*}
\end{theorem}

\begin{proof}
By the assumption $\domdim\Lambda\geq 3$, there exists an injective $\Lambda$-module $I$ such that in the projective resolution
\begin{align}
\cdots\rightarrow P_3\rightarrow P_2\rightarrow P_1\rightarrow P_0\rightarrow I\rightarrow 0\nonumber
\end{align}
$P_0$, $P_1$ and $P_2$ are certainly projective-injective $\Lambda$-modules. Moreover, if $\domdim I=d$, $d\geq 3$, then each $P_i$, $2\leq i\leq d$ are projective-injective. By proposition \ref{PROP all higher syzygies are filtered}, every $P_i$ and $\Omega^i(I)$, $i\geq 2$ are filtered by $\cB(\Lambda)$, therefore we can carry the resolution of $\Omega^2(I)$ into $\modd\bm\varepsilon(\Lambda)$ as
\begin{align}
\cdots\rightarrow\Hom_{\Lambda}(\cP, P_3)\rightarrow \Hom_{\Lambda}(\cP,P_2)\rightarrow \Hom_{\Lambda}(\cP,\Omega^2(I))\rightarrow 0.\nonumber
\end{align}

By proposition \ref{PROP domdim geq 3 implies bijection}, $\Hom_{\Lambda}(\cP,\Omega^2(I))\in\modd\bm\varepsilon(\Lambda)$ is injective and by lemma \ref{LEM filtered proj inj implies it is filtered in epsilon}, $\Hom_{\Lambda}(\cP,P_i)$ $d\geq i\geq 2$ are projective-injective. Therefore 
\begin{align*}
\domdim_{\Lambda}I&=2+\domdim\Omega^2(I)\\&=2+\domdim_{\bm\varepsilon(\Lambda)}\Hom_{\Lambda}(\cP,\Omega^2(I))\\
&=2+\domdim_{\bm\varepsilon(\Lambda)} I'
\end{align*}
where $I'=\Hom_{\Lambda}(\cP,\Omega^2(I))$ is an injective $\bm\varepsilon(\Lambda)$-module. 
By the characterization \ref{defdomdim gercek tanim} of the dominant dimension, we can take the supremum,
\begin{align}
\begin{split}\nonumber
\domdim\Lambda&=\sup\left\{\domdim I\,\vert\, I\in\modd\Lambda \text{ is injective non-projective } \right\}\\
&=\sup\left\{2+\domdim \Omega^2(I)\,\vert\, \Omega^2(I)\in Filt(\cB(\Lambda)),\, I\in\modd\Lambda \right\}\\
&=2+\sup\left\{\domdim_{\bm\varepsilon(\Lambda)} I'\,\vert\, I'\cong\Hom_{\Lambda}(\cP,\Omega^2(I))\in\modd\bm\varepsilon(\Lambda)\right\}\\
&=2+\domdim\bm\varepsilon(\Lambda).
\end{split}
\end{align}
\end{proof}

\begin{remark} In Theorem \ref{THM domdim reduction}, we assume that $\domdim\Lambda\geq 3$. We will explain the mechanism behind it in  section \ref{section3}. Briefly, when global dimension is finite, semisimple components might arise in the syzygy filtered algebra and this is the obstacle. Let $\gldim\Lambda<\infty$ and $\domdim\Lambda=2$ with an injective module $I$ satisfying $\pdim I=2$. Then the syzygy filtered algebra has semisimple components. By the definition \ref{defdomdim}, the dominant dimension of semisimple algebra is infinity, and the dominant dimension of connected components can take any value. For example, let $(3,2,3,2,2,2,2,2)$ be Kupisch series of $\Lambda$. Then $\gldim\Lambda=6$ and $\domdim\Lambda=2$. $\bm\varepsilon(\Lambda)$ splits into linear Nakayama algebra $L$ given by Kupisch series $(2,2,2,2,1)$ and a semisimple component $\Aa_1$. We get $\gldim\bm\varepsilon(\Lambda)=\max\left\{4,0\right\}=4$, however $\domdim\bm\varepsilon(\Lambda)=\min\left\{\domdim L,\domdim\Aa_1\right\}=\min\left\{4,\infty\right\}=4$ which is greater than $\domdim\Lambda=2$.
\end{remark}

\begin{proposition}\label{PROP small domdim} If $\varphi\dim\Lambda=2$, then $1\leq \domdim\Lambda\leq 2$.
\end{proposition}

\begin{proof}
Any injective module is quotient of a projective-injective $\Lambda$-module, therefore dominant dimension has always lower bound one.\par

If $\Lambda$ is of finite global dimension or Gorenstein, then it is clear that dominant dimension is bounded by $2$. So, we focus on the remaining case.\par

Assume to the contrary, let $\domdim\Lambda\geq 3$. Therefore the projective resolution\begin{align*}
\cdots\rightarrow P_2\rightarrow P_1\rightarrow P_0\rightarrow I\rightarrow 0
\end{align*}
of any injective $\Lambda$-module $I$ implies that $\Omega^2(I)$ is not projective and at least $P_0, P_1,P_2$ are projective-injective modules by the definition \ref{defdomdim gercek tanim}.\\
We can choose an injective module $I$ such that $\Omega^1(I)$ is simple module. Since $\Omega^2(I)$ is not projective, $P_1=P(\Omega^1(I))$ is minimal projective module, i.e. $\rad P_1\cong\Omega^2(I)$ is not projective. This forces that $\Omega^1(I)$ is an element of the top set $\cS'(\Lambda)$. On the other hand $\Omega^1(I)$ is the socle of $P_0$, therefore $\Omega^1(I)$ is also an element of the socle set $\cS(\Lambda)$. By proposition \ref{PROP properties of the basesetproperties} \ref{item last in prop 2.18}, $\Omega^1(I)\in \cS(\Lambda)\cap\cS'(\Lambda)$ implies that $\Omega^1(I)$ is an element of the base set $\cB(\Lambda)$.\par
$P_0$ is projective-injective module, so $P=P(\topp P_0)$ is minimal projective where $\tau\topp P\cong\topp P_0\cong\topp I$. By propositions \ref{PROP x in base set iff x is omegaone of radical} and \ref{PROP x is an element of the base set iff it is second syzygy} , $\Omega^2(\topp P)\cong\Omega^1(\rad P)$ is isomorphic to $\Omega^1(I)$, which shows that $I\cong\rad P$. By lemma \ref{LEM projectivecannotbequotientlemma}, $I$ cannot be injective module, which creates the contradiction. Therefore $\domdim\Lambda\leq 2$.
\end{proof}

\begin{theorem}\label{THM domdim}
If $\Lambda$ is cyclic non-selfinjective Nakayama algebra then $\varphi\dim\Lambda\geq \domdim\Lambda$.
\end{theorem}
\begin{proof} If global dimension of $\Lambda$ is finite, then $\gldim\Lambda=\varphi\dim\Lambda\geq \domdim\Lambda$ is clear.\\
When global dimension is infinite, but algebra is Gorenstein, then dominant dimension is again bounded by  $\varphi$-dimension. Now we assume that algebra is not Gorenstein and is of infinite global dimension. By proposition \ref{PROP infinite global dimension, implies epsilon is cyclic}, any higher filtered algebra is cyclic, hence there is no semisimple components.

Let $\varphi\dim\Lambda=2d$.  Therefore $\varphi\dim\bm\varepsilon^{d-1}(\Lambda)=2$, and by proposition \ref{PROP small domdim}
\begin{align*}
\domdim\bm\varepsilon^{d-1}(\Lambda)\leq \varphi\dim\bm\varepsilon^{d-1}(\Lambda)=2.
\end{align*}
By the reductions \ref{thm reductionvar} and \ref{THM domdim reduction} we get
\begin{gather*}
\domdim\bm\varepsilon^{d-2}(\Lambda)\leq \varphi\dim\bm\varepsilon^{d-2}(\Lambda)=4.\\
\vdots\\
\domdim\bm\varepsilon(\Lambda)\leq \varphi\dim\bm\varepsilon(\Lambda)=2d-2\\
\domdim\Lambda\leq \varphi\dim\Lambda=2d.
\end{gather*}
which is the upper bound for dominant dimension. 
\end{proof}
This result appears in \cite{marczinzik2018upper}. We give another proof.

\begin{corollary}\label{domdim}
Dominant dimension of cyclic non-selfinjective Nakayama algebra $\Lambda$ is bounded by $2r$, where $r$ is the number of relations defining the algebra.
\end{corollary}
\begin{proof}
By the  results \ref{THM domdim} and \ref{Sen evenfidim}, $\domdim\Lambda\leq \varphi\dim\Lambda\leq 2r$.
\end{proof}

\subsection{Results about right finitistic and $\varphi$-dimensions}

We define $\findim\Lambda^{op}$ as finitistic dimension  of the algebra $\Lambda$ with respect to injective resolutions i.e.
\begin{align}
\findim\Lambda^{op}:=\sup\left\{\indim_{\Lambda} M\vert\,\, \indim_{\Lambda} M<\infty,\,\, M\in\text{mod-}\Lambda  \right\}\nonumber
\end{align}
In general left and right finitistic dimensions of an algebra can be different. But in the case of Nakayama algebras, we prove that they are same. To perform this, we need dual constructions stated in section \ref{section1}.

Injective $\Lambda$ modules are characterized by their tops using the relations given in \ref{relations}:
\begin{gather*}
 P_{(k_{2r-1})+1}=I_{k_{2}}\twoheadrightarrow\ldots\twoheadrightarrow I_{k_4+1}  \quad\text{ have simple } S_{k_{2r-1}+1}\text{  as their top}\\
P_{k_1+1}=I_{k_4}\twoheadrightarrow\ldots\twoheadrightarrow I_{k_2+1}   \quad\text{ have simple } S_{k_1+1} \text{  as their top}\\
\qquad\qquad\qquad\vdots\qquad\qquad\qquad\qquad\quad\quad\quad\\
P_{(k_{2r-3})+1}=I_{k_{2r}} \twoheadrightarrow\ldots\twoheadrightarrow I_{k_4+1}  \quad\quad\text{ have simple } S_{k_{2r-3}+1} \text{  as their top}
\end{gather*}

Let $\cT(\Lambda)$ be the complete set of representatives of tops of injective modules over $\Lambda$. By using the system of relations \ref{relations}, it is 
\begin{align*}
\cT(\Lambda)=\left\{S_{k_1+1}, S_{k_3+1},\ldots,S_{k_{2r-1}+1}\right\}.
\end{align*}
We call $\cT(\Lambda)$ as \emph{opposite top set}.
We define the set $\cT'(\Lambda)$ which is the complete set of representatives of inverse Auslander-Reiten translates of the tops of indecomposable injective modules. Hence $S_i\in\cT(\Lambda)$ if and only if $\tau^{-1} S_i\in \cT'(\Lambda)$. Because $\tau^{-1} S_i\cong S_{i-1}$, we get 
\begin{align*}
\cT'(\Lambda)=\left\{S_{k_{1}}, S_{k_3},\ldots,S_{k_{2r}-1}\right\}.
\end{align*}
We call $\cT'(\Lambda)$ as \emph{opposite socle set}.

\begin{definition}\label{defdualshortest} An indecomposable $\Lambda$-module $M$ satisfying $\topp M\in\cT(\Lambda)$ and $\soc M\in\cT'(\Lambda)$  is called \emph{shortest} if the composition factors of $M$ except $\topp M$ and $\soc M$ are not elements of $\cT(\Lambda)$ and $\cT'(\Lambda)$.
\end{definition}

\begin{definition} \label{definitionofnablaset}let $\Lambda$ be  a cyclic Nakayama algebra defined by the irredundant system of $r$-relations \ref{relations}. For each $j\in\{1,\dots,r\}$ let $\nabla_j$ be a shortest indecomposable uniserial module with 
$\soc\nabla_j\cong S_{k_{2j+1}}$ and $\topp\nabla_j\cong S_{k_{2j-1}+1}$. The complete set of representatives of modules $\nabla_j$'s is called \emph{the opposite base set} and denoted by $\nabla(\Lambda)$. Explicity we have
\begin{center}
${\bf\nabla}(\Lambda):=\left\{ \nabla_1\cong\begin{vmatrix}
    S_{k_{1}+1} \\
    \vdots  \\
    S_{k_{3}}
\end{vmatrix}\!, \nabla_2\cong\begin{vmatrix}
    S_{k_{3}+1}  \\
    \vdots  \\
   S_{k_{5}}
\end{vmatrix}\!,..,\nabla_j\cong\begin{vmatrix}
   S_{k_{2j-1}+1}  \\
    \vdots  \\
    S_{k_{2j+1}}
\end{vmatrix}\!,..,\nabla_r\cong \begin{vmatrix}
   S_{k_{2r-1}+1}  \\
    \vdots  \\
    S_{k_{1}}
\end{vmatrix}\!\right\}.$ 
\end{center}
\end{definition}

\begin{proposition}\label{PROP properties of the nabla filtration} Regarding the opposite base set $\nabla(\Lambda)$, we have
\begin{enumerate}[label=\arabic*)]
\item the top of each $\nabla_i$ is an element of the opposite top set $\cT(\Lambda)$, i.e. $\topp\nabla_i\in\cT(\Lambda)$. 
\item Any element $S$ of the opposite socle set $\cT'(\Lambda)$ is a socle of an element of $\nabla(\Lambda)$ i.e. $\soc\nabla_i\cong S$.
\item Any simple $\Lambda$-module $S$ appears in the composition series of exactly one $\nabla_i$. Equivalently, the simple composition factors of distinct $\nabla_i$'s are disjoint.
\item Distinct elements of the opposite base set are Hom-orthogonal i.e. $\Hom_{\Lambda}\left(\nabla_i,\nabla_j\right)\cong 0$ when $i\neq j$.
\item Each $\nabla_i$ is a submodule of projective-injective module.
\item $\nabla_i$ is simple $\Lambda$-module if and only if $S\cong\nabla_i$ satisfies $S\in\cT'(\Lambda)\cap\cT(\Lambda)$.
\end{enumerate}
\end{proposition}

One can repeat all the constructions we discussed for projective resolutions and $\cB(\Lambda)$- filtered modules for injective resolutions and  $\nabla$-filtered modules. We do not want to progress in those directions however the following statements can be concluded by using duality between projective $\Lambda$-modules and injective $\Lambda^{op}$-modules:
\begin{remarks} 
\begin{enumerate}[label=\roman*)]
\item The cosyzygy filtered algebra is the endomorphism algebra of injective envelopes of simple modules in $\cT'(\Lambda)$, i.e.
\begin{align*}
\eta(\Lambda):=\End_{\Lambda}\left(\bigoplus\limits_{S\in \cT'(\Lambda)}I(S)\right)\nonumber
\end{align*}

\item If the second cosyzygy $\Sigma^2(M)$ of $M$ is not trivial, then it has $\nabla(\Lambda)$-filtration.
\item Dual $\varphi$ function can be constructed as:
$$\varphi_R(M):=\min\{t\ |\ \rank\left(DL^t\langle \add M\rangle\right)=\rank\left(DL^{t+j}\langle \add M\rangle\right)\text{ for }\forall j\geq 1\}.$$
where $DL[M]:=[\Sigma M]$ and gives map $DK_0\mapsto DK_0$, where $DK_0$ is abelian group generated by all symbols $[X]$ modulo relations:
\begin{itemize}
\item $[A_1]=[A_2]+[A_3]$ if $A_1\cong A_2\oplus A_3$
\item $[I]=0$ if $I$ is injective.
\end{itemize}
\item Therefore we can define $\varphi\dim\Lambda^{op}$ as:
\begin{align*}
\varphi\dim\Lambda^{op}:=\sup\left\{\varphi_R(M)\vert\,\, \text{for all}\, M\in \text{mod-}\Lambda \right\}
\end{align*}
\item Let $\Lambda$ be cyclic non-selfinjective Nakayama algebra. $\varphi\dim\Lambda^{op}=2$ if and only if $\eta(\Lambda)$ is selfinjective.
\item Another result we need is dual of Theorem \ref{Thm difference between dims}: If $\Lambda$ is Nakayama algebra then $\varphi\dim\Lambda^{op}-\findim\Lambda^{op}\leq 1$.
\end{enumerate}

\end{remarks}

\begin{proposition}\label{PROP findim 1 implies projective injective modules have filtration} $\findim\Lambda=1$ if and only if $\cS'(\Lambda)=\cT'(\Lambda)$.
\end{proposition}
Recall that we say  module $M$ is periodic if there is number $i$ such that $M\cong\Omega^i(M)$ .\begin{proof}
$(\Rightarrow)$. Let $\findim\Lambda$ be one. Therefore either a simple module is of projective dimension one or its first syzygy is periodic. Therefore $\Omega^1(S)$ is periodic if and only if $S$ is top of minimal projective $P$. By using the relations \ref{relations}, $S\in\left\{S_{k_1},S_{k_3},\ldots,S_{k_{2r-1}}\right\}$, and the top of $\Omega^1(S)$ is in the set $\left\{S_{k_1+1},\ldots,S_{k_{2r-1}+1}\right\}$ by lemma  \ref{topsocConsecutive}. Since $\Omega^1(S)$ is periodic, $\Omega^1(S)\cong\Omega^i(S)$ for some $i$, therefore it has $\cB(\Lambda)$-filtration by proposition \ref{PROP all higher syzygies are filtered}. This means that the top of $\Omega^1(S)\in\cS'(\Lambda)$. This is true for all the tops of minimal projectives which is the set of size $r$, hence we get $\cS'(\Lambda)=\cT'(\Lambda)$.\par
$(\Leftarrow)$ If $\cS'(\Lambda)=\cT'(\Lambda)$, then $\cS(\Lambda)=\cT(\Lambda)$ and $\cB(\Lambda)=\nabla(\Lambda)$. By using the system of relations \ref{relations}, we see that each projective-injective module is $\cB(\Lambda)$-filtered. Moreover, they have the same $\cB(\Lambda)$-length. We get $\varphi\dim\Lambda=2$ (\cite{sen2018varphi}). By Theorem \ref{Thm difference between dims}, it is enough to show that there is no module of projective dimension two. If there was a module $M$ such that $\Omega^2(M)$ is projective, then the exact sequence
\begin{align}
0\rightarrow P_2=\Omega^2(M)\rightarrow P_1\rightarrow \Omega^1(M)\rightarrow 0\nonumber
\end{align}
would imply that $\Omega^1(M)$ is quotient of $\Delta_i\in\cB(\Lambda)$ for some $i$. Hence it cannot be a submodule of any projective module, so there is no module of projective dimension two.
\end{proof}

\begin{proposition}\label{PROP findim1 implies findim op is 1} $\findim\Lambda=1$ if and only if $\findim\Lambda^{op}=1$.
\end{proposition}
\begin{proof}
By proposition \ref{PROP findim 1 implies projective injective modules have filtration}, $\findim\Lambda=1$ if and only if $\cS'(\Lambda)=\cT'(\Lambda)$ if and only if every projective-injective module is in $Filt(\cB(\Lambda))$. If there exists $M$ such that $\Sigma^2(M)$ is injective module, then $\Sigma^2(M)$ has $\nabla(\Lambda)=\cB(\Lambda)$-filtration. This makes $\Sigma^2(M)$ projective-injective module which is not possible. Therefore $\findim\Lambda^{op}=1$.
\end{proof}
\begin{theorem}\label{Thm findimop}
For cyclic Nakayama algebras left and right finitistic dimensions are same, i.e. $\findim\Lambda=\findim\Lambda^{op}$.
\end{theorem}
\begin{proof}
if global dimension is finite, it is a well known result. Therefore we consider the case of infinite global dimension. \par
We will prove it by induction on the dimension.
The previous proposition \ref{PROP findim1 implies findim op is 1} verifies that if $n=1$ and $n=2$, $\findim\Lambda=1$ if and only if $\findim\Lambda^{op}=1$ and $\findim\Lambda=2$ if and only if $\findim\Lambda^{op}=2$ (combined with Theorem \ref{Thm difference between dims} ). Assume that $\findim\Lambda=\findim\Lambda^{op}=d$ for all $d\leq n$. Now, let $\findim\Lambda=n+1$. We get:
\begin{align*}
\findim\Lambda=n+1 \iff \findim\bm{\varepsilon}(\Lambda)=n-1, \text{by prop\,\ref{PROP findim reduction}}\\
\findim\bm{\varepsilon}(\Lambda)=n-1 \iff \findim\bm{\varepsilon}(\Lambda)^{op}=n\!-\!1\,\, \text{by induction hypothesis}\\
\findim\bm{\varepsilon}(\Lambda)^{op}=n-1 \iff \findim\Lambda^{op}=n+1\,\, \text{by dual prop\,\ref{PROP findim reduction}}
\end{align*}
\end{proof}

\begin{theorem}\label{fidimop} For cyclic Nakayama algebras left and right $\varphi$-dimensions are same i.e. $\varphi\dim\Lambda=\varphi\dim\Lambda^{op}$.
\end{theorem}
\begin{proof}
It is enough to consider global dimension is infinite.\\
By Theorem \ref{Thm difference between dims}, there are two possibilities:
\begin{align*}
\varphi\dim\Lambda=\findim\Lambda\\
\varphi\dim\Lambda=1+\findim\Lambda
\end{align*}
The Theorem \ref{Thm findimop} together with the equalities of left and right finitistic dimensions imply
\begin{align*}
\varphi\dim\Lambda=\varphi\dim\Lambda^{op}.
\end{align*}\end{proof}

In \cite{bmr}, authors show equalities of left and right $\varphi$-dimension for truncated path algebras. 

We can put together all the results on the upper bounds.

\begin{theorem}
Let $\Lambda$ be cyclic non-selfinjective Nakayama algebra  defined by $r$ many irredundant system of relations over $N$ vertices. Then, $\varphi\dim\Lambda$, $\findim\Lambda$, $\gordim\Lambda$, $\domdim\Lambda$ are bounded by $2r$ where $r=\vert\cB(\Lambda)\vert$. 
\end{theorem}

\begin{corollary} The upper bound in terms of the rank of the algebra is $2N-2$.
\end{corollary}
\begin{proof}
By the previous result, each dimension is bounded by $2r$. By the characterization of selfinjective algebras \ref{Rem characterization of selfinjective algebras}, $r=N$ if and only if algebra is selfinjective. Hence the greatest value that $r$ can take is $N-1$, therefore $2r\leq 2N-2$.
\end{proof}

 \begin{example}\label{example} We want to show that the upper bound is sharp. We reconsider the example given in \cite{sen2018varphi}. Let $\Lambda$ be cyclic Nakayama algebra on $N$ vertices with Kupisch series $(2N+1,2N+1,\ldots,2N+1,2N)$. This has $r=N-1$ and $\varphi\dim(\Lambda)=2r=2N-2$. Indeed, one can check that, it is Gorenstein and $\findim\Lambda=\varphi\dim\Lambda=\gordim\Lambda=\domdim\Lambda=\findim\Lambda^{op}=\varphi\dim\Lambda^{op}=2r=2N-2$. 
\end{example}

\section{The Structure of Syzygy Filtered algebras}\label{section3}

The proof of the first part of Theorem \ref{bigthm2} is stated in \ref{THM fidim becomes stable }. Now we give the proof for the second part \ref{thmsplit} which we recall below.

\begin{theorem}\label{thm en son splitting}
 If $\Lambda$ is a cyclic connected Nakayama algebra of finite global dimension, then there exists a positive integer $k$ such that $\bm\varepsilon^k(\Lambda)$ is a cyclic connected Nakayama algebra and $\bm\varepsilon^{k+1}(\Lambda)$ is not cyclic. $\bm\varepsilon^{k+1}(\Lambda)$ can split into components that are either linear Nakayama algebras or semisimple components or both.
\end{theorem}

\subsection{When is the syzygy filtered algebra non-cyclic}

\begin{proposition}\label{PROP gldim 2 implies filtered algebra is semisimple} If $\Lambda$ is a cyclic Nakayama algebra of global dimension two, then the syzygy filtered algebra is semisimple.
\end{proposition}
\begin{proof}
Any simple module which is top of minimal projective have projective dimension two by the assumption on the global dimension. Therefore $\Omega^2(S)$ is projective module and an element of the base set by proposition \ref{PROP x is an element of the base set iff it is second syzygy}. Therefore the corresponding module $\Hom_{\Lambda}(\cP,\Omega^2(S))$ is simple projective $\bm\varepsilon(\Lambda)$-module by the categorical equivalence \ref{diagram categorical equivalence}. We need to show that $P=\Omega^2(S)$ cannot be a submodule of another projective $\Lambda$-module $P'$ which has $\cB(\Lambda)$-filtration. Assume to the contrary that $P'$ is in $Filt(\cB(\Lambda))$ and $P$ is proper submodule of $P'$. We get the exact sequence
\begin{align*}
0\rightarrow P\rightarrow P'\rightarrow Q=\faktor{P'}{P}\rightarrow 0
\end{align*}
where $Q$ is not trivial. Since $P,P'$ have $\cB(\Lambda)$-filtration, $Q$ has $\cB(\Lambda)$-filtration, because $Filt(\cB(\Lambda))$ is exact category by the corollary \ref{COR category filt is exact}. Moreover, the result \ref{PROP X has filtration then it is second syzygy} implies that there exists $\Lambda$-module $M$ such that $\Omega^2(M)\cong Q$. This means $\Omega^3(M)\cong\Omega^1(S)\cong P$ and $\pdim M=3$ which contradicts to $\gldim\Lambda=2$. Therefore $\cB(\Lambda)$-filtered projective module $P$ cannot be submodule of another $P'\in Filt(\cB(\Lambda))$ which makes the category $Filt(\cB(\Lambda))$ semisimple category. By the categorical equivalence \ref{diagram categorical equivalence}, $\modd\bm\varepsilon(\Lambda)$ is semisimple algebra.
\end{proof}

\begin{example} Let $\Lambda$ be given by the Kupisch series $(7,6,5,4,3,6,5,4,3)$. Then the base set is $\left\{P_3,P_8\right\}$ and it is clear that the other projective modules are not filtered by it. Therefore $\bm\varepsilon(\Lambda)\cong\Aa_1\oplus\Aa_1$.
\end{example}

\begin{proposition}\label{PROP pdim simple is 2 implies linear} Let $\Lambda$ be cyclic Nakayama algebra with a simple module $S$ satisfying $p\dim S=2$. Then $\bm{\varepsilon}(\Lambda)$ is not cyclic Nakayama algebra.
\end{proposition}

\begin{proof}
Since $\pdim S=2$, it has to be top of minimal projective $P$. Otherwise, $\Omega^1(S)$ would be isomorphic to $\rad P'\cong P''$ where $P'$ is not minimal, and makes $\pdim S=1$.\\
In \ref{PROP x is an element of the base set iff it is second syzygy}, we showed that the second syzygy of top of a minimal projective is an element of the base set $\cB(\Lambda)$. Moreover by the assumption $\pdim S=2$, $\Omega^2(S)$ is projective module. By the categorical equivalence, the corresponding module $\Hom_{\Lambda}(\cP,\Omega^2(S))$ 
\begin{itemize}
\item is projective, because $\Omega^2(S)$ is projective
\item is simple, because $\Omega^2(S)\in\cB(\Lambda)$.
\end{itemize}
Therefore $\bm\varepsilon(\Lambda)$ has at least one simple projective module which shows it is not cyclic.
\end{proof}

\begin{proposition}\label{PROP semisimple iff no extension in lambda} $\bm\varepsilon(\Lambda)$ has semisimple component if and only if the corresponding module in $\modd\Lambda$ has no extensions in $Filt(\cB(\Lambda))$.
\end{proposition}
\begin{proof}
By the remark \ref{Rem equivalnece of resolutions}, any nonsplit exact sequence 
 \begin{align}
 0\rightarrow A\rightarrow B\rightarrow C\rightarrow 0\nonumber
 \end{align}
 in $\modd\bm\varepsilon(\Lambda)$ is equivalent to the exact sequence
 \begin{align}
 0\rightarrow \Delta A\rightarrow\Delta B\rightarrow \Delta C\rightarrow 0\nonumber
 \end{align}
 in $Filt(\cB(\Lambda))$. Therefore, a simple $\bm\varepsilon(\Lambda)$-module $S$ has no extension if and only if the corresponding module $\Delta S$ has no extension in $Filt(\cB(\Lambda))$, where $\Hom_{\Lambda}(\cP,\Delta S)\cong S$. Therefore none of the terms (except $\Delta S$) of any nonsplit short exact sequence involving $\Delta S$ in $\modd\Lambda$ has $\cB(\Lambda)$-filtration.
\end{proof}

\begin{corollary} Let $\Lambda$ be cyclic Nakayama algebra with a simple module $S$ satisfying $\pdim S=2$. Then $\Hom_{\Lambda}(\cP,\Omega^2(S))$ is semisimple if and only if none of the projective modules $P$ having $\Omega^2(S)$ as a proper submodule have $\cB(\Lambda)$-filtration.
\end{corollary}
\begin{proof}
$\pdim S=2$ implies that $\Omega^2(S)$ is projective and by proposition \ref{PROP x in base set iff x is omegaone of radical} $\Omega^2(S)$ is an element of $\cB(\Lambda)$. Therefore $\Hom_{\Lambda}(\cP,\Omega^2(S))=S'$ is simple projective $\bm\varepsilon(\Lambda)$-module.\par By proposition \ref{PROP semisimple iff no extension in lambda}, $S'$ is simple module of the semisimple component if and only if $\Omega^2(S)$ has no extension in $Filt(\cB(\Lambda))$. Therefore any indecomposable projective $\Lambda$-module $P$ having the projective module $\Omega^2(P)$ as a proper submodule cannot have $\cB(\Lambda)$-filtration. 
\end{proof}

\begin{proposition}\label{Prop simple} If there exists a simple module $S$ satisfying $\pdim S=2$ minimally, i.e. there is at least one  simple module $S'$ such that $S'\ncong S$ and $\pdim S'>2$, then $\bm\varepsilon(\Lambda)$ has at least one connected component which is linear Nakayama algebra. 
\end{proposition}
\begin{proof}
By the result \ref{PROP pdim simple is 2 implies linear}, $\bm\varepsilon(\Lambda)$ is not cyclic. By the assumption of the statement, there is a simple module $S'$ such that $\pdim S'\geq 3$. Therefore in the projective resolution 
\begin{align*}
\cdots\rightarrow P_3\rightarrow P_2\rightarrow P_1\rightarrow P_0\rightarrow S'\rightarrow 0
\end{align*} 
of $S'$, the exact sequence
\begin{align}
0\rightarrow \Omega^3(S')\rightarrow P_2\rightarrow \Omega^2(S')\rightarrow 0\nonumber
\end{align}
is nonsplit exact. By \ref{PROP all higher syzygies are filtered}, the sequence
\begin{align}
0\rightarrow \Hom_{\Lambda}(\cP,\Omega^3(S'))\rightarrow\Hom_{\Lambda}(\cP, P_2)\rightarrow \Hom_{\Lambda}(\cP,\Omega^2(S'))\rightarrow 0\nonumber
\end{align}
in $\modd\bm\varepsilon(\Lambda)$ is exact by the categorical equivalence. Therefore the component having the exact sequence cannot be semisimple which means it is linear Nakayama algebra.
\end{proof}

\begin{example} We will compare the algebras given by the Kupisch series $(3,2,2)$ and $(3,2,3,2,2)$. Both of them have global dimension $3$ and unique simple modules with projective dimension two. However the syzygy filtered algebras are given by $(2,1)$ and $(2,1)\oplus\Aa_1$. 
\end{example}

\subsection{When is the higher syzygy filtered algebra non-cyclic}

\begin{proposition}\label{PROP higher algebra is semisimple } If all simple modules which are tops of the minimal projective $\Lambda$-modules have the same projective dimension $2k$, then $\bm\varepsilon^k(\Lambda)$ is semisimple.
\end{proposition}
\begin{proof}
We prove it by induction. The case $k=1$ is proved in \ref{PROP gldim 2 implies filtered algebra is semisimple}.\\
\begin{claim}
If $k=2$, then $\bm\varepsilon(\Lambda)$ is cyclic and $\bm\varepsilon^2(\Lambda)$ is semisimple.
\end{claim}
\begin{proof}
$\Lambda$ is cyclic Nakayama algebra so we can construct $\bm\varepsilon(\Lambda)$ by the definition \ref{deffilteredalg}. By the categorical equivalence \ref{diagram categorical equivalence} and the remark \ref{Rem equivalnece of resolutions}, simple $\bm\varepsilon(\Lambda)$-modules are of the form $\Hom_{\Lambda}(\cP,\Delta S)$. $\Delta S$ is the second syzygy of simple $\Lambda$-module $S'$ which is the top of a minimal projective by propositions \ref{PROP X has filtration then it is second syzygy} and \ref{PROP x in base set iff x is omegaone of radical}. By the assumption $\pdim S'\neq 2$, therefore simple $\bm\varepsilon(\Lambda)$-modules are not projective. This shows that the syzygy filtered algebra is cyclic, therefore we can constuct $\bm\varepsilon^2(\Lambda)$. Moreover, by the reduction \ref{PROP gldim reduction by 2}, $\gldim\Lambda=\gldim\bm\varepsilon(\Lambda)+2$, which makes $\gldim\bm\varepsilon(\Lambda)=2$. By proposition \ref{PROP gldim 2 implies filtered algebra is semisimple}, $\bm\varepsilon^2(\Lambda)$ is semisimple.
\end{proof}
Now we can look arbitrary $k$. Assume that the claim holds for all $k=1,2,\ldots,m$. Let $k=m+1$. Since $\bm\varepsilon(\Lambda)$ is cyclic and of global dimension $2k$ which satisfies the induction hypothesis, claim follows.
\end{proof}

\begin{proposition}\label{PROP pdim simple 4 implies cyclic then linear} Let $\Lambda$ be a cyclic Nakayama algebra with a simple module $S$ satisfying $\pdim S=4$ minimally, i.e. there is no other simple module $S'$ such that $\pdim S'=2$ and there exists at least one simple module $S''$ with $\pdim S''>4$. Then $\bm\varepsilon(\Lambda)$ is cyclic Nakayama algebra and $\bm\varepsilon^2(\Lambda)$ has linear component.
\end{proposition}

\begin{proof}
Since $\pdim S=4$, its projective cover has to be a minimal projective i.e. $\rad P(S)$ is not projective. Therefore $\Omega^2(S)$ is an element of $\cB(\Lambda)$. Moreover, by the assumption on the minimality, elements of the base set $\cB(\Lambda)$ are not projective. Therefore the corresponding $\bm\varepsilon(\Lambda)$-module $\Hom(\cP,\Omega^2(S))$ is simple by the categorical equivalence \ref{diagram categorical equivalence} but not projective in $\modd\bm\varepsilon(\Lambda)$, hence $\bm\varepsilon(\Lambda)$ is cyclic.\\

Let's denote $\Hom_{\Lambda}(\cP,\Omega^2(S))$ by $S'$. In $\modd\bm\varepsilon(\Lambda)$, $\pdim S'$ is two, by the reduction \ref{PROP carrying resolution on lambda to epsilon}. Now we can apply proposition  \ref{Prop simple} in order to conclude that $\bm\varepsilon^2(\Lambda)$ has component which is linear Nakayama algebra.
\end{proof}

\begin{proposition}\label{PROP pdim even implies general reduction} Let $\Lambda$ be a cyclic Nakayama algebra with a simple module $S$ satisfying $\pdim S=2k$ minimally, i.e. there is no other simple module $S$ such that $\pdim S=2k'$ with $k'<k$ and there exists at least one simple module $S''$ with $\pdim S''>2k$. Then $\bm\varepsilon^{k-1}(\Lambda)$ is cyclic Nakayama algebra and $\bm\varepsilon^k(\Lambda)$ has component which is linear Nakayama algebra.
\end{proposition}
\begin{proof}
Proof by induction. We gave the proofs of $m=1$ and $m=2$ in \ref{Prop simple} and \ref{PROP pdim simple 4 implies cyclic then linear} respectively. We assume that the statement is true for all $m=1,\ldots,k$. We need to analyze the case $m=k+1$. Assume that there is a simple module $S$ with projective dimension $2k+2$ and it is minimal in the sense that there is no other simple module with projective dimension $2k'$ where $k'<k+1$. $S$ has to be top of minimal projective module, otherwise $\pdim S$ would be one. By proposition \ref{PROP x is an element of the base set iff it is second syzygy}, $\Omega^2(S)$ is an element of the base set $\cB(\Lambda)$, and it is not projective, i.e. $\pdim\Omega^2(S)=2k$. If we apply the syzygy filtered algebra construction, the corresponding $\bm\varepsilon(\Lambda)$-module $\Hom_{\Lambda}(\cP,\Omega^2(S))$ is simple and not projective. Therefore $\bm\varepsilon(\Lambda)$ is cyclic Nakayama algebra. If we denote the corresponding module $\Hom_{\Lambda}(\cP,\Omega^2(S))$ by $S'$, then $\pdim_{\bm\varepsilon(\Lambda)}S'=\pdim_{\Lambda} S-2=2k$. This means $S'$ is top of a minimal $\bm\varepsilon(\Lambda)$-projective module. Since the reduction is two, therefore $\pdim S'=2k$ is minimal, i.e. there is no other simple $\bm\varepsilon(\Lambda)$-module with smaller even projective dimension. By the induction hypothesis claim follows.  
\end{proof}

\subsection{Applications of the splitting results}

Based on the syzygy filtration construction, we give the proof of the Madsen's result \cite{madsen2005projective}.

\begin{proposition}\label{PROP gldim infinite iff no simple even module} Global dimension of $\Lambda$ is infinite if and only if there is no simple module of even projective dimension.
\end{proposition}
\begin{proof}
By propositions \ref{Prop simple}, \ref{PROP pdim simple 4 implies cyclic then linear} and \ref{PROP pdim even implies general reduction} if there exists a simple module with even projective dimension, then we reach linear Nakayama algebra, so global dimension is finite.
Assume that there is no simple module with even projective dimension. Therefore there is no simple projective module in $\bm\varepsilon(\Lambda)$ which makes it cyclic and $\rank\Lambda\geq\rank\bm\varepsilon(\Lambda)$. Furthermore there is no simple module in $\modd\bm\varepsilon(\Lambda)$ with even projective dimension, otherwise we can lift the resolution to $\modd\Lambda$ and get $\pdim S$ even, not possible by the assumption. We can construct $\bm\varepsilon^2(\Lambda)$, and the same arguments are true, so there is no simple module of even projective dimension, and therefore there is no simple projective module. In each step we get nontrivial cyclic algebras $\bm\varepsilon^k(\Lambda)$. On the other hand in each step we get $\rank\bm\varepsilon^k(\Lambda)\geq \rank\bm\varepsilon^{k+1}(\Lambda)$. Since the rank cannot reach to zero, each algebra is nontrivial, $\lim_{i\rightarrow\infty}\rank\bm\varepsilon^i(\Lambda)$ has to stabilize, which means that we reach selfinjective algebra, so global dimension of $\Lambda$ is infinite. 
\end{proof}


\begin{theorem}\label{RES upper bound gldim} If global dimension of cyclic Nakayama algebra $\Lambda$ is finite, then there exists $m$ such that $\bm\varepsilon^m(\Lambda)$ is not cyclic, and 
\begin{align}
\gldim\Lambda\leq 2m+r_{m-1}-C\nonumber
\end{align}
where $r_{m-1}$ is the number of relations of the cyclic algebra $\bm\varepsilon^{m-1}(\Lambda)$ and $C$ is the number of connected components of $\bm\varepsilon^m(\Lambda)$.
\end{theorem}

\begin{proof}
By proposition \ref{PROP gldim infinite iff no simple even module}, global dimension of $\Lambda$ is finite if and only if there exists at least one simple module $S$ with even projective dimension. Therefore we can use propositions \ref{PROP pdim even implies general reduction} or \ref{PROP higher algebra is semisimple } depending on the conditions on the other simple modules. 
Assume that $\Lambda$ satisfies the conditions in proposition \ref{PROP pdim even implies general reduction}. Therefore there exists $m$ such that $\bm\varepsilon^{m-1}(\Lambda)$ is cyclic, $\bm\varepsilon^m(\Lambda)$ is not cyclic and 
\begin{align}
\gldim\Lambda=2m+\gldim\bm\varepsilon^m(\Lambda).\nonumber
\end{align}
On the other hand, the rank of $\bm\varepsilon^m(\Lambda)$ is the number of relations of the cyclic algebra $\bm\varepsilon^{m-1}(\Lambda)$ which is denoted by $r_{m-1}$. We recall a well-known result on linear Nakayama algebras (or more generally epresentation directed algebras), global dimension of linear Nakayama algebra $L$ is bounded by $\rank L-1$. If $\bm\varepsilon^m(\Lambda)$ splits into components $L_1\oplus L_2\oplus\cdots\oplus L_C$, then $\gldim\bm\varepsilon^m(\Lambda)\leq \max_i\left\{\gldim L_i\right\}$. Therefore global dimension of $\bm\varepsilon^m(\Lambda)$ can be at most $r_{m-1}-C$. We conclude that 
\begin{align}
\gldim\Lambda&=2m+\gldim\bm\varepsilon^m(\Lambda)\nonumber\\&\leq 2m+r_{m-1}-C.\nonumber
\end{align}
Now we assume that $\Lambda$ satisfies the conditions in proposition \ref{PROP higher algebra is semisimple }. In this case $\gldim\bm\varepsilon^{m-1}(\Lambda)=2$, and $r_{m-1}=C$ gives the number of simple summands of the semisimple algebra $\bm\varepsilon^m(\Lambda)$. There exists at least one component, so $r_{m-1}\geq 1$. This implies the desired inequality
\begin{align}
\gldim\Lambda=2m\leq 2m+r_{m-1}-C.\nonumber
\end{align}
\end{proof}

\begin{remark}

We give another proof of the main theorem of \cite{madsen2018bounds} by using syzygy filtration method.
\begin{theorem}\label{madsen}\cite{madsen2018bounds}
Let $\Lambda$ be a Nakayama algebra with a simple module $S$ of even projective dimension.  Choose $m$ minimal such that a simple $\Lambda$ module has projective dimension equal to $2m$. Then the global dimension of $\Lambda$ is bounded by $N+m-1$ where $N$ is the number of vertices of $\Lambda$
\end{theorem}
\begin{proof}
By the result \ref{RES upper bound gldim}, we have the inequality
\begin{align}
\gldim\Lambda\leq 2m+r_{m-1}-C.\nonumber
\end{align}
We need to find the possible minimal value of $C$ and maximal value of $r_{m-1}$. It is clear that $C=1$ is the minimum. Let $r_j$ denote the number of relations of higher syzygy filtered algebras $\bm\varepsilon^{j}(\Lambda)$ where $1\leq j\leq m-1$ and $r$ denotes the number of relations of $\Lambda$. Therefore the result \ref{PROP filtered algebra is Nakayama} applied to higher syzygy filtered algebras implies $\rank\bm\varepsilon(\Lambda)=r$ and $\rank\bm\varepsilon^j(\Lambda)=r_{j-1}$ for $2\leq j\leq m$. In each reduction, ranks of the algebras reduce, therefore $r_{m-1}<r_{m-2}<\cdots<r_1<r$ which is equivalent to the system of inequalities
\begin{gather*}
r+1\leq N\\
r_1+1\leq r\\
r_2+1\leq r_1\\
\vdots\\
r_{m-1}+1\leq r_{m-2}.
\end{gather*}
If we add all the terms, we get $r_{m-1}+m\leq N$. Therefore
\begin{align}
\gldim\Lambda\leq 2m+r_{m-1}-1\leq 2m+(N-m)-1=N+m-1.\nonumber
\end{align}
\end{proof}

\end{remark}

\begin{corollary}\cite{gustafson1985global} Let $\Lambda$ be cyclic Nakayama algebra of finite global dimension. Then $gl\dim\Lambda\leq 2N-2$ where $N$ is the number of vertices.
\end{corollary}
\begin{proof}
By the previous result, global dimension of $\Lambda$ where $\rank\Lambda=N$ is bounded by $N+m-1$ where $2m$ is projective dimension of a simple $\Lambda$-module which is minimal. On the other hand, $m$ is the number of reductions from $\Lambda$ to $\bm\varepsilon^m(\Lambda)$. Since the minimal value of $\rank\bm\varepsilon^m(\Lambda)=1$, $m$ can be at most $N-1$. Therefore we get the upper bound $2N-2$ for global dimension.
\end{proof}
\begin{remark} Indeed the Gustafson's example is the unique algebra such that $\gldim\Lambda=2N-2$ and $\rank\Lambda=N$. We give a proof of this. When $N=2$, the only algebra with finite global dimension $2$ is given by the Kupisch series $(3,2)$. Assume that there exist $\Lambda$ with $\rank\Lambda=3$ and $\gldim\Lambda=4$. Therefore the syzygy filtered algebra satisfies $\gldim\bm\varepsilon(\Lambda)=2$, and the rank has to satisfy $\rank\bm\varepsilon(\Lambda)\leq 2$. The rank cannot be one, because rank one Nakayama algebra is either semisimple or selfinjective, which forces that the Kupisch series is $(3,2)$. It has one injective module, by proposition \ref{PROP every injective eps mod is second syzygy of injective lambda} together with $\rank\Lambda=3$, there can be at most one injective $\Lambda$-module. Therefore the Kupisch series of $\Lambda$ is either of the form $(n,n,n-1)$ or $(n,n-1,n-1)$. Among them only $(4,4,3)$ satisfies all the conditions. The sequence of Kupisch series is $(3,2), (4,4,3),(5,5,5,4),(6,6,6,6,5)$ etc. By induction on the rank, claim follows.
\end{remark}

We want to emphasis that number of relations helps to lower the bound obtained in \cite{madsen2018bounds}. Notice that global dimension attains the bound if and only if $r=N-1$ and $r_i=r_{i-1}-1$ for all $1\leq i\leq m-1$. If we take $r=N$, $\Lambda$ becomes selfinjective algebra.
\begin{example} Let $\Lambda$ be cyclic Nakayama algebra of $N=8$ vertices with relations:
\begin{align*}
\alpha_{4}\alpha_{3}=0, \hspace{1cm}\alpha_{6}\alpha_{5}=0,\hspace{1cm} \alpha_{2}\alpha_{1}\alpha_{8}=0
\end{align*}
By simple computation, $m=1$, and global dimension is $2$. So it is smaller than $r=3$. 
\end{example}

\subsection{Future Directions}\label{future}
Proposition \ref{PROP filtered module category is equivalent to wide subcategory} suggests that the syzygy filtered algebra construction can be carried into other classes of algebras. In a series of forthcoming papers \cite{stz}, we give complete classifications Nakayama algebras which are Auslander-Gorenstein and finitistic Auslander and linear Nakayama algebras which are higher Auslander algebras.  We show that algebra $A$ is selfinjective if and only if $A$ is equivalent to its wide subcategory $\cW(A)$ cogenerated by projective-injective $A$-module, so the result \ref{PROP algebra is selfinjective iff equivalent to filtered algebra} holds in general. If global dimension of $A$ is $d$ and dominant dimension is at least one, we show that global dimension of the wide subcategory $\cW(A)$ is at most $d-2$. This suggests that we can use the wide subcategory approach to give upper bounds for global dimension in general. 

\bibliographystyle{alpha}

\end{document}